\documentclass[letterpaper,reqno,11pt,oneside]{amsart}

\usepackage{amsmath,amsthm,amsfonts,amssymb}
\usepackage{enumitem,comment}
\usepackage[usenames,dvipsnames]{color}
\usepackage{mathtools}
\usepackage{microtype}
\usepackage[extension=pdf]{hyperref}
\usepackage[totalheight=9.6in,totalwidth=5.95in,centering]{geometry}
\usepackage{graphics,graphicx}
\usepackage{mdwlist}

\usepackage[style=alphabetic,natbib=true,maxnames=99,isbn=false,doi=false,url=false,firstinits=true,hyperref=auto,arxiv=abs,backend=bibtex]{biblatex}
\DeclareFieldFormat[article,inbook,incollection,inproceedings,patent,thesis,unpublished]{title}{#1\isdot}
\renewbibmacro{in:}{%
  \ifentrytype{article}{}{%
    \printtext{\bibstring{in}\intitlepunct}}} 
\defbibheading{apa}[\refname]{\section*{#1}}
\DeclareFieldFormat{sentencecase}{\MakeSentenceCase{#1}} \renewbibmacro*{title}{%
  \ifthenelse{\iffieldundef{title}\AND\iffieldundef{subtitle}}{}
  {\ifthenelse{\ifentrytype{article}\OR\ifentrytype{inbook}%
      \OR\ifentrytype{incollection}\OR\ifentrytype{inproceedings}%
      \OR\ifentrytype{inreference}} {\printtext[title]{%
        \printfield[sentencecase]{title}%
        \setunit{\subtitlepunct}%
        \printfield[sentencecase]{subtitle}}}%
    {\printtext[title]{%
        \printfield[titlecase]{title}%
        \setunit{\subtitlepunct}%
        \printfield[titlecase]{subtitle}}}%
    \newunit}%
  \printfield{titleaddon}}

\definecolor{darkblue}{rgb}{0.13,0.13,0.39}
\hypersetup{colorlinks=true,urlcolor=darkblue,citecolor=darkblue,linkcolor=darkblue,%
  pdftitle=A Pfaffian representation for flat ASEP}

\mathtoolsset{showmanualtags,showonlyrefs}

\newtheorem{thm}{Theorem}[section] 
\newtheorem{lem}[thm]{Lemma}
 
\newtheorem{prop}[thm]{Proposition} \newtheorem{cor}[thm]{Corollary}
 
\theoremstyle{definition} \newtheorem{rem}[thm]{Remark} \newtheorem*{rem*}{Remark}
 
 \newcounter{assum}

\newcommand{\I}{{\rm i}} \newcommand{\pp}{\mathbb{P}} 
 \newcommand{\ee}{\mathbb{E}} \newcommand{\rr}{\mathbb{R}}
\newcommand{\nn}{\mathbb{N}} \newcommand{\zz}{\mathbb{Z}}

 \newcommand{\cc}{{\mathbb{C}}} \newcommand{\p}{\partial}
\newcommand{\uno}[1]{\mathbf{1}_{#1}}

\newcommand{\ep}{\varepsilon} \newcommand{\vs}{\vspace{6pt}}\newcommand{\wt}{\widetilde}

\newcommand{\bk}{{m}}

\newcommand{\K}{K_{\Ai}} \newcommand{\m}{\overline{m}}

 \newcommand{\qand}{\quad\text{and}\quad}
\newcommand{\qqand}{\qquad\text{and}\qquad}
\newcommand{\mfh}{\mathfrak{h}}

\newcommand{\mff}{\mathfrak{f}}
\newcommand{\mfe}{\mathfrak{e}}
\newcommand{\mfs}{\mathfrak{s}}
\newcommand{\mfg}{\mathfrak{g}}
\newcommand{\mfgp}{\mfg_{\rm p}}
\newcommand{\mfgu}{\mfg_{\rm u}}
\newcommand{\tmfgp}{\tilde{\mfg}_{\rm p}}
\newcommand{\tmfgu}{\tilde{\mfg}_{\rm u}}
\newcommand{\mfu}{\mathfrak{u}}
\newcommand{\mfuu}{\mathfrak{u}_{\rm u}}
\newcommand{\mfup}{\mathfrak{u}_{\rm p}}
\newcommand{\mfuap}{\mathfrak{u}_{\rm ap}}
\newcommand{\tmfuap}{\tilde{\mathfrak{u}}_{\rm ap}}

\newcommand{\mfv}{\mathfrak{v}}
\newcommand{\mfZ}{\mathfrak{Z}}

\newcommand{\mcP}{\mathcal{P}}

\newcommand{\Kmu}{K}
\newcommand{\KASEP}{K^{\rm flat}}
\newcommand{\KASEPeta}{K^{{\rm flat},\eta}}
\newcommand{\KASEPzero}{K^{{\rm flat},0}}
\newcommand{\wtKASEP}{\wt{K}^{\rm flat}}
\newcommand{\wtKASEPz}{\wt{K}^{{\rm flat},\zeta}}
\newcommand{\wtKASEPeta}{\wt{K}^{\eta}}
\newcommand{\wtKASEPzero}{\wt{K}^{0}}
\newcommand{\KCDRP}{\bar{K}}

\newcommand{\psiASEP}{\psi}

\newcommand{\cpsiASEP}{\check{\psi}}

\newcommand{\vsigma}{\sigma}

\DeclareMathOperator{\Ai}{Ai}
\DeclareMathOperator{\tr}{tr}

\renewcommand{\idotsint}{\int}

\def\XXint#1#2#3{{\setbox0=\hbox{$#1{#2#3}{\int}$}
\vcenter{\hbox{$#2#3$}}\kern-.5\wd0}}

\newcommand*{\pvint}{
  \mathop{\,\,\vphantom{\intop}\!\!\!
    \mathpalette\pvintop\relax}\nolimits}
\newcommand*{\pvintop}[2]{
  \ooalign{$#1\intop$\cr\hidewidth$#1-$\hidewidth}}

\newcommand{\inv}[1]{\frac{1}{#1}}
\newcommand{\abs}[1]{\ensuremath{\left|#1\right|}}

\DeclareMathOperator{\pf}{Pf}
\DeclareMathOperator*{\Res}{Res}
\DeclareMathOperator{\sgn}{sgn}

\newcommand{\pochinf}[1]{\ensuremath{\left(#1\right)_\infty}}
\newcommand{\taufac}{_\tau!}
\newcommand{\qfac}{_q!}
\newcommand{\twopii}[1]{(2\pi\I)^{#1}}
\newcommand{\itwopii}[1]{\frac{1}{(2\pi\I)^{#1}}}

\numberwithin{equation}{section}

\let\oldmarginpar\marginpar
\renewcommand\marginpar[1]{\-\oldmarginpar[\raggedleft\footnotesize #1]%
  {\raggedright{\small\textsf{#1}}}}

\allowdisplaybreaks[2]

\let\nnu\nu

\renewcommand{\nu}{n^{\rm u}}
\newcommand{\np}{n^{\rm p}}
\newcommand{\nf}{n^{\rm f}}
\newcommand{\vnu}{\vec{n}^{\rm u}}
\newcommand{\vnp}{\vec{n}^{\rm p}}
\newcommand{\vnf}{\vec{n}^{\rm f}}
\newcommand{\ku}{k_{\rm u}}
\newcommand{\kp}{k_{\rm p}}
\newcommand{\kf}{k_{\rm f}}
\renewcommand{\wp}{w^{\rm p}}
\newcommand{\wmp}{w^{\text{\rm --p}}}
\newcommand{\wu}{w^{\rm u}}
\newcommand{\wf}{w^{\rm f}}
\newcommand{\vwp}{\vec{w}^{\rm p}}
\newcommand{\vwmp}{\vec{w}^{\text{\rm --p}}}
\newcommand{\vwu}{\vec{w}^{\rm u}}
\newcommand{\vwf}{\vec{w}^{\rm f}}
\newcommand{\zp}{z^{\rm p}}
\newcommand{\zmp}{z^{\text{\rm --p}}}
\newcommand{\zu}{z^{\rm u}}
\newcommand{\vzp}{\vec{z}^{\rm p}}
\newcommand{\vzmp}{\vec{z}^{\text{\rm --p}}}
\newcommand{\vzu}{\vec{z}^{\rm u}}

\newcommand{\ts}{\widetilde{s}}

\newcommand{\ty}{\widetilde{y}}
\newcommand{\tl}{\widetilde{\lambda{}}}

\newcommand{\nuhalfflat}{\nnu^\text{\rm h-fl}_{k,m}}
\newcommand{\tnuhalfflat}{\wt{\nnu}^\text{\rm h-fl}_{k,m}}
\newcommand{\nuflat}{\nnu^{\rm flat}_{k,m}}
\newcommand{\bnuflat}{\bar{\nnu}^{\rm flat}_{k,m}}
\newcommand{\mfZhalfflat}{\mfZ^\text{\rm h-fl}_{\bk}}
\newcommand{\mfZflat}{\mfZ^{\rm flat}_{\bk}}
\newcommand{\bmfZflat}{\bar{\mfZ}^{\rm flat}_{\bk}}
\newcommand{\thfl}{{\text{\rm h-fl}}}

\begin{document}

\title{A Pfaffian representation for flat ASEP}

\author{Janosch Ortmann} \address[J.~Ortmann]{
  Department of Mathematics\\
  University of Toronto\\
  40 St. George Street\\
  Toronto, Ontario\\
  Canada M5S 2E4} \email{janosch.ortmann@utoronto.ca}

\author{Jeremy Quastel} \address[J.~Quastel]{
  Department of Mathematics\\
  University of Toronto\\
  40 St. George Street\\
  Toronto, Ontario\\
  Canada M5S 2E4} \email{quastel@math.toronto.edu}

\author{Daniel Remenik} \address[D.~Remenik]{
  Departamento de Ingenier\'ia Matem\'atica and Centro de Modelamiento Matem\'atico\\
  Universidad de Chile\\
  Av. Beauchef 851, Torre Norte\\
  Santiago\\
  Chile} \email{dremenik@dim.uchile.cl}

\maketitle

\begin{abstract}
  We obtain a Fredholm Pfaffian formula for an appropriate generating function of the
  height function of the asymmetric simple exclusion process starting from flat
  (periodic) initial data. Formal asymptotics lead to the GOE Tracy-Widom distribution.
\end{abstract}

\tableofcontents

\section{Introduction}
\label{sec:intro}

In this article, we consider the one-dimensional asymmetric simple exclusion process
(ASEP) with \emph{flat} initial data, meaning that initially even sites are occupied and
odd sites are empty.  The particles then perform nearest neighbour asymmetric random
walks, in continuous time, the only interaction being that jumps to already occupied sites
are suppressed.  ASEP is perhaps the most popular of the small class of special partially
solvable discretizations of the Kardar-Parisi-Zhang (KPZ) equation.  In recent years, a
series of breakthroughs on these models has led to exact formulas for one point
distributions of the solutions of KPZ for very special initial data.  The first formulas
were for the narrow wedge and half-Brownian initial conditions \cite{ACQ,
  corwinQuastelRarefaction}, which were obtained by taking limits of earlier formulas for
ASEP with step and step-Bernoulli initial conditions due to Tracy and Widom
\cite{tracyWidomASEP3,tracyWidomASEP4}. The remaining key initial data, flat, half-flat and
Brownian, are more challenging and had to wait several years (\cite{borCorFerrVeto,oqr-half-flat}, and
the present work).  Meanwhile there was a parallel effort in the physics community working
directly from KPZ, using the method of replicas to derive rigorous formulas for the
moments, and then writing divergent series for appropriate generating functions for the
one point distributions of KPZ, which could then be formally manipulated into convergent
Fredholm series \cite{cal-led-rosso, dotsenkoGUE, imamSasamReplicaHalfBr, imamSasamStatCorrKPZ}.  

In this context, P. Le Doussal and P. Calabrese \cite{cal-led} derived what in mathematics
would be called conjectural expressions for the flat case, as Fredholm Pfaffians.  To a
certain extent, this work is an attempt to make sense of their remarkable,
but highly non-rigorous computations.  Our main results provide analogous formulas for
flat ASEP.  Formal asymptotics give the correct fluctuations in the large time limit,
which in this case correspond to the GOE Tracy-Widom distribution \cite{tracyWidom2}, in a
new form. In the weakly asymmetric limit one obtains versions of the formulas in \cite{cal-led}.
However the technicalities of this, as well as those of providing a rigorous justification of the
large time asymptotics, are so involved that we have left them for a future article.

It should be noted that an earlier formula existed for the one point distribution of ASEP
with flat initial data, due to E. Lee \cite{lee}.  However, this formula is similar to the
first formulas of Tracy and Widom for the step case
\citet{tracyWidomASEP1} and are not in a suitable form for
asymptotics, even at a formal level.  In the step case, it already required an
extraordinary argument \cite{tracyWidomASEP2} to rewrite this as what is now referred to
as a \emph{small contour formula}, which is indeed suitable for asymptotics
\cite{tracyWidomASEP3}.  In the flat case this method has not been successful so far, and
thus we follow an alternative approach.

In a prequel article \cite{oqr-half-flat}, we derived a formula for a generating function
in the half-flat case, meaning that initially only even positive sites are occupied.
Because such data is \emph{left-finite}, meaning that there is a leftmost particle, the
duality methods of \cite{BCS} apply, and with suitable guessing, inspired, in part, by the
formulas of \cite{lee}, one is able to obtain an exact solution to the duality equations,
which provide formulas for the exponential moments of the ASEP height function and can
then be turned into a certain generating function. 

In this article we take a limit of the half-flat moment formula probing into the positive
region, which reproduces the flat initial data. We follow the broad lines of
\cite{cal-led}. The asymptotics is extremely interesting, nonstandard, and unusually
involved. One is forced to make a large deformation of contours, and the flat formula is
an enormous sum of the residues passed as one makes this deformation. In this sum, a
pairing structure arises which leads to the emergence of Pfaffians.

Naturally, the remarkable formulas beg the question whether there is a more direct approach,
and some recent physics work have indicated possibilities in this direction
\cite{deNardisEtAl-quench,cal-led-quench}.

\vs

\noindent{\bf Outline.} The rest of the article is organized as follows. Our main results
are discussed in Section \ref{sec:main}. Section \ref{sec:flat-lim} contains the
derivation of the flat ASEP moment formulas as a limit of our earlier half-flat
formulas. Section \ref{sec:pfaffian} is devoted to unravelling the Pfaffian structure
behind the moment formulas and rewriting them in a simpler way. The resulting formulas are
then summed in Section \ref{sec:generatingfunction} to form a generating function and
showing how, in a certain case, they lead to a Fredholm Pfaffian. Section
\ref{sec:app-bose} contains a discussion of analogous results for the moments of the
solution of the KPZ/SHE equation (or, more precisely, the delta Bose gas) with flat
initial data. Appendix \ref{sec:goe-limit} presents a formal critical point analysis
supporting the conjecture that the ASEP height function fluctuations are described in the
long time limit by the GOE Tracy-Widom distribution when the initial data is flat. As
explained in Remark \ref{remark2.2}.vii, the rigorous asymptotics cannot be obtained
directly by these methods and remains the subject for future work. Finally, Appendix
\ref{sec:linalg} contains an overview of Pfaffians and Fredholm Pfaffians.

\section{Main results}\label{sec:main}

\subsection{Preliminaries}
\label{sec:prelim}

The one-dimensional asymmetric simple exclusion process (ASEP) with jump rates
${p\in [0,1/2)}$ to the right and $q=1-p\geq p$ to the left is the continuous time Markov
process with state space $\{0,1\}^{\zz}$ and generator
\begin{align}
  \mathcal{L} f(\eta) & =\sum_{x\in \zz} p \eta(x) (1-\eta(x+1)) (f( \eta-\uno{x}+
  \uno{x+1}) - f(\eta) )\\ & \qquad+ q \eta(x+1) (1-\eta(x)) (f( \eta+\uno{x}- \uno{x+1})
  - f(\eta) ),
\end{align}
where $\uno{x}(y)=1$ if $y=x$ and $0$ otherwise. Writing $\hat{\eta}(x)=2\eta(x)-1$, we define the ASEP \emph{height function} by
\begin{equation}
  \label{defofheight}
  h(t,x) = \begin{cases}
    2N^{\rm flux}_0(t) + \sum_{0<y\leq x}\hat{\eta}(t,y), & x>0,\\
    2N^{\rm flux}_0(t), & x=0,\\
    2N^{\rm flux}_0(t)-  \sum_{x<y\leq 0}\hat{\eta}(t,y), & x<0,
  \end{cases}
\end{equation}
where $N^{\rm flux}_0(t)$ is the net number of particles which crossed from site $1$ to $
0$ up to time $t$, meaning that particle jumps $1\to 0$ are counted as $+1$ and jumps
$0\to 1$ are counted as $-1$. For $x\notin\zz$ we define $h(t,x)$ by linear interpolation.
With these conventions $h(t,x)$ is made of line segments with slopes $\pm1$, and it
evolves with its own Markovian dynamics, which flips $\vee\mapsto \wedge$ at rate $q$ and
$\wedge\mapsto\vee$ at rate $p$. For $q>p$, the general trend is an upwards moving height
function.

An important parameter in the formulas will be
\begin{equation}\label{tau}
  \tau = p/q \in [0,1].
\end{equation} 
We are primarily concerned with the \emph{flat},
\begin{equation}
  \label{eq:flat-initial-condition}
  \eta^{\rm flat}_0(x)=\uno{x\in2\zz},
\end{equation}
and \emph{half-flat},
\begin{equation}
  \label{eq:halfflat-initial-condition}
  \eta^\thfl_0(x)=\uno{x\in2\zz_{>0}},
\end{equation}
initial conditions (with the  notation $\zz_{>n}=\{ n+1,n+2,\ldots\}$ which will be used
throughout). The superscripts flat and h-fl will  be used for probabilities
and expectations computed with respect to these initial conditions.

In order to state our results we need to recall some definitions. Given a measure space
$(X,\Sigma,\mu)$, the \emph{Fredholm Pfaffian} (introduced in \cite{rainsCorr}) of an
anti-symmetric $2\times2$-matrix kernel $K(\lambda_1,\lambda_2)=\left[
  \begin{smallmatrix}
    K_{1,1}(\lambda_1,\lambda_2) & K_{1,2}(\lambda_1,\lambda_2)\\
    -K_{1,2}(\lambda_2,\lambda_1) & K_{2,2}(\lambda_1,\lambda_2)
  \end{smallmatrix}\right]$ acting on $L^2(X)$ is defined as the (formal) series
\begin{equation}\label{eq:fredPf}
  \pf\!\big[J-K\big]_{L^2(X)}=
  \sum_{k\geq0}\frac{(-1)^k}{k!}\int_{X^k}\mu(d\lambda_1)\dotsm\mu(d\lambda_k)\,\pf\!\big[K(\lambda_a,\lambda_b)\big]_{a,b=1}^k,
\end{equation}
where, we recall, the Pfaffian of a $2k\times 2k$ matrix $A=\{a_{i,j}\}$ is given by
\begin{equation}
  \pf\!\big[A\big]=\frac{1}{2^k k!}\sum_{\sigma\in S_{2n}}\sgn(\sigma)\prod_{i=1}^k
  a_{\sigma(2i-1),\sigma(2i)}.
\end{equation}
Appendix \ref{sec:linalg} surveys basic facts about Fredholm Pfaffians.
   
We will also need various $q$-analogues of standard calculus functions.  The $q$ here is the $q$ in \emph{quantum
  calculus}  \cite{gasperRahman}.  Our parameter \eqref{tau} came with the unfortuitous name $\tau$ instead of
$q$, since historically the left jump rate was called $q$.  In an attempt to keep with
standard practice, we will use $q$ when discussing the definitions of the $q$-deformed
functions, but set $q\mapsto\tau$ within the computations. The standard \emph{$q$-numbers} are given
by
\[[k]_q=\frac{1-q^k}{1-q}.\] Based on them one defines the \emph{$q$-factorial}
\begin{equation}
  \label{eq:q-fact}
  n\qfac=[1]_q[2]_q\dotsc[n]_q=\frac{(q;q)_n}{(1-q)^n},
\end{equation}
where the \emph{$q$-Pochhammer symbols} are given by
\begin{equation}\label{Eq:DefQPoch}
  (a;q)_n = \prod_{k=0}^{n-1}\left(1-q^k a\right).
\end{equation}
This last definition can be extended directly to $n=\infty$ when $|q|<1$, yielding the
infinite $q$-Pochhammer symbols. The standard $q$-exponentials are then defined as
\begin{equation}
e_q(z)=\frac{1}{\pochinf{(1-q)z;q}}
\qqand  E_q(z)=\pochinf{-(1-q)z;q}\label{eq:standardqexpspoch}
\end{equation}
or (thanks to the $q$-Binomial Theorem) through their series representations
\begin{equation}\label{eq:standardqexps}
e_q(z)=\sum_{k\geq0}\frac{z^k}{k\qfac}
\qqand E_q(z)=\sum_{k\geq0}\frac{q^{\inv2k(k-1)}z^k}{k\qfac}.
\end{equation}
The first series converges absolutely only for $|z|<1$, but it extends analytically to
$\Re(z)<0$ through the formula in \eqref{eq:standardqexpspoch}, where it approximates
$\exp(z)$ well as $q\nearrow 1$. And, of course, $E_q(z)=e_q(-z)$.  In
\cite{oqr-half-flat} we obtained a formula for
$\ee^\thfl\big[e_\tau(\zeta\tau^{h(t,x)/2})\big]$. 

In the flat case one is naturally
led to use a different kind of $q$-exponential.
The \emph{symmetric $q$-numbers} \cite{nelsonGartley,frenkelTuraev} are given by
\begin{equation}
  \widetilde{[k]}_q = \frac{q^{k/2} - q^{-k/2}}{q^{1/2} - q^{-1/2} }
\end{equation}
and the \emph{symmetric $q$-factorial} by
\begin{equation}
  \widetilde{\,k_q!\,} =
  \widetilde{[1]}_q\widetilde{[2]}_q\dotsc\widetilde{[k]}_q=q^{-\inv4k(k-1)}k\qfac
\end{equation}
with $\widetilde{\,0_q!\,}=1$.  The \emph{symmetric $q$-exponential}
function is then defined as
\begin{equation}\label{eq:symmqexp}
  \exp_q(z) = \sum_{k=0}^\infty\frac{z^k}{\widetilde{\,k_q!\,}} = \sum_{k=0}^\infty \frac{ q^{\frac14k(k-1)} }{k\qfac}z^k.
\end{equation}
The function $f(z) =\exp_q(z)$ is the unique solution of
\begin{equation}
  \frac{\delta_q f(z)}{\delta_q z}  = f(z), \qquad f(0)=1
\end{equation}
where the symmetric $q$-difference operator is defined by $\delta_q f(z) = f(q^{1/2} z) -
f(q^{-1/2} z)$ (see for instance \cite{gasperRahman}).  The standard $q$-exponentials
do satisfy analogous $q$-difference equations, but, unlike them, the symmetric $q$-exponential satisfies
\begin{equation}\label{eq:symm}
  \exp_q(z)= \exp_{q^{-1}}(z),
\end{equation}
which is a natural symmetry in the flat problem (see Remark \ref{remark2.2}.v).

\subsection{Fredholm Pfaffian formula}
\label{sec:fred-pf-intro}

Now we can state one of our main results. We need to introduce some functions (in all that
follows we fix a parameter $\beta>0$): for $s\in\cc$, $\lambda\in\rr_{\geq0}$,
$y\in\cc$ with $|y|=1$ and $\zeta\in\cc\setminus\rr_{>0}$, define
\begin{equation}
  \psiASEP(s,\lambda,y;\zeta) = 
  (-\zeta)^s\frac{\tau^{-\beta s^2 + \frac34s-1}}y
  \frac{1-\tau^{s/2}y}{1+\tau^{s/2}y}\,\frac{\pochinf{-\tau^{-s/2}y;\tau}}{\pochinf{\!-\tau^{s/2}y;\tau}}
  e^{-\lambda \frac{1-\tau^{s/2}y}{1+\tau^{s/2}y}+t\Big[\frac{1}{1+\tau^{-s/2} y}-\frac1{1+\tau^{s/2} y}\Big]}
  \label{eq:psiASEP-intro}
\end{equation}
and let, for $\omega\in\rr$,
\begin{equation}\label{eq:cpsiASEP-intro}
\cpsiASEP(\omega,\lambda,y;\zeta)=\inv{2\pi\I}\int_{\I\rr}ds\,e^{s\omega}\psiASEP(s,\lambda,y;\zeta),
\end{equation}
which should be regarded as the inverse double-sided Laplace transform of $\psiASEP$ (in
$s$), see \eqref{eq:DLapInv}; additionally, if $z,z_1,z_2\in\rr$ and
$\sigma,\sigma_1,\sigma_2\in\{-1,1\}$, let
\begin{equation}
\begin{aligned}
  F_1(z,y) & =  \sum_{m=1}^\infty (-z)^{m}(1-\tau)^{2m}\tau^{(\frac12+2\beta)m^2-m}\frac{1+y^2}{y^2-1}
    \frac{\pochinf{\tau^{1+m}y^2;\tau}\pochinf{\tau^{1+m}/y^2;\tau}}
    {\pochinf{\tau y^2;\tau}\pochinf{\tau/y^2;\tau}}\\
  F_2(z_1,z_2;\sigma_1,\sigma_2) & = \sum_{m_1,m_2=1}^\infty (-z_1)^{m_1} (-z_2)^{m_2}
    (-\sigma_1\sigma_2)^{m_1\wedge m_2+1} \sgn(\sigma_2m_2-\sigma_1 m_1)\\
    &\hspace{2.8in}\times\frac{\tau^{(\inv4+\beta)(m_1^2+m_2^2)-\inv4(m_1+m_2)}}{(m_1)\taufac
      \,(m_2)\taufac},\\
  F_3(z) &= - \sum_{m=1}^\infty\frac{\tau^{(\frac14+\beta)m^2-\inv4m}}{m\taufac}(-z)^m.  
\end{aligned}\label{eq:DefF}
\end{equation}

\begin{thm}\label{thm:fredPf-intro}
  For $\zeta\in\cc$ with $|\zeta|<\tau^{1/4}$, we have
  \begin{equation}
    \label{eq:mainresult}
    \ee^{\rm flat}\!\left[\exp_\tau\!\big(\zeta\tau^{\inv2h(t,0)}\big)\right]
    =\pf\!\big[J-\wtKASEPz\big]_{L^2(\rr_{\geq0})}
  \end{equation}
  where the $2\times 2$ matrix kernel $\wtKASEP$ is given in terms of the functions
  defined in \eqref{eq:psiASEP-intro} and \eqref{eq:DefF}, with any $\beta>0$, through (here
  $C_{0,1}$ is the unit circle centered at the origin)
  \begin{align}
    \wtKASEPz_{1,1}(\lambda_1,\lambda_2)& = \int_{\rr^2} d\omega_1\,d\omega_2\,\inv{\pi\I}\pvint_{\!\!\!\!C_{0,1}}dy\,
    \cpsiASEP(\omega_1,\lambda_1,y)\cpsiASEP(\omega_2,\lambda_2,\tfrac1y)F_1(e^{-(\omega_1+\omega_2)},y)\\
    &\quad+\frac12\sum_{\sigma_1,\sigma_2\in\{-1,1\}}\int_{\rr^2}d\omega_1\,d\omega_2\,\cpsiASEP(\omega_1,\lambda_1,\sigma_1)
    \cpsiASEP(\omega_2,\lambda_2,\sigma_2)F_2(e^{-\omega_1},e^{-\omega_2};\sigma_1,\sigma_2)\\
    \wtKASEPz_{1,2}(\lambda_1,\lambda_2)&=\frac12\sum_{\sigma\in\{-1,1\}}\,\int_{-\infty}^\infty
    d\omega\, F_3(e^{-\omega})\cpsiASEP(\omega;\lambda_1,\sigma)\label{eq:kernelBessel}\\
    \wtKASEPz_{2,2}(\lambda_1,\lambda_2)&=\frac12\sgn(\lambda_2-\lambda_1).
  \end{align}
\end{thm}

\begin{rem}\mbox{}\label{remark2.2}
  \begin{enumerate}[label=(\roman*),leftmargin=*]
  \item The symbol $\pvint$ in the $y$ integral appearing in the definition of
    $\wtKASEP_{1,1}$ denotes that this is a principal value integral. This is because
    the $y$ contour goes through the singularities of $F_1(z,y)$ at $y=\pm1$ (see the
    discussion starting after \eqref{eq:pvPfaffPre}).
  \item The parameter $\beta>0$ which appears in \eqref{eq:psiASEP-intro} and
    \eqref{eq:DefF} is necessary to make \eqref{eq:cpsiASEP-intro} a convergent integral,
    but otherwise has no meaning.
  \item The three functions in \eqref{eq:DefF} should be regarded as certain
    $q$-deformations of classical special functions. For example, $F_3(z)-1$ is a
    $q$-deformation of the exponential; in fact,
     $F_3(z)=1-\exp_\tau(-z)$ if $\beta=0$.
    Similarly, $F_1$ is related to
    the Bessel function $J_0$:
    \[F_1(z,y)=\frac{1+y^2}{y^2-1}\big(J_0(2\sqrt{z};\tau,y^2)-1\big),\] where
    $J_0(z;\tau,y)$ is a 2-parameter ($\tau\in(0,1)$, $y\in\cc$ with $|y|=1$) deformation
    of the Bessel function given by $J_0(z;\tau,y) =
    \sum_{m=0}^\infty(1-\tau)^{2m}\tau^{(\inv2+2\beta)m^2-m} \frac{(-z^2/4)^m}{ (\tau
      y^2;\tau)_m(\tau/y^2;\tau)_m }$.  As $y\to\pm1$, $J_0(z;\tau,y)$ becomes $\wt
    J^{(1)}_0(z;\tau)= \sum_{m=0}^\infty\tau^{(\inv2+2\beta)m^2-m} \frac{(-z^2/4)^m}{
      (m\taufac)_m^2}$, which in the case $\beta=0$ can be regarded as a symmetric version
    of the Jackson $q$-Bessel function $J^{(1)}_0(z;\tau)=\sum_{m=0}^\infty
    \frac{(-z^2/4)^m}{(1-\tau)^{2m}(m\taufac)_m^2}$ (which, as $\tau\to 1$ as well,
    becomes the standard Bessel function $J_0(z)$). $F_2$ is a bit more difficult to
    recognize, so we will not attempt it here, but in view of (146) in \cite{cal-led} it
    should be possible to write it in terms of certain $q$-deformations of the Bessel and
    hyperbolic sine functions. 
      \item The statistical symmetry of flat ASEP under the transformation $(p,q)\mapsto
    (q,p)$ together with $h\mapsto -h$ is respected by the formula because the symmetric
    $q$-exponential satisfies \eqref{eq:symm}.
  \item There are similar, although less appealing, formulas for generating functions of
    $\tau^{\inv2h(0,x)}$ defined in terms of other $q$-deformations of the exponential
    function, see Theorem \ref{thm:taulaplflat-pre}. These include formulas for the case
    of $E_\tau$, but not $e_\tau$. This last case does not seem to be accessible through
    our methods, because the moments of $\tau^{\inv2h(0,x)}$ are expected to grow like
    $e^{ck^2}$ for some $c>0$, and hence $\ee^{\rm
      flat}\!\left[e_\tau\big(\zeta\tau^{\inv2h(t,0)}\big)\right]$ cannot be computed by
    summing moments.  It is worth emphasizing here that the use of the function $\exp_\tau(z)$
    is \emph{intrinsic} to the problem, and \emph{not} just a convenient choice. The formulas for
    other generating functions are \emph{not} Fredholm Pfaffians.
  \item The $\exp_\tau$ generating function on the left hand side of \eqref{eq:mainresult} does 
  \emph{not} in general determine the distribution of $h(t,0)$. On the other hand, the formula in
    \cite{lee} is for the distribution function of $h(t,0)$, but it is not obvious how to
    obtain \eqref{eq:mainresult} from Lee's formula.
  \item Focusing on the left hand side of \eqref{eq:mainresult}, it is tempting to believe
    that $\exp_\tau(z)$ behaves sufficiently like $\exp(z)$ so that the key identity
    \begin{equation}\label{eq:exptrick}
      \lim_{t\to\infty} \ee^{\rm flat}\!\left[\exp\{-e^{ \alpha(h(t,0)-\frac12t-t^{1/3} r) } \}\right] 
      =\pp^{\rm flat}\!\left(\lim_{t\to\infty} \frac{h(t,0)-\frac12t}{t^{1/3}} > r \right)
    \end{equation}
    for $\alpha<0$ still holds with $\exp$ replaced by $\exp_\tau$.  Indeed, this is the
    case with $e_\tau$ (this fact has been used succesfully in recent years to derive the
    asymptotics of some related models, although not for the half-flat or flat initial
    conditions, see for instance \cite{borCor,borodinCorwinFerrari,ferrVeto-q-TASEP}).
        
    For $x>-(1-\tau)^{-1}$ it is the case that $\exp_\tau(z)$ looks quite like $\exp(z)$,
    and in fact it converges to it uniformly on $[-a,\infty)$ for any $a>0$, as
    $\tau\nearrow 1$.  However, $\exp_\tau(z)$ has a largest real zero at $x_0\sim -
    (1-\tau)^{-1}$, and as $x$ decreases below $x_0$, it begins to oscillate wildly, with
    zeros at $x_k\in [x_0q^{-(k+1)/2},x_0q^{-k/2}]$ for each $k\in\zz_{\geq1}$ -- their
    precise disposition is unknown \cite{nelsonGartley} -- and reaching size approximately
    $e^{|\!\log|x||^2/ \log(1-\tau)|}$ in between.  Unfortunately, this genuinely
    precludes estimations cutting off this bad region.
  \item A formal steepest descent analysis shows that, setting
    $\zeta=-\tau^{-t/4+t^{1/3}r/2}$, the right hand side of \eqref{eq:mainresult} leads in
    the long time limit, as expected, to the GOE Tracy-Widom distribution. The limit
    is obtained in the form
    \begin{equation}
      \label{eq:GOE-pfaffian-intro}
      F_{\rm GOE}(r)=  \pf\!\left[J-K_r\right]_{L^2([0,\infty))}
    \end{equation} with 
    \begin{equation}
      K_r(\lambda_1,\lambda_2)=\left[
        \begin{matrix}
          \tfrac12(\p_{\lambda_1}-\p_{\lambda_2})\K(\lambda_1+r,\lambda_2+r) & -\tfrac12\Ai(\lambda_1+r)\\
          \tfrac12\Ai(\lambda_2+r) & \tfrac12\sgn(\lambda_2-\lambda_2)
        \end{matrix}\right],\label{eq:Kr}
    \end{equation}
    a formula for the Tracy-Widom GOE distribution which is essentially equivalent to one
    implicit in \cite{cal-led}, and which is also very similar to (but not quite the same
    as) a formula appearing in \cite{ferrariPolyGOE}. Here the \emph{Airy kernel} $\K$ is
    defined as
    \begin{equation}
      \K(\lambda_1,\lambda_2)=\int_0^\infty d\xi\Ai(\lambda_1+\xi)\Ai(\lambda_2+\xi).\label{eq:airyKernel}
    \end{equation}
    A proof of \eqref{eq:GOE-pfaffian-intro} together with the formal $t\to\infty$
    asymptotic analysis leading to this Fredholm Pfaffian is given in Appendix
    \ref{sec:goe-limit}.  
  \item Although we have not included it, a similar formal asymptotic analysis in the
    weakly asymmetric limit leads to the Le Doussal-Calabrese Pfaffian formula for flat
    KPZ \cite{cal-led} (see also Section \ref{sec:app-bose} for related formulas).
  \end{enumerate}
\end{rem}

\subsection{Moment formula and emergence of the Pfaffian structure}
\label{sec:moment}

The sum in \eqref{eq:symmqexp} includes a brutal cutoff by $q^{n^2/4}$, so that the left
hand side of \eqref{eq:mainresult} is practically a finite sum of moments.  Thus, in a
sense, the Pfaffian formula is mostly indicating a nice algebraic structure for a sum of
moments.  Therefore we state separately our moment formula.  It involves a kernel $\KASEP$
which is related to, but not quite the same as the kernel $\wtKASEPz$ appearing in
\eqref{eq:mainresult}. Since the meaning of the terms in $\KASEP$ only becomes apparent in
the computations, we do not repeat the detailed formula here.
 
\begin{thm}\label{prop:secondPfaffian}
  Let $\KASEP$ be the the $2\times2$ matrix kernel given by \eqref{eq:hatKmu} and
  \eqref{eq:hatKmatrix}. Then for any $\bk\in\zz_{\geq0}$ we have
  \begin{equation}\label{eq:flatASEPmoments-intro}
    \ee^{\rm flat}\!\left[\tau^{\inv2\bk h(t,0)}\right]
    =\bk\taufac\tau^{-\inv4m^2}\sum_{k=0}^{\bk}\frac{(-1)^{k}}{k!}\quad\,
      \smashoperator{\sum_{\substack{m_1,\dotsc,m_{k}=1,\\m_1+\dotsm+m_k=\bk}}^\infty}\qquad
    \int_{(\rr_{\geq0})^k}d\vec\lambda\,\pf\!\big[\KASEP(\lambda_a,\lambda_b;m_a,m_b)\big]_{a,b=1}^k.
  \end{equation}
\end{thm}

The route to these moment formulas for the flat initial data is by
taking appropriate limits of formulas for the half-flat initial case
\eqref{eq:halfflat-initial-condition}. The limit we are interested in consists in starting
ASEP with the shifted half-flat initial condition
$\eta_0(y)=\uno{y\in2\zz_{>-x}}$, considering the variable $h(t,0)$ and computing
the limit when $x\to\infty$. More precisely, we will obtain moment formulas for flat ASEP
through the identity
\begin{equation}
  \label{eq:momPreShift}
  \ee^{\rm flat}\!\left[\tau^{\inv2\bk h(t,0)}\right]=\lim_{x\to\infty}\ee^{2\zz_{>-x}}\!\left[\tau^{\inv2\bk h(t,0)}\right]
\end{equation}
for $m\geq0$, where the superscript on the right hand side simply refers to the initial condition
specified above. Introduce now the random
variables
\begin{equation}
N_x(t)=\sum_{y=-\infty}^x\eta_t(y),\label{eq:Nx}
\end{equation}
which are of course finite if and only if the initial condition is left-finite. It is not hard to
check that when all particles start to the right of the origin one has $N^{\rm
  flux}_0(t)=N_0(t)$, and hence from \eqref{defofheight} we get
\begin{equation}
  \label{eq:heightNx}
  h(t,x)=2N_x(t)-x
\end{equation}
in the half-flat case. On the other hand a simple coupling argument shows that
$h(t,0)$, with initial condition $\eta_0(y)=\uno{y\in2\zz_{>-x}}$ has the same
distribution as the shifted observable $h(t,2x)$ with initial condition
$\eta_0(y)=\uno{y\in2\zz_{>0}}$. Hence, using \eqref{eq:momPreShift} and \eqref{eq:heightNx} we deduce that
\begin{equation}\label{eq:xlimit}
  \ee^{\rm flat}\!\left[\tau^{\inv2mh(t,0)}\right]=\lim_{x\to\infty}\ee^\thfl\!\left[\tau^{m(N_{2x}(t)-x)}\right].
\end{equation}
  
In Theorem 1.3 of \cite{oqr-half-flat} we obtained the following formula for the moments
appearing on the right hand side of \eqref{eq:xlimit}:

\begin{thm}\label{thm:main-hf}
  Let $m\geq0$. Then
  \begin{equation}\label{eq:stpt1}
    \ee^\thfl\!\left[\tau^{mN_{x}(t)}\right]
    =m\taufac\sum_{k=0}^{m}\nuhalfflat(t,x)
  \end{equation}
  with
  \begin{multline}\label{eq:nuk-intro}
    \nuhalfflat(t,x)=\inv{k!}\sum_{\substack{n_1,\dotsc,n_k\geq1\\n_1+\dotsm+n_k=m}}\itwopii k
    \idotsint_{C_{0,\tau^{-\eta}}^k}d\vec w\,\det\!\left[\frac{-1}{w_a\tau^{n_a}-w_b}\right]_{a,b=1}^{k}\\
    \times\prod_{a=1}^k\mff(w_a;n_a)\mfgp(w_a;n_a)\prod_{1\leq a<b\leq k}\mfh_1(w_a,w_b;n_a,n_b),
  \end{multline}
  where $C_{0,\tau^{-\eta}}$ is a circle of radius $\tau^{-\eta}$ centered at
  the origin, with $\eta\in(0,1/4)$, $C_{0,\tau^{-\eta}}^k$ denotes the product of $k$
  copies of $C_{0,\tau^{-\eta}}$, and where
  \begin{gather}
    \mff(w;n)=(1-\tau)^{n}e^{(q-p)t\left[\frac{1}{1+w}-\frac1{1+\tau^{n}w}\right]}
    \Big(\tfrac{1+\tau^{n}w}{1+w}\Big)^{x-1},
    \end{gather} and $\mfgp$ and $ \mfh_1$ are given in \eqref{eq:germans-hf}.
\end{thm}
  
For simplicity, throughout the rest of the paper we will omit the bound on the indices
in products such as $1\leq a\leq k$ and $1\leq a<b\leq k$ when no confusion can arise
and the factors involved in the products are defined in terms of a collection of $k$
variables. A similar convention will sometimes be used for sums. Additionally, we will
continue using the notation $C^k$ for the product of $k$ copies of a given contour $C$
in the complex plane.

One of the most interesting parts of the story is the computation of the limit
\eqref{eq:xlimit}, which involves taking the limit $x\to\infty$ of the right hand side of
\eqref{eq:nuk-intro} with $x$ replaced by $2x$ and after multiplying by $\tau^{-mx}$. We
will start by rewriting this formula in a slightly different way which will be better
adapted for the calculation of this limit. Introduce the following functions:
  \begin{equation}
  \begin{gathered}
  \mff_1(w;n)=\tfrac{(1-\tau)^{n}}{w(1-\tau^{n})}
    e^{(q-p)t\left[\frac{1}{1+w}-\frac1{1+\tau^{n}w}\right]},
    \qquad\mff_2(w;n)=\left(\tfrac{1+\tau^{n}w}{1+w}\tau^{-\frac12n}\right)^{2x-1},\\
    \mfgp(w;n)=\frac{\pochinf{\!-w;\tau}}{\pochinf{\!-\tau^{n}w;\tau}}
    \frac{\pochinf{\tau^{2n}w^2;\tau}}{\pochinf{\tau^{n}w^2;\tau}},\\
    \mfgu(w_a;n_a)=\frac{\pochinf{\!-w_a;\tau}}{\pochinf{\!-\tau^{n_a}w_a;\tau}}
  \frac{\pochinf{\tau^{2n_a}w_a^2;\tau}}{\pochinf{\tau^{1+n_a}w_a^2;\tau}},\\
    \mfh_1(w_1,w_2;n_1,n_2)=\frac{\pochinf{w_1w_2;\tau}\pochinf{\tau^{n_1+n_2}w_1w_2;\tau}}
    {\pochinf{\tau^{n_1}w_1w_2;\tau}\pochinf{\tau^{n_2}w_1w_2;\tau}},\\
    \mfh_2(w_a,w_b;n_a,n_b)=\frac{(w_a\tau^{n_a}-w_b\tau^{n_b})(w_b-w_a)}{(w_a\tau^{n_a}-w_b)(w_b\tau^{n_b}-w_a)}.
  \end{gathered}\label{eq:germans-hf}
\end{equation}
 $\mfgu$ will not appear in the coming
formula \eqref{eq:tnuk-intro}, but will appear later in Theorem
\ref{thm:flat-moments}. The subscripts in $\mfgp$ and $\mfgu$ stand for ``paired'' and
``unpaired'', a terminology which we will explain shortly and which will become clear in
Section \ref{sec:flat-lim}.  Using these definitions in \eqref{eq:nuk-intro} (with $x$
replaced by $2x$, and multiplied by $\tau^{-mx}$) and expanding the determinant in that
formula using the \emph{Cauchy determinant formula}
\begin{equation}\label{eq:cauchydet}
  \det\!\left[\frac1{x_a-y_b}\right]_{a,b=1}^k=\frac{\prod_{a<b}(x_a-x_b)(y_b-y_a)}{\prod_{a,b}(x_a-y_b)},
\end{equation}
we obtain
\begin{equation}\label{eq:stpt}
  \ee^\thfl\!\left[\tau^{m(N_{2x}(t)-x)}\right]=m\taufac\sum_{k=0}^{m}\tnuhalfflat(t,2x)
\end{equation}
with
\begin{multline}\label{eq:tnuk-intro}
  \tnuhalfflat(t,2x)=\inv{k!}\sum_{\substack{n_1,\dotsc,n_k\geq1\\n_1+\dotsm+n_k=\bk}}\itwopii k\idotsint_{C_{0,\tau^{-\eta}}^k}d\vec w\,\prod_a\mff_1(w_a;n_a)\mff_2(w_a;n_a)\mfgp(w_a;n_a)\\
  \times\prod_{a<b}\mfh_1(w_a,w_b;n_a,n_b)\mfh_2(w_a,w_b;n_a,n_b).
\end{multline}
We have added a tilde in $\tnuhalfflat(t,2x)$ to indicate the fact that we have multiplied
by $\tau^{-mx}$ in \eqref{eq:stpt} (in other words,
$\tnuhalfflat(t,2x)=\tau^{-mx}\nuhalfflat(t,2x)$ in view of \eqref{eq:stpt1}). Note that
in \eqref{eq:tnuk-intro} we have written $\tau^{-mx}$ as $\prod_a\tau^{-n_ax}$.

Note that the only factors in the integrand on the right hand side of
\eqref{eq:tnuk-intro} that depend on $x$ are those of the form $\mff_2(w_a;n_a)$. An easy
computation shows that the base in this power has modulus strictly less than 1 if and only
if $|w_a|>\tau^{-n_a/2}$. This suggests that for each $a=1,\dotsc,k$ we should deform the
corresponding contour $C_{0,\tau^{-\eta}}$ to some contour lying just outside the ball of
radius $\tau^{-n_a/2}$.

As we perform this deformation we will cross many poles. The residue calculus associated
to this deformation is quite complicated, and is explained in detail in Sections
\ref{sec:poles} and \ref{sec:contDefLimit}. The result has two properties which turn out
to be crucial for the sequel. First, there are two types of poles that are crossed (see
the list in page \pageref{it:S1}): the ones coming from $\mfgp(w_a,n_a)$, which occur at
$w_a=\pm\tau^{-n_a/2}$, and the ones coming from $\mfh_1(w_a,w_b;n_a,n_b)$, which occur at
$w_b=\tau^{-n_a}/w_a$ whenever $n_a=n_b$. We will refer to these two types of poles
respectively as \emph{unpaired poles} and \emph{paired poles}. This pairing
structure is the key to the emergence of the Pfaffian.

The second crucial property is that the residues of the unpaired and paired poles are such that $x$ disappears from the
corresponding factors (see \eqref{eq:mff2crucial} and \eqref{eq:h1crucial}). As a
consequence of this, after deforming the contours $x$ only appears in factors of the form
$\mff_2(w_a;n_a)$ with $w_a$ living in the deformed contour, which are such
that $|\mff_2(w_a;n_a)|\to0$ as $x\to\infty$. This will allow us to derive an exact formula for
the limit of $\tnuhalfflat(t,2x)$, which after further simplification and a change of
variables which turns the contours into circles of radius 1 (see Section
\ref{sec:simplif}), leads to the result that follows. To state it we define the functions
\begin{gather}\label{eq:2.26a}
  \mfe(w_1,w_2;n_1,n_2)=\frac{(1-\tau^{n_1}w_1w_2)(1-\tau^{n_2}w_1w_2)}{(1-w_1w_2)(1-\tau^{n_1+n_2}w_1w_2)}\\
  \mfh(w_1,w_2;n_1,n_2)=\mfh_1(w_1,w_2;n_1,n_2)\mfh_2(w_1,w_2;n_1,n_2).
\end{gather}
Furthermore, given generic functions $a(z,n)$ and $b(z_1,z_2;n_1,n_2)$, we define
modified functions $\tilde{a}$ and $\tilde{b}$ through
\begin{equation}\label{eq:2.26b}
\tilde a(z,n)=a(\tau^{-\inv2n}z,n)\qqand\tilde
b(z_1,z_2;n_1,n_2)=b(\tau^{-\inv2n_1}z_1,\tau^{-\inv2n_2}z_2;n_1,n_2).
\end{equation}
In the formula that follows this modifier will be applied to all the functions defined in
\eqref{eq:germans-hf} and \eqref{eq:2.26a} (the modification arises through the change of variables
performed at the end of Section \ref{sec:simplif}, see \eqref{eq:defzs} and \eqref{eq:tmcP}).  
  
\begin{thm}\label{thm:flat-moments}
  Let
  \begin{equation}\label{xtoinflimit}
    \nuflat(t)=\lim_{x\to\infty}\tnuhalfflat(t,2x).
  \end{equation}
  Then, with the notation introduced in \eqref{eq:germans-hf}, \eqref{eq:2.26a} and
  \eqref{eq:2.26b}, and letting $C_{0,1}$ be a circle of radius $1$ centred at the origin,
  \begin{equation}
    \begin{split}
      &\nuflat(t)=\sum_{\substack{\ku,\kp\geq0\\\ku+2\kp=k}}\inv{\ku!2^{\kp}\kp!}\sum_{\sigma_1,\dotsc,\sigma_{\ku}\in\{-1,1\}}
      \sum_{\substack{\nu_1,\dotsc,\nu_{\ku}\geq1\\\sigma_a\nu_a\neq\sigma_b\nu_b,\,\,a\neq
          b}}\sum_{\np_1,\dotsc,\np_{\kp}\geq1}
      \uno{\sum_a\nu_a+2\sum_a\np_a=m}\\
      &\qquad\times\itwopii k\idotsint_{C_{0,1}^{\kp}}d\vec{w}^{\rm
        p}\,\prod_{a=1}^{\ku}(\tilde\mff_1\tmfgu)(\sigma_a;n_a)
      \prod_{a=1}^{\kp}(\tilde\mff_1\tmfgp)(\zp_a;\np_a)(\tilde\mff_1\tmfgp)(\tfrac{1}{\zp_a};\np_a)\\
      &\qquad\times\prod_{a=1}^{\ku}\tfrac12\sigma_a\tau^{1-\inv2\nu_a}\prod_{a=1}^{\kp}\frac{(-1)^{\np_a}}{\zp_a}\tau^{2-\inv2\np_a(\np_a+1)}(\tau^{-\np_a})
      \tilde\mfh_2(\zp_a,\tfrac{1}{\zp_a};\np_a,\np_a)\\
      &\qquad\times\prod_{1\leq a<b\leq \ku}\tilde\mfh(\sigma_a,\sigma_b;\nu_a,\nu_b)
      \prod_{a\leq\ku,\,b\leq\kp}\tau^{-\nu_a\np_b}\tilde\mfe(\sigma_a,\zp_b;\nu_a,\np_b)\tilde\mfe(\sigma_a,\tfrac{1}{\zp_b};\nu_a,\np_b)\\
      &\qquad\times\prod_{1\leq a<b\leq
        \kp}\tau^{-2\np_a\np_b}\tilde\mfe(\zp_a,\zp_b;\np_a,\np_b)\tilde\mfe(\zp_a,\tfrac{1}{\zp_b};\np_a,\np_b)
      \tilde\mfe(\tfrac1{\zp_a},\zp_b;\np_a,\np_b)\tilde\mfe(\tfrac{1}{\zp_a},\tfrac{1}{\zp_b};\np_a,\np_b).
    \end{split}\label{eq:nuflat}
  \end{equation}
\end{thm}

In view of \eqref{eq:xlimit} and \eqref{eq:stpt}, this provides a first formula for $\ee^{\rm
  flat}[\tau^{mh(t,0)/2}]$, which will be later turned into
\eqref{eq:flatASEPmoments-intro}. A minor issue with this formula is that the numerator
and denominator in the factors of the form $\tilde\mfh(\sigma_a,\sigma_b;\nu_a,\nu_b)$
vanish when $\sigma_a\sigma_b=1$ and $\nu_a-\nu_b$ is even, in which case the factor is to
be interpreted as in \eqref{eq:defht} below (see the discussion preceding that
definition).

This formula can be regarded as a microscopic (and rigorous) version of formulas
(107)-(108) of \cite{cal-led}, which can be recovered in the weakly asymmetric limit (see also
Section \ref{sec:app-bose}).

To get an idea of how the pairing structure arising from the contour deformation leads to
a Pfaffian formula for $\nuflat(t)$, let us consider first the product
\begin{equation}
  \prod_{1\leq a<b\leq \kp}\tilde\mfe(\zp_a,\zp_b;\np_a,\np_b)\tilde\mfe(\zp_a,\zmp_b;\np_a,\np_b)
  \tilde\mfe(\zmp_a,\zp_b;\np_a,\np_b)\tilde\mfe(\zmp_a,\zmp_b;\np_a,\np_b).\label{eq:es}
\end{equation}
It turns out that this product has the structure of a Schur Pfaffian. To see this, we
recall first \emph{Schur's Pfaffian identity}:
\begin{equation}
  \prod_{1\leq
    a<b\leq2k}\frac{x_a-x_b}{x_a+x_b}=\pf\!\left[\frac{x_a-x_b}{x_a+x_b}\right]_{a,b=1}^{2k}.\label{eq:schur-pf-intro}
\end{equation}
Letting $y_a=\frac{x_a-1}{x_a+1}$ we obtain the identity
\begin{equation}\label{eq:schur-nl-pf-intro}
   \pf\!\left[\frac{y_b-y_a}{y_ay_b-1}\right]_{a,b=1}^{2k}
   =\prod_{1\leq a<b\leq2k} \frac{y_b-y_a}{y_ay_b-1}.
 \end{equation}
 Although the four factors inside the product in \eqref{eq:es} are not of the form
 $(y_b-y_a)/(y_ay_b-1)$, this expression appears after rearranging the factors in the
 numerator and denominator. In fact, letting
 \[y_{2a-1}=\zp_a\qqand y_{2a}=\zmp_a\]
 for $a=1,\dotsc,\kp$, one checks that the product in \eqref{eq:es} is given exactly by
 $\prod_{ a<b}
 \frac{y_b-y_a}{y_ay_b-1}$,
 and thus it equals the Pfaffian in \eqref{eq:schur-nl-pf-intro}.

 In Section \ref{sec:pfaffian} we will show how the last two lines of \eqref{eq:nuflat}
 (which include in particular the product in \eqref{eq:es})
 can be turned into the product of two Pfaffians, times the products of the factors of the
 form $\tau^{-\nu_a\np_b}$ and $\tau^{-2\np_a\np_b}$. We will then show how the two types
 of variables (paired and unpaired) can be put on the same footing, and finally how the
 whole integrand in \eqref{eq:nuflat} can (almost) be written as single Pfaffian. The
 result is \eqref{eq:flatASEPmoments-intro}. From this we can form a generating
 function. This is done in Theorem \ref{thm:exptau-pre}, a particular case of which becomes,
 after further manipulations, Theorem \ref{thm:fredPf-intro}.

 The Pfaffian structure which we have described was first discovered in the context of the
 delta Bose gas with flat initial data in \cite{cal-led}, using a rough singularity
 analysis. They started from the divergent series for the generating function, while we
 start from moment formulas \eqref{eq:flatASEPmoments} which are more reminiscent of the
 Borodin-Corwin approach \cite{borCor}. In hindsight, the two give rather similar
 results. However, the existence of nice exact expressions for the flat case at the ASEP
 level is a surprise, and it is far from clear that it is replicated in other solvable
 models such as $q$-TASEP, the log-Gamma polymer, and the O'Connell-Yor semi-discrete
 polymer.

 We next describe in detail the core of the argument, which is the residue summation.

\section{Moment formula: contour deformation}\label{sec:flat-lim}

\begin{quotation}
  \centering
  \small{\emph{And though the poles were rather small \\They had to count them all}}
\end{quotation}

\vspace{4pt}
The goal of this section is to prove Theorem \ref{thm:flat-moments}. We are trying to
compute the limit as $x\to\infty$ of $\tnuhalfflat$, where $\tnuhalfflat$ is defined in
\eqref{eq:tnuk-intro}. As we noted in the introduction, the only factor in the integrand on the right hand side
of \eqref{eq:tnuk-intro} that depends on $x$ is $\mff_2(w_a,n_a)$, and the base in this power is
strictly bounded by $1$ in modulus whenever $|w_a|>\tau^{-n_a/2}$. Hence in order to compute
the limit, we will deform the $w_a$ contours to some contour which has this property,
which for simplicity we choose to be the contour $C^{n_a}_0$ given by a circle of
radius $\tau^{-n_a/2-\eta}$ centered at $0$.

\subsection{Poles of the integrand}\label{sec:poles}

In order to write down the result of the contour deformation we need to determine first
what poles are crossed as we deform each of the contours. It will be useful
to list first the possible singularities of the integrand:
\begin{enumerate}[label=(\alph*),leftmargin=*]
  \item\label{it:S1} $\mff_1(w_a,n_a)$ has a simple pole at $w_a=0$, but this point is inside the
    original contour $C_{0,\tau^{-\eta}}$. It also has essential singularities at $w_a=-1$ and
    $w_a=-\tau^{-n_a}$, but the first one is contained inside $C_{0,\tau^{-\eta}}$ while the second is
    outside $C^{n_a}_0$. So this factor will not contribute any poles.
  \item\label{it:S2}  $\mff_2(w_a,n_a)$ has simple poles at $w_a=-1$ or
    $w_a=-\tau^{-n_a}$ (depending on the sign of $2x-1$), but as in \ref{it:S1} they are never in the
    deformation region.
  \item\label{it:S3} $\mfgp(w_a,n_a)$ has simple poles at $w_a=-\tau^{-n_a-\ell}$ for any
    $\ell\in\zz_{\geq0}$, but all these points lie outside $C^{n_a}_0$. It also has
    simple poles at $w_a=\pm\tau^{-n_a/2-\ell/2}$ for any $\ell\in\zz_{\geq0}$. These
    points are inside the deformation region only when $\ell=0$, and we have, for $\vsigma=\pm1$,
    \begin{equation}
      \Res_{w_k=\vsigma\tau^{-\inv2n_k}}\mfgp(w_k,n_k)=-\tfrac12\vsigma\tau^{-\inv2n_k}\mfgu(\vsigma\tau^{-\inv2n_k},n_k).
      \label{eq:resmfg}
    \end{equation}
    Observe also that, crucially, for $\vsigma=\pm1$,
    \begin{equation}
      \mff_2(\vsigma\tau^{-\frac12n_k},n_k)=\left(\tfrac{1+\vsigma\tau^{n_k/2}}{1+\vsigma\tau^{-n_k/2}}\tau^{-\frac12n_k}\right)^{2x-1}=\vsigma.\label{eq:mff2crucial}
    \end{equation}
  \item\label{it:S4} The factors on the denominator of $\mfh_1(w_a,w_b;n_a,n_b)$ vanish when
    $w_aw_b=\tau^{-n_a-\ell_a}$ with $\ell_a\in\zz_{\geq0}$ and $w_aw_b=\tau^{-n_b-\ell_b}$ with
    $\ell_b\in\zz_{\geq0}$. Observe that whenever one of the two factors vanishes, the factor
    $\pochinf{w_aw_b;\tau}$ in the numerator of $\mfh_1(w_a,w_b;n_a,n_b)$ is
    going to vanish as well. Thus there can only be a (simple) pole when both factors on the
    denominator vanish at the same time, i.e., when
    $w_b=w_a^{-1}\tau^{-n_a-\ell_a}=w_a^{-1}\tau^{-n_b-\ell_b}$ for some $\ell_a,\ell_b\in\zz_{\geq0}$. Whether these points lie inside the deformation
    region or not will depend on which stage of the deformation we are at:
    \begin{enumerate}[label=(\alph{enumi}.\roman*)]
    \item\label{it:S4i} Assume first that $w_a\in C_{0,\tau^{-\eta}}$ and we are deforming the $w_b$ contour. The
      possible singularities occur at $w_b=w_a^{-1}\tau^{-n_b-\ell_b}$ for some
      $\ell_b\in\zz_{\geq0}$. This point has modulus
      $\tau^{\eta-n_b-\ell_b}\geq\tau^{-\inv2n_b-\eta}$ (recall $\eta\in(0,1/4)$) and thus it is outside the
      deformation region. As a consequence, this factor does not contribute a pole in this case.
    \item\label{it:S4ii} Now assume that $w_a\in C^{n_a}_0$ as we deform the $w_b$ contour. There is a
      singularity when $w_{b}=\bar w_1=\bar w_2$ with $\bar
      w_1=w_a^{-1}\tau^{-n_{b}-\ell_1}$ and $\bar w_2=w_a^{-1}\tau^{-n_{a}-\ell_2}$.
      Observe that $|\bar w^1|=\tau^{n_a/2+\eta-n_b-\ell_1}$ which is
      smaller than $\tau^{-n_{b}/2-\eta}$ if and only if
      $\ell_1\leq(n_a-n_{b})/2+2\eta$. Similarly we have $|\bar
      w^2|=\tau^{-n_a/2+\eta-\ell_2}$ which is smaller than $\tau^{-n_{b}/2-\eta}$ if
      and only if $\ell_2\leq(n_{b}-n_a)/2+2\eta$. Remembering that we also need to have
      $\ell_1,\ell_2\geq0$ these conditions on $\ell_1,\ell_2$ (and the fact that $\eta\in(0,1/4)$) imply that $n_{a}=n_b$, and
      thus also that $\ell_1=\ell_2=0$. Hence we have a simple pole at
      $w_{b}=w_a^{-1}\tau^{-n_{a}}$ when $n_{a}=n_b$, whose residue is computed to be
      \begin{equation}
      \Res_{w_b=w_a^{-1}\tau^{-n_{a}}}\mfh_1(w_{a},w_b;n_{a},n_a)
      =\tfrac{(-1)^{n_a+1}}{w_a}\tau^{-n_a(n_a+1)/2}(\tau^{-n_a}-1).\label{eq:resh1}
    \end{equation}
    Note that, crucially, in the case $n_{a}=n_b$ we have
    \begin{equation}
      \mff_2(w_b^{-1}\tau^{2-n_b},n_{b})\mff_2(w_b,n_{b})=\left(\tfrac{\tau+w_b^{-1}\tau^{2}}{\tau+w_b^{-1}\tau^{2-n_b}}
      \tfrac{\tau+\tau^{n_b}w_b}{\tau+w_b}\tau^{-n_b}\right)^{x-1}=1.\label{eq:h1crucial}
    \end{equation}
  \item\label{it:S3iii} The third possibility of interest is when $w_a=\vsigma\tau^{-n_a/2}$ for $\vsigma=\pm1$ (which
    happens when a residue in $w_a$ has already been evaluated as one of the poles from
    \ref{it:S3}). Exactly as in \ref{it:S4ii}, a singularity may only arise when $n_a=n_b$
    at $w_{b}=\vsigma\tau^{-n_{a}/2}$. But in this case the factor
    $\mfh_2(w_{b},\vsigma\tau^{-n_a/2};n_{a},n_a)$ has a double zero at this
    point, and hence the double zero in the denominator is canceled, which means that there is no pole in this case.
    \end{enumerate}
  \item\label{it:S5}  Finally, the factors on the denominator of $\mfh_2(w_a,w_b;n_a,n_b)$ vanish when
    $w_b=w_a\tau^{n_a}$ and $w_b=w_a\tau^{-n_b}$. The first point is inside $C_{0,\tau^{-\eta}}$ whenever
    $w_a$ is in $C_{0,\tau^{-\eta}}$ or $C^{n_a}_0$. The second point is outside $C^{n_b}_0$ whenever
    $w_a$ is in $C_{0,\tau^{-\eta}}$ or $C^{n_a}_0$. So this factor will not contribute any poles.
\end{enumerate}

Finally, consider the product $(\mfh_1\mfh_2)(w_a,\vsigma_b\tau^{-n_b/2};n_a,n_b)$ at
$w_a=\vsigma_a\tau^{-n_a/2}$ for $\vsigma_a,\vsigma_b=\pm1$, which occurs when both
$w_a$ and $w_b$ are evaluated as residues coming from the poles in \ref{it:S3}, with $w_b$
having been evaluated first. When $n_a-n_b$ is even and $\vsigma_{a}\vsigma_b=1$,
both the numerator and the denominator of the factor coming from $\mfh_1$ vanish. Since
the evaluation at these points comes from a residue computation (in $w_a$), the product should be computed as
\begin{equation}\label{eq:defht}
(\mfh_1\mfh_2)(\vsigma_a\tau^{-\inv2n_a},\vsigma_b\tau^{-\inv2n_b};n_a,n_b)
:=\qquad\smashoperator{\lim_{w\to\vsigma_a\tau^{-n_a/2}}}\qquad(\mfh_1\mfh_2)(w,\vsigma_b\tau^{-\inv2n_b};n_a,n_b).
\end{equation}
The right hand side is computed in the next lemma. We stress that, in the formulas that
will follow, we will use the extension of $\mfh_1\mfh_2$ defined in \eqref{eq:defht}
mostly without reference.

\begin{lem}\label{lem:mfeu}
  For $n_a,n_b\in\zz_{\geq0}$ and $\vsigma_a,\vsigma_b=\pm1$ we have
  \[\lim_{w\to\vsigma_{a}\tau^{-n_a/2}}
    ~~(\mfh_1\mfh_2)(w,\vsigma_b\tau^{-\inv2n_b};n_{a},n_b)
  =(-\vsigma_a\vsigma_b)^{n_a\wedge n_b}\tau^{-\inv2n_an_b}\tfrac{\big|\vsigma_b\tau^{\frac12n_b}
    -\vsigma_a\tau^{\frac12n_a}\big|}{1-\vsigma_a\vsigma_b\tau^{\frac12(n_a+n_b)}}.\]
\end{lem}

\begin{proof}
  We rewrite the limit as
  \begin{equation}\label{eq:limmfh1mfh2}
    \lim_{w\to1}(\mfh_1\mfh_2)(\vsigma_a\tau^{-\frac12n_a}w,\vsigma_b\tau^{-\frac12n_b};n_a,n_b).
  \end{equation}
  A simple computation (see the proof of Lemma \ref{lem:simplif} for a similar one) leads to
  \[(\mfh_1\mfh_2)(\vsigma_a\tau^{-\frac12n_a}w,\vsigma_b\tau^{-\frac12n_b};n_a,n_b)
  =\prod_{\ell=1}^{n_a}\frac{1-\tau^{\ell-\frac12n_a-\frac12n_b}\bar
    w}{1-\tau^{\ell-\frac12n_a+\frac12n_b}\bar w}\]
  where $\bar w=\vsigma_a\vsigma_bw$. This product equals
  \begin{align}
    \prod_{\ell=1}^{n_a}\frac{1-\tau^{1-\ell+\frac12n_a-\frac12n_b}\bar w}{1-\tau^{\ell-\frac12n_a+\frac12n_b}\bar w}
    &=\prod_{\ell=1}^{n_a}\frac{\bar w-\tau^{\ell-1-\frac12n_a+\frac12n_b}}{1-\tau^{\ell-\frac12n_a+\frac12n_b}\bar w}(-\tau^{1-\ell+\frac12n_a-\frac12n_b})\\
    &=(-1)^{n_a}\tau^{\frac12n_a-\frac12n_an_b}\frac{\bar w-\tau^{\frac12n_b-\frac12n_a}}{1-\tau^{\frac12n_a+\frac12n_b}\bar w}
    \frac{\prod_{\ell=2}^{n_a}(\bar w-\tau^{\ell-1-\frac12n_a+\frac12n_b})}{\prod_{\ell=1}^{n_a-1}(1-\tau^{\ell-\frac12n_a+\frac12n_b}\bar w)}.
  \end{align}
  The last ratio equals $\prod_{\ell=1}^{n_a-1}\frac{\vsigma_a\vsigma_bw-\tau^{\ell-\frac12n_a+\frac12n_b}}{1-\tau^{\ell-\frac12n_a+\frac12n_b}\vsigma_a\vsigma_bw}$,
  and each factor in this product tends to $\vsigma_a\vsigma_b$ as $w\to1$ unless $\ell-\frac12n_a+\frac12n_b=0$, in
  which case it tends to $-1$. Therefore the whole ratio goes to
  $(\vsigma_a\vsigma_b)^{n_a-2}(-1)^{\uno{n_a-n_b\in2\zz_{>0}}}(\vsigma_a\vsigma_b)^{1-\uno{n_a-n_b\in2\zz_{>0}}}
  =(\vsigma_a\vsigma_b)^{n_a-1}(-\vsigma_a\vsigma_b)^{\uno{n_a-n_b\in2\zz_{>0}}}$,
  and thus the limit in \eqref{eq:limmfh1mfh2} equals
 \begin{multline}
    -(-\vsigma_a\vsigma_b)^{n_a+\uno{n_a-n_b\in2\zz_{>0}}-1}\tau^{\frac12n_a-\frac12n_an_b}
    \tfrac{\vsigma_a\vsigma_b-\tau^{\frac12n_b-\frac12n_a}}{1-\tau^{\frac12n_a+\frac12n_b}\vsigma_a\vsigma_b}\\
    =\begin{dcases*}
      (-1)^{n_a\wedge n_b}\tau^{-\frac12n_an_b}
      \tfrac{\big|\tau^{\frac12n_b}-\tau^{\frac12n_a}\big|}{1-\tau^{\frac12n_a+\frac12n_b}}
      & if $\vsigma_a\vsigma_b=1$\\
      \tau^{-\frac12n_an_b}
      \tfrac{\tau^{\frac12n_b}+\tau^{\frac12n_a}}{1+\tau^{\frac12n_a+\frac12n_b}}
      & if $\vsigma_a\vsigma_b=-1$,
    \end{dcases*}
  \end{multline}
which gives the claimed formula.
\end{proof}

\subsection{Contour deformation and computation of the limit}\label{sec:contDefLimit}

As we mentioned, we will move the contours one by one, starting from $w_k$, then $w_{k-1}$, and so on. From the above
discussion, there are two types of poles which are crossed. First, as we expand the
$w_a$ contour we will cross singularities at $w_a=\pm\tau^{-n_a/2}$ as in \ref{it:S3}
above. We will refer to these as \emph{unpaired poles}. The $n$ variables associated to
these poles will be called \emph{unpaired variables}, and will be denoted with a
superscript u (note that unpaired poles have no $w$ variables associated to them, since a residue has been computed).

The second type of pole appears at $w_a=w_b^{-1}\tau^{-n_b}$ (for $a<b$) when $w_b\in
C^{n_b}_0$ and $n_a=n_b$. We refer to these as \emph{paired poles}. To each paired pole
corresponds a pair of $w$ variables and a pair of $n$ variables: $w_b$, which is
integrated over $C^{n_b}_0$, and $w_a$, which is set to $w_a=w_b^{-1}\tau^{-n_b}$, and
similarly $n_b$ and $n_a$, which are set to be equal. We call these \emph{paired
  variables}, and denote them with a superscript p.

Finally, in addition to the paired and unpaired variables, there are \emph{free variables},
denoted with a superscript f, corresponding to the $w$ variables (and associated $n$ variables)
which, after the deformation, end up integrated on a $C^n_0$ contour (note that these have no
paired variable associated to them).

The result of deforming all the $w_a$ contours involves a sum over all possible ways in
which the $k$ original variables can be grouped into $\ku$ unpaired variables, $\kp$ pairs
of paired variables, and $\kf$ free variables, which of course means that
$\ku+2\kp+\kf=k$. Given such a grouping, the original $w_a$'s and $n_a$'s will be split into
strings of variables: $\vnu$ for the unpaired ones (of size $\ku$), $\vnp$ and $\vwp$ for
the paired ones (of size $\kp$), and $\vnf$ and $\vwf$ for the free ones (of size
$\kf$). It is important to note that in the case of the paired variables, the variables
forming these strings will appear twice in the formulas that will follow, once associated
to $\vwp$ and once associated to the ``paired'' string $\vwmp$ defined by
\begin{equation}\label{eq:defwmp}
\wmp_a=\tau^{-\np_a}/\wp_a\qquad\text{for }a=1,\dotsc,\kp.
\end{equation}
These variables should be regarded as ``virtual'', in the sense that they are not being
directly summed or integrated, but are instead defined directly in terms of other
variables. On the other hand, in the case of the unpaired variables there is an additional
string $\vec\vsigma$ of $\pm1$ variables, which will indicate whether the unpaired pole is
evaluated at $\tau^{-n_a/2}$ or $-\tau^{-n_a/2}$.  It will be convenient to introduce
another string of virtual variables $\vwu$, defined as follows:
\begin{equation}\label{eq:defwu}
\wu_a=\vsigma_a\tau^{-\inv2\nu_a}\qquad\text{for }a=1,\dotsc,\ku.
\end{equation}

Before writing down the result of the contour deformation we need to introduce some
additional notation which will help us express some products arising from the pairing
structure which emerges from the contour deformation. Assume we are given a collection of
pairs of strings of unpaired variables $(\vwu,\vnu)$, paired variables $(\vwp,\vnp)$ and
free variables $(\vwf,\vnf)$. Given a function $g$ of a single $w$ and a single $n$ we define
\begin{equation}\label{eq:aug-g}
\mcP^{\rm {u,p,f}} g(\vwu,\vwp,\vwf;\vnu,\vnp,\vnf)
=\prod_{a=1}^{\ku}g(\wu_a,\nu_a)\prod_{a=1}^{\kp}g(\wp_a,\np_a)g(\wmp_a,\np_a)\prod_{a=1}^{\kf}g(\wf_a,\nf_a).
\end{equation}
That is, we multiply over all possible variables, accounting for the extra set of paired ones.
Similarly, given a function $h$ of a pair $w_1,w_2$ and a pair $n_1,n_2$ we define
\begin{equation}\label{eq:aug-h}
\begin{split}
\mcP^{\rm {u,p,f}}&(\vwu,\vwp,\vwf;\vnu,\vnp,\vnf)
=\prod_{1\leq a<b\leq\ku}h(\wu_a,\wu_b;\nu_a,\nu_b)\prod_{1\leq a<b\leq\kf}h(\wf_a,\wf_b;\nf_a,\nf_b)\\
&\times\prod_{1\leq a<b\leq\kp}h(\wp_a,\wp_b;\np_a,\np_b)h(\wp_a,\wmp_b;\np_a,\np_b)h(\wmp_a,\wp_b;\np_a,\np_b)h(\wmp_a,\wmp_b;\np_a,\np_b)\\
&\times\prod_{a\leq\ku,\,b\leq\kp}h(\wu_a,\wp_b;\nu_a,\np_b)h(\wu_a,\wmp_b;\nu_a,\np_b)\prod_{a\leq\ku,\,b\leq\kf}h(\wu_a,\wf_b;\nu_a,\nf_b)\\
&\times\prod_{a\leq\kp,\,b\leq\kf}h(\wp_a,\wf_b;\np_a,\nf_b)h(\wmp_a,\wf_b;\np_a,\nf_b).
\end{split}
\end{equation}
Here we multiply over all possible pairs of different variables, accounting again for the
extra set of paired variables, but omitting factors with variables from the same pair
(that is, omitting factors of the form $h(\wp_a,\wmp_a;\np_a,\np_a)$).

We extend the notations \eqref{eq:aug-g} and \eqref{eq:aug-h} as follows: if any group of
variables is omitted from the superscript (and the argument), it simply means that those
variables are omitted from the products. Thus, for example,
\begin{equation}\label{eq:aug-g2}
\mcP^{\rm {u}}g(\vwu;\vnu)=\prod_{a=1}^{\ku}g(\wu_a,\nu_a)
\end{equation}
while\noeqref{eq:aug-g2}
\begin{multline}
\label{eq:aug-h2}
\mcP^{\rm {u,p}}h(\vwu,\vwp;\vnu,\vnp)\\
=\quad\smashoperator{\prod_{1\leq
  a<b\leq\ku}}\quad h(\wu_a,\wu_b;\nu_a,\nu_b)\quad\smashoperator{\prod_{a\leq\ku,\,b\leq\kp}}\quad
h(\wu_a,\wp_b;\nu_a,\np_b)h(\wu_a,\wmp_b;\nu_a,\np_b)\\
\times\prod_{1\leq a<b\leq\kp}h(\wp_a,\wp_b;\np_a,\np_b)h(\wp_a,\wmp_b;\np_a,\np_b)h(\wmp_a,\wp_b;\np_a,\np_b)h(\wmp_a,\wmp_b;\np_a,\np_b).
\end{multline}
We will slightly abuse this notation by omitting the superscript in $\mcP$ (which in any
case is implicit in the superscripts of the argument). For convenience we will also write
\begin{equation}\label{eq:aug-mfh}
\mcP\mfh(\cdot;\cdot)
=\mcP\mfh_1(\cdot;\cdot)\mcP\mfh_2(\cdot;\cdot).
\end{equation}
Finally, we introduce the index set\noeqref{eq:aug-mfh}\noeqref{eq:LambdaSet}
\begin{multline}\label{eq:LambdaSet} 
\Lambda^{\bk}_{k_1,k_2,k_3}=\Big\{(\vec{\vsigma},\vec{n}^1,\vec{n}^2,\vec{n}^3)\in\{-1,1\}^{k_1}\times\zz_{\geq1}^{k_1}\times\zz_{\geq_1}^{k_2}\times\zz_{\geq1}^{k_3}\!:\\
\sum_{a=1}^{k_1}n^1_a+2\sum_{a=1}^{k_2}n^2_a+\sum_{a=1}^{k_3}n^3_a=\bk,\vsigma_an^1_a\neq\vsigma_bn^1_b~\forall\,1\leq a<b\leq k_1\Big\}
\end{multline}
and the notation $C^{\vec{n}}_0=C^{n_1}_0\times\dotsm\times C^{n_k}_0$ for $\vec{n}$ of
size $k$. 

We are ready now to state the result of the complete contour deformation.

\begin{prop}\label{prop:contourDeformationResult}
For every $k,\bk\in\zz_{\geq0}$ with $k\leq\bk$ we have
\begin{equation}
\tnuhalfflat(t,2x)=\sum_{\substack{\ku,\kp,\kf\geq0\\\ku+2\kp+\kf=k}}\inv{\ku!2^{\kp}\kp!\kf!}\mfZhalfflat(\ku,\kp,\kf),\label{eq:tnukmfZ}
\end{equation}
where $\tnuhalfflat$ was defined in \eqref{eq:tnuk-intro} and
\begin{equation}
\begin{split}
&\mfZhalfflat(\ku,\kp,\kf)=\quad\quad\smashoperator{\sum_{(\vec{\vsigma},\vnu,\vnp,\vnf)\in\Lambda^{\bk}_{\ku,\kp,\kf}}}\quad\quad
\itwopii{\kp+\kf}\int_{C^{\vnp}_0\times C^{\vnf}_0}d\vwp\,d\vwf\,\mcP\mff_1(\vwu,\vwp,\vwf;\vnu,\vnp,\vnf)\\
&\hspace{1.2in}\times\mcP\mff_2(\vwf;\vnf)
\mcP\mfgu(\vwu;\vnu)\mcP\mfgp(\vwp,\vwf;\vnp,\vnf)\mcP\mfh(\vwu,\vwp,\vwf;\vnu,\vnp,\vnf)\\
&\hspace{1.2in}\times\prod_{a=1}^{\ku}\tfrac12\wu_a
\prod_{a=1}^{\kp}(-1)^{\np_a}\tau^{-\np_a(\np_a+1)/2}(1-\tau^{\np_a})\wmp_a\,\mfh_2(\wp_a,\wmp_a;\np_a,\np_a).
\end{split}\label{eq:mfZhalfflat}
\end{equation}
\end{prop}

The proposition follows from a slightly more general result which we prove next. Define
the annulus $R^n_0=\big\{w\in\cc\!:\tau^{-\eta}\leq|w|\leq\tau^{-n/2-\eta}\big\}$ and the
set $D=\prod_{n\geq1}R^n_0\times\{n\}$.  We assume as given functions
$\phi\!:D\longrightarrow\cc$ and $\psi\!:D\times D\longrightarrow\cc$ such that
$\phi(w,n)$ is analytic in $w\in R^n_0$ and $\psi(w_1,w_2;n_1,n_2)$ is analytic in each
variable $w_1\in R^{n_1}_0$ and $w_2\in R^{n_2}_0$ (we are slightly abusing notation here
to write $\psi(w_1,w_2;n_1,n_2)$ instead of $\psi((w_1,n_1),(w_2,n_2))$).

\begin{prop}\label{prop:x-deformation}
  Fix $k,\bk\in\zz_{\geq1}$ with $k\leq\bk$ and functions as $\phi$ and $\psi$ as above. Then
  \begin{align}
    &\sum_{\substack{n_1,\dotsc,n_k\geq1\\n_1+\dotsm+n_k=\bk}}\itwopii{k}\int_{C_{0,\tau^{-\eta}}^k}d\vec
    w\,\prod_{a}\phi(w_a,n_a)\mfgp(w_a,n_a)\prod_{a<b}\psi(w_a,w_b;n_a,n_b)\mfh_1(w_a,w_b;n_a,n_b)\\[-30pt]
    &\hspace{0.5in}=\sum_{\substack{\ku,\kp,\kf\geq0\\\ku+2\kp+\kf=k}}\frac{k!}{\ku!2^{\kp}\kp!\kf!}
    \sum_{(\vec\vsigma,\vec{n}^{\rm u},\vec{n}^{\rm{p}},\vec{n}^{\rm
        f})\in\Lambda^{\bk}_{\ku,\kp,\kf}}\itwopii{\kp+\kf}\int_{C^{\vnp}_0}d\vwp\int_{C^{\vnf}_0}d\vwf\,\\
    &\hspace{1.5in}\times\mcP\phi(\vwu,\vwp,\vwf;\vnu,\vnp,\vnf)\mcP\psi(\vwu,\vwp,\vwf;\vnu,\vnp,\vnf)\\
    &\hspace{1.5in}\times\mcP\mfgu(\vwu;\vnu)\mcP\mfgp(\vwp,\vwf;\vnp,\vnf)
    \mcP\mfh_1(\vwu,\vwp,\vwf;\vnu,\vnp,\vnf)\\
    &\hspace{1.5in}\times\prod_{a=1}^{\ku}\tfrac12\wu_a
    \prod_{a=1}^{\kp}(-1)^{\np_a}\tau^{-\np_a(\np_a+1)/2}(1-\tau^{\np_a})\wmp_a\,\psi(\wp_a,\wmp_a;\np_a,\np_a).
  \end{align}
\end{prop}

\begin{proof}
  The proof is by induction in $k$ (for fixed $\bk$). First we check the case $k=1$. Note
  that $\mcP\mfh_1(\vec w;\vec n)=1$ in this case (because the product is empty), and thus
  the only poles crossed as we deform are the ones coming from $\mfgp(w_1,n_1)$ as in
  \ref{it:S3} in page \pageref{it:S3}. On the other hand, the condition $\ku+2\kp+\kf=1$
  implies that the sum on the right hand side has two terms: one with $\ku=1$ and the
  other two set to 0, and one with $\ku=1$ and the other two set to 0. Since
  $\phi(w_1,n_1)$ is analytic on the region of interest, one checks directly that the
  first term corresponds to the computation of the residues at the two poles (keeping in
  mind that the poles are passed from the outside, so we get an additional minus sign),
  while the second term corresponds to the integral on the deformed contour.

We assume now that the formula holds for a given $k$ and rewrite the integrand in the case $k+1$, for fixed $w_{k+1}\in C_{0,\tau^{-\eta}}$, as
\[(\phi\hspace{0.05em}\mfgp)(w_{k+1},n_{k+1})\prod_{a=1}^k(\tilde\phi\hspace{0.05em}\mfgp)(w_a,n_a)\prod_{1\leq
  a<b\leq k}(\psi\hspace{0.05em}\mfh_1)(w_a,w_b;n_a,n_b)(w_a,w_b;n_a,n_b)\] with
\[\tilde\phi(w_a,n_a)=\phi(w_a,n_a)(\psi\hspace{0.05em}\mfh_1)(w_a,w_{k+1};n_a,n_{k+1}).\]
By \ref{it:S4i} in page \pageref{it:S4i} and the hypotheses on $\phi$ and $\psi$, $\tilde\phi(\cdot,n)$ is analytic in the deformation region and thus, using the inductive hypothesis, the integral equals
  \begin{equation}
  \begin{split}
    &\sum_{n_{k+1}=1}^{\bk-k}\inv{2\pi\I}\int_{C_{0,\tau^{-\eta}}}dw_{k+1}\,\mfgp(w_{k+1},n_{k+1})\phi(w_{k+1},n_{k+1})\sum_{\substack{\ku,\kp,\kf\geq0\\\ku+\kp+\kf=k}}\frac{k!}{\ku!2^{\kp}\kp!\kf!}
    \\
    &\quad\times\qquad\smashoperator{\sum_{(\vec\vsigma,\vec{n}^{\rm u},\vec{n}^{\rm{p}},\vec{n}^{\rm
        f})\in\Lambda^{\bk-n_{k+1}}_{\ku,\kp,\kf}}}\qquad\quad\itwopii{\kp+\kf}\int_{\wt C^{\vnp}_0\times C^{\vnf}_0}d\vwp\,d\vwf
    \,\mcP\tilde\phi(\vwu,\vwp,\vwf;\vnu,\vnp,\vnf)\\
&\qquad\times\mcP\mfgu(\vwu;\vnu)\mcP\mfgp(\vwp,\vwf;\vnp,\vnf)\mcP\mfh(\vwu,\vwp,\vwf;\vnu,\vnp,\vnf)\\
&\qquad\times\prod_{a=1}^{\ku}\tfrac12\wu_a
\prod_{a=1}^{\kp}(-1)^{\np_a}\tau^{-\np_a(\np_a+1)/2}(1-\tau^{\np_a})\wmp_a\,\psi(\wp_a,\wmp_a;\np_a,\np_a).
  \end{split}\label{eq:Dkinduc}
\end{equation}

Now we have to deform the $w_{k+1}$ contour to $C^{n_{k+1}}_0$. We list the singularities that are crossed:
\begin{enumerate}[label=(\alph*),leftmargin=*]
\item Since $\phi$ is analytic in $w_{k+1}$ in the deformation region, the first line of
  \eqref{eq:Dkinduc} only contributes the poles of $\mfgp(w_{k+1},n_{k+1})$ as specified
  in \ref{it:S3} in page \pageref{it:S3}. By \eqref{eq:resmfg} it is clear that computing
  the associated residues corresponds to grouping $(w_{k+1},n_{k+1})$ together with the
  unpaired variables (i.e., turning it into $(\wu_{\ku+1},\nu_{\ku+1})$, keeping
  \eqref{eq:defwu} in mind).
\item Now we consider the second line of \eqref{eq:Dkinduc}. The factors involving $\phi$
  and $\psi$ are analytic in $w_{k+1}$ in the deformation region. As for the factors
  involving $\mfh_1$, by \ref{it:S4} in page \pageref{it:S4}, the only ones which
  contribute poles are those of the form $\mfh_1(\wf_a,w_{k+1};\nf_a,n_{k+1})$ (see in
  particular \ref{it:S4ii}). In view of \eqref{eq:resh1}, this corresponds to removing
  $(\wf_a,\nf_a)$ from the group of free variables and adding it to the paired variables
  together with $(w_{k+1},n_{k+1})$, which becomes its pair (i.e., the pairs
  $(\wf_a,\nf_a)$ and $(w_{k+1},n_{k+1})$ become $(\wp_{\kp+1},\np_{\kp+1})$ and
  $(\wmp_{\kp+1},\np_{\kp+1})$, keeping \eqref{eq:defwmp} in mind.
\item The last two lines of \eqref{eq:Dkinduc} do not depend on $w_{k+1}$, and thus contribute no additional poles.
\end{enumerate}
In addition to the residues in (a) and (b) above, which respectively increase $\ku$ and
$\kp$ by 1, we get one more term with $w_{k+1}$ integrated on the contour
$C^{n_{k+1}}_0$. As in the above cases, this can be thought of as grouping
$(w_{k+1},n_{k+1})$ together with the free variables, this time increasing $\kf$ by one.

The above discussion explains the form taken by the integrand on the right hand side of
the identity claimed in the proposition (recall that we are passing poles from the
outside, accounting for additional minus sign in each case). It only remains to verify
that the inductive argument also yields the combinatorial factor
$k!/(\ku!2^{\kp}\kp!\kf!)$. Consider first those terms in the resulting sum which are
grouped into sizes $(\ku+1,2\kp,\kf)$ (which satisfy $\ku+1+2\kp+\kf=k+1$). In terms of
our inductive argument, these terms come from deforming the last contour when the first
$k$ variables have been grouped into sizes either $(\ku,2\kp,\kf)$ (so the last one is set
as unpaired), $(\ku+1,2(\kp-1),\kf+1)$ (so the last one is set as paired, which can be
done in $\kf+1$ ways), or $(\ku+1,2\kp,\kf-1)$ (so the last one is set as free). By the
symmetry of the integrand, and since
\begin{multline}
\frac{k!}{\ku!2^{\kp}\kp!\kf!}+\frac{k!}{(\ku+1)!2^{\kp-1}(\kp-1)!(\kf+1)!}(\kf+1)+\frac{k!}{(\ku+1)!2^{\kp}\kp!(\kf-1)!}\\=\frac{(k+1)!}{(\ku+1)!2^{\kp}\kp!\kf!},
\end{multline}
we obtain the desired combinatorial factor on the resulting formula. The other cases (where $\kp$ or $\kf$ are increased by one) can be treated similarly.
\end{proof}

\begin{proof}[Proof of Proposition \ref{prop:contourDeformationResult}]
  In Proposition \ref{prop:x-deformation} set $\phi=\mff_1\mff_2$ and $\psi_2=\mfh_2$. By
  \ref{it:S1} and \ref{it:S2} in the list in page \pageref{it:S1} as well as the paragraph
  preceding \eqref{eq:defht}, these functions satisfy the hypotheses of the
  proposition. The result follows from the formula given in that result, noting
  that, by \eqref{eq:mff2crucial} and \eqref{eq:h1crucial}, many of the $\mff_2$ factors
  disappear.
\end{proof}

We get now to the crucial point. The only factor depending on $x$ in
\eqref{eq:mfZhalfflat} is $\mff_2(\wf_a,\nf_a)$, which corresponds to the free variables. As we mentioned, 
$\big|\frac{1+\tau^{\nf_a}\wf_a}{1+\wf_a}\tau^{-\frac12\nf_a}\big|<1$
whenever $|\wf_a|>\tau^{-\frac12\nf_a}$. Consequently,
\[\lim_{x\to\infty}\mff_2(\wf_a,\nf_a)=0.\]
Now we apply to this to the computation of the limit as $x\to\infty$ of
$\tnuhalfflat(t,2x)$, which in view of Proposition \ref{prop:contourDeformationResult} amounts to
computing the limit of $\mfZhalfflat(\ku,\kp)$. Since all the sums
in \eqref{eq:mfZhalfflat} are finite, while all the integrals are over finite contours and
the integrands are continuous along them, we may take the limit inside and conclude that
every term with $\kf\geq1$ vanishes in the limit. Using this, and keeping in mind the
notation introduced in \eqref{eq:aug-g}-\eqref{eq:LambdaSet}, we deduce the following

\begin{prop}\label{prop:x-limit}
  For every $\bk\in\zz_{\geq0}$ we have
  \begin{gather}
    \nuflat(t):=\lim_{x\to\infty}\tnuhalfflat(t,2x)=\sum_{\substack{\ku,\kp\geq0\\\ku+2\kp=k}}\inv{\ku!2^{\kp}\kp!}\mfZflat(\ku,\kp)
    \shortintertext{with}
    \begin{aligned}
      &\mfZflat(\ku,\kp)=\quad\quad\smashoperator{\sum_{(\vec{\vsigma},\vnu,\vnp)\in\Lambda^{\bk}_{\ku,\kp}}}\quad\quad
      \itwopii{\kp}\int_{C^{\vnp}_0}d\vwp\,\mcP\mff_1(\vwu,\vwp;\vnu,\vnp)\\
      &\hspace{1.6in}\times
      \mcP\mfgu(\vwu;\vnu)\mcP\mfgp(\vwp;\vnp)\mcP\mfh(\vwu,\vwp;\vnu,\vnp)\prod_{a=1}^{\ku}\tfrac12\wu_a\\
      &\hspace{1.6in}\times\prod_{a=1}^{\kp}(-1)^{\np_a}\tau^{-\np_a(\np_a+1)/2}(1-\tau^{\np_a})\wmp_a\,\mfh_2(\wp_a,\wmp_a;\np_a,\np_a).
    \end{aligned}
  \end{gather}
\end{prop}

\subsection{Simplification of the \texorpdfstring{$\mfh$}{h} factors}\label{sec:simplif}

The next step is to observe that there is an important simplification in the factors
appearing in $\mcP\mfh$ (which, we recall, is defined as the product of $\mcP\mfh_1$
and $\mcP\mfh_2$), due to the emergence of the paired structure. Define
\begin{equation}
   \label{eq:defmfe}
   \mfe(w_1,w_2;n_1,n_2)=\frac{(1-\tau^{n_1}w_1w_2)(1-\tau^{n_2}w_1w_2)}{(1-w_1w_2)(1-\tau^{n_1+n_2}w_1w_2)}.
\end{equation}

\begin{lem}\label{lem:simplif}
  For any $w_1,w_2\in\cc$ and $n_1,n_2\in\zz_{\geq0}$ for which both sides are well
  defined,
  \begin{multline}
    \mfh_1(w_1,w_2;n_1,n_2)\mfh_1(\tau^{-n_1}/w_1,\tau^{-n_2}/w_2;n_1,n_2)\\
    \times\mfh_2(\tau^{-n_1}/w_1,w_2;n_1,n_2)\mfh_2(w_1,\tau^{-n_2}/w_2;n_1,n_2)
    =\tau^{-n_1n_2}\mfe(w_1,w_2;n_1,n_2).
  \end{multline}
\end{lem}

\begin{proof}
  Let $h_{1,2}(w_1,w_2;n_1,n_2)=\mfh_1(w_1,w_2;n_1,n_2)\mfh_2(\tau^{-n_1}/w_1,w_2;n_1,n_2)$. Noting
  that the first factor in each of the Pochhammer symbols in $\mfh_1$ cancel with like
  factors in $\mfh_2$ we get
  \[h_{1,2}(w_1,w_2;n_1,n_2)=\frac{\pochinf{\tau w_1w_2;\tau}\pochinf{\tau^{1+n_1+n_2}w_1w_2;\tau}}
    {\pochinf{\tau^{1+n_1}w_1w_2;\tau}\pochinf{\tau^{1+n_2}w_1w_2;\tau}}.\]
  Then the product
  $h_{1,2}(w_1,w_2;n_1,n_2)h_{1,2}(\tau^{-n_1}/w_1,\tau^{-n_2}/w_2;n_1,n_2)$, which is
  what we are trying to compute, equals
  \[\frac{\pochinf{\tau w_1w_2;\tau}}{\pochinf{\tau^{1+n_1}w_1w_2;\tau}}\frac{\pochinf{\tau/(w_1w_2);\tau}}
    {\pochinf{\tau^{1-n_1}/(w_1w_2);\tau}}\frac{\pochinf{\tau^{1+n_1+n_2}w_1w_2;\tau}}
    {\pochinf{\tau^{1+n_2}w_1w_2;\tau}}
    \frac{\pochinf{\tau^{1-n_1-n_2}/(w_1w_2);\tau}}{\pochinf{\tau^{1-n_2}/(w_1w_2);\tau}}.\]
  Writing $\bar w=w_1w_2$, the product of the first two ratios above equals
  \[\prod_{\ell=0}^{n_1-1}\frac{1-\tau^{\ell+1}\bar w}{1-\tau^{1-n_1+\ell}/\bar w}
  =\prod_{\ell=0}^{n_1-1}\frac{1-\tau^{\ell+1}\bar w}{1-\tau^{n_1-\ell-1}\bar w}
  (-\tau^{n_1-\ell-1}\bar w)=\frac{1-\tau^{n_1}\bar w}{1-\bar w}
  \prod_{\ell=0}^{n_1-1}(-\tau^{n_1-\ell-1}\bar w),\]
  while in a similar way one checks that the product of the last two ratios equals
  \begin{equation}
    \prod_{\ell=0}^{n_1-1}\frac{1-\tau^{1-n_1-n_2+\ell}/\bar
      w}{1-\tau^{1+\ell+n_2}\bar w}
    =\frac{1-\tau^{n_2}\bar w}{1-\tau^{n_1+n_2}\bar
      w}\prod_{\ell=0}^{n_1-1}\frac{-1}{\tau^{n_1+n_2-\ell-1}\bar w}.
  \end{equation}
  Multiplying out the last two expressions gives the result.
\end{proof}

Now note that, in view of \eqref{eq:defwmp} and \eqref{eq:defwu},
$\tau^{-\nu_a}/\wu_a=\wu_a$, $\wmp_a=\tau^{-\np_a}/\wp_a$ and
$\tau^{-\np_a}/\wmp_a=\wp_a$. Recalling the definition of the product $\mcP\mfh_1\mcP\mfh_2$ through
\eqref{eq:aug-h2} and pairing the factors coming from the last two products in a suitable
way, the lemma implies that
\begin{multline}
\label{eq:simplif-mfh}
\mcP\mfh(\vwu,\vwp;\vnu,\vnp)
=\mcP\mfh(\vwu;\vnu)\quad\smashoperator{\prod_{a\leq\ku,\,b\leq\kp}}\quad\tau^{-\nu_a\np_b}\mfe(\wu_a,\wp_b;\nu_a,\np_b)\\
\times\quad\smashoperator{\prod_{1\leq
    a<b\leq\kp}}\quad\tau^{-2\np_a\np_b}\mfe(\wp_a,\wp_b;\np_a,\np_b)\mfe(\wmp_a,\wp_b;\np_a,\np_b).
\end{multline}

Finally, we introduce the change of variables
\begin{equation}
\wp_a=\tau^{-\frac12\np_a}\zp_a,\qquad\wmp_a=\tau^{-\frac12\np_a}\zmp_a,\qqand\wu_a=\tau^{-\frac12\nu_a}\zu_a.\label{eq:defzs}
\end{equation}
(Note that, in view of \eqref{eq:defwmp} and \eqref{eq:defwu}, $\zp_a$ and $\zmp_a$
satisfy $\zp_a=1/\zmp_a$, while $\zu_a=\vsigma_a$). For convenience we also
introduce the following notation: given functions $g(z,n)$ and $h(z_1,z_2;n_1,n_2)$ we write
\begin{equation}
  \label{eq:tmcP}
  \tilde g(z,n)=g(\tau^{-\inv2n}z,n)\qqand\tilde h(z_1,z_2;n_1,n_2)=h(\tau^{-\inv2n_1}z_1,\tau^{-\inv2n_2}z_2;n_1,n_2).
\end{equation}
Under the above change of variables the contour $C_0^{\np_a}$ is mapped to a circle
$C_{0,\tau^{-\eta}}$ of radius $\tau^{-\eta}$ centered at 0. Using this and
\eqref{eq:simplif-mfh} in Proposition \ref{prop:x-limit}, together with \eqref{eq:xlimit}, \eqref{eq:stpt} and
\eqref{eq:tnuk-intro}, and multiplying by the factor $\prod_a\tau^{-\np_a/2}$ introduced
by the change of variables, we obtain the following:

\begin{thm}\label{thm:flatSeries}
For every $\bk\in\zz_{\geq0}$ we have
\[\ee^{\rm flat}(\tau^{\frac12\bk h(t,0)})=\bk\taufac\sum_{k=0}^{\bk}\nuflat(t)\]
with
\[\nuflat(t)=\sum_{\substack{\ku,\kp\geq0\\\ku+2\kp=k}}\inv{\ku!2^{\kp}\kp!}\mfZflat(\ku,\kp),\]
where
\begin{equation}\label{eq:mfZflat}
\begin{split}
&\mfZflat(\ku,\kp)=\quad\quad\smashoperator{\sum_{(\vec{\vsigma},\vnu,\vnp)\in\Lambda^{\bk}_{\ku,\kp}}}\quad\quad
\itwopii{\kp}\int_{C^{\kp}_{0,1}}d\vzp\,\mcP\tilde\mff_1(\vzu,\vzp;\vnu,\vnp)\mcP\tmfgu(\vzu;\vnu)\mcP\tmfgp(\vzp;\vnp)\\
&\hspace{0.4in}\times
\prod_{1\leq a<b\leq\kp}\tau^{-2\np_a\np_b}\tilde\mfe(\zp_a,\zp_b;\np_a,\np_b)\tilde\mfe(\zmp_a,\zp_b;\np_a,\np_b)
\prod_{a\leq\ku,\,b\leq\kp}\tau^{-\nu_a\np_b}\tilde\mfe(\zu_a,\zp_b;\nu_a,\np_b)\\
&\hspace{0.4in}\times\mcP\tilde\mfh(\vzu;\vnu)\prod_{a=1}^{\ku}\tfrac12\tau^{-\inv2\nu_a}\zu_a
\prod_{a=1}^{\kp}(-1)^{\np_a}\tau^{-\inv2(\np_a)^2}(\tau^{-\np_a}-1)\zmp_a\,\tilde\mfh_2(\zp_a,\zmp_a;\np_a,\np_a),
\end{split}
\end{equation}
$C_{0,1}$ is a circle of radius 1 centered at the origin, $\zmp_a=1/\zp_a$ and $\zu_a=\vsigma_a$.
\end{thm}

In particular, this proves Theorem \ref{thm:flat-moments}. Note that in \eqref{eq:mfZflat}
we have taken the $\zp_a$ contours to lie at circles of radius $1$ instead of
$\tau^{-\eta}$ as specified by the change of variables \eqref{eq:defzs}. We may do this at
this point because, although both $\tmfgp(\zp_a,\np_a)$ and $\tmfgp(\zmp_a,n_a)$ have
simple poles at $\zp_a=\pm1$, the singularities are cancelled by the double zeros of
$\tilde\mfh_2(\zp_a,\zmp_a;\np_a,\np_a)$ at $\zp_a=1/\zp_a=\pm1$.

\section{Moment formula: Pfaffian structure}\label{sec:pfaffian}

In this section we will prove Theorem \ref{thm:fredPf-intro}, starting from the formula
given in Theorem \ref{thm:flatSeries}. The structure of the proof is similar to the one
provided in \cite{cal-led} for the case of the delta Bose gas.

In order to unveil the Pfaffian structure lying behind the formula in Theorem
\ref{thm:flatSeries} it will be convenient to consider first the case when $\ku$ is even.

\subsection{Case $\ku$ even}\label{sec:kueven}

Define the following variables:
\begin{gather}\label{eq:defyvars}
y_{2a-1}=\zp_a,\qquad y_{2a}=1/\zp_a\qquad y_{2\kp+b}=\vsigma_b,\\
m_{2a-1}=m_{2a}=\np_a,\qqand m_{2\kp+b}=\nu_b
\end{gather}
for $1\leq a\leq\kp$ and $1\leq b\leq\ku$. We will see next how the pairing structure, and
in particular the simplification presented in Section \ref{sec:simplif}, yields formulas
which can be naturally turned into a Pfaffian.

A computation shows that
\begin{equation}
  \prod_{1\leq a<b\leq\kp}\tilde\mfe(\zp_a,\zp_b;\np_a,\np_b)
   \tilde\mfe(\zmp_a,\zp_b;\np_a,\np_b)
   =\prod_{1\leq a<b\leq2\kp}\frac{\tau^{\inv2m_{b}}y_{b}-\tau^{\inv2m_{a}}y_{a}}{\tau^{\inv2m_a+\inv2m_b}y_{a}y_{b}-1}
   \label{eq:pairedPf}
\end{equation}
and, similarly,
\begin{equation}
\prod_{a\leq\ku,b\leq\kp}\tilde\mfe(\zu_a,\zp_b;\nu_a,\np_b)
=\prod_{a\leq2\kp,b\leq\ku}\frac{\tau^{\inv2m_{b}}y_{b}-\tau^{\inv2m_{2\kp+a}}y_{2\kp+a}}
  {\tau^{\inv2m_a+\inv2m_{2\kp+b}}y_{a}y_{2\kp+b}-1}.\label{eq:crossedPf}
\end{equation}
In order to complete the nonlinear Schur Pfaffian \eqref{eq:schur-nl-pf-intro} in
$y_1,\dotsc,y_{2\kp+\ku}$ we need to multiply by some missing factors, yielding
\begin{multline}
  \label{eq:preSchurPf}
   ({\rm L.H.S.}\eqref{eq:pairedPf})\times({\rm L.H.S.}\eqref{eq:crossedPf})
   =\qquad\quad\smashoperator{\prod_{2\kp+1\leq a<b\leq 2\kp+\ku}}\qquad\quad\frac{\tau^{\inv2m_a+\inv2m_b}y_{a}y_{b}-1}{\tau^{\inv2m_{b}}y_{b}-\tau^{\inv2m_{a}}y_{a}}\\
   \times\prod_{a=1}^{\kp}\frac{\tau^{\inv2m_{2a-1}+\inv2m_{2a}}y_{2a-1}y_{2a}-1}{\tau^{\inv2m_{2a}}y_{2a}-\tau^{\inv2m_{2a-1}}y_{2a-1}}
   \pf\!\left[\frac{\tau^{\inv2m_b}y_b-\tau^{\inv2m_a}y_a}{\tau^{\inv2(m_a+m_b)}y_ay_b-1}\right]_{a,b=1}^{2\kp+\ku}.
\end{multline}
Now consider the factor $\mcP\tilde\mfh(\vzu;\vnu)$ on the right hand side of
\eqref{eq:mfZflat}. Since $\zu_a=\vsigma_a$, \eqref{eq:defht} and Lemma
\ref{lem:mfeu} yield
$\tilde\mfh(\zu_a,\zu_b;\nu_a,\nu_b)=(-\vsigma_a\vsigma_b)^{\nu_a\wedge\nu_b}\frac{|\vsigma_b\tau^{\inv2\nu_b}-\vsigma_a\tau^{\inv2\nu_a}|}{1-\vsigma_a\vsigma_b\tau^{\inv2(\nu_a+\nu_b)}}$. Using
this and \eqref{eq:preSchurPf} together with the facts that $\vsigma_a=y_{2\kp+a}$,
$\frac{|\vsigma_b\tau^{\frac12\nu_b}-\vsigma_a\tau^{\frac12\nu_a}|}{\vsigma_a\tau^{\frac12\nu_a}-\vsigma_b\tau^{\frac12\nu_b}}=\vsigma_a\vsigma_b\sgn(\vsigma_b\nu_b-\vsigma_a\nu_a)$ and $\frac{\tau^{\inv2m_{2a-1}+\inv2m_{2a}}y_{2a-1}y_{2a}-1}{\tau^{\inv2m_{2a}}y_{2a}-\tau^{\inv2m_{2a-1}}y_{2a-1}}
  =\frac{\zp_a(\tau^{\np_a}-1)}{\tau^{\inv2\np_a}(1-(\zp_a)^2)}$ yields
\begin{multline}
  \mcP\tilde\mfh(\vzu;\vnu)\times({\rm L.H.S.}\eqref{eq:pairedPf})\times({\rm L.H.S.}\eqref{eq:crossedPf})
  =\prod_{1\leq a<b\leq\ku}(-\vsigma_a\vsigma_b)^{\nu_a\wedge\nu_b+1}\sgn(\vsigma_a\nu_a-\vsigma_b\nu_b)\\
  \times\prod_{a=1}^{\kp}\frac{\zp_a(\tau^{\np_a}-1)}{\tau^{\inv2\np_a}(1-(\zp_a)^2)}
  \pf\!\left[\frac{\tau^{\inv2m_b}y_b-\tau^{\inv2m_a}y_a}{\tau^{\inv2(m_a+m_b)}y_ay_b-1}\right]_{a,b=1}^{2\kp+\ku}.\label{eq:pfaffian1}
\end{multline}

It remains to recognize the cross-product on the right hand side of \eqref{eq:pfaffian1}
as a Pfaffian. As will be clear in the proof of the next lemma, the factors in that product can be thought
of as the entries of a degenerate Schur matrix. In the case $\vsigma_a\equiv 1$, this
identity was discovered and checked for small values of $k$ using Mathematica in \cite{cal-led}. 
\begin{lem}\label{lem:intpfaffian}
For  $k\in \{2,4,6,\ldots\}$, positive integers $m_1,\dotsc,m_k$, and  $\vsigma_1,\dotsc\vsigma_k\in\{-1,1\}$,
  \begin{align}\prod_{1\leq a<b\leq k}&(-\vsigma_a\vsigma_b)^{m_a\wedge
    m_b+1}\sgn(\vsigma_am_a-\vsigma_bm_b)
  =\pf\!\left[(-\vsigma_a\vsigma_b)^{m_a\wedge
      m_b}\sgn(\vsigma_bm_b-\vsigma_am_a)\right]_{a,b=1}^k.\end{align} 
\end{lem}
\begin{proof} If any $\vsigma_am_a=\vsigma_bm_b$, both sides vanish, as $\sgn(0)=0$, so we
  can assume they are all distinct. Switching $(\vsigma_a,m_a)$ with a different
  $(\vsigma_b,m_b)$ induces a sign change on both sides of the identity, hence by making
  finitely many such switches, we can assume that $\vsigma_i=1$, $i=1,\ldots, n$, and
  $\vsigma_i=-1$,
  $i=n+1,\ldots,k$.  Our matrix then has the block form $\left[\begin{smallmatrix} A & D \\
      -D^{\sf T} & B \end{smallmatrix}\right]$ where $A$ is $k\times k$, $B$ is
  $(k-n)\times(k-n)$ and the $k\times(n-k)$ matrix $D$ has all entries $D_{ab}=1$.  Since
  $D$ is rank one, $\pf\!\left[\begin{smallmatrix} A & D \\ -D^{\sf T} &
      B \end{smallmatrix}\right] =\pf\!\left[A\right]\pf\!\left[B\right]$ by
  \eqref{eq:pfaffianblockform} below.  Choose $p(h,m)$ to be a function of $h\in \rr$,
  $m\in \zz$ satisfying $p(h,m)\to \infty $ as $h\to 0$ if $m$ is even, and $p(h,m)\to 0 $
  as $h\to 0$ if $m$ is odd.  Let $A_{a,b}(\vec{h},\vec{m}) = \frac{ p(h_b,m_b) -
    p(h_a,m_a) }{ p(h_b,m_b) + p(h_a,m_a) }$ for $a,b=1,\ldots,n$ and
  $B_{a,b}(\vec{h},\vec{m}) = \frac{ p(h_b,m_b) - p(h_a,m_a) }{ p(h_b,m_b) + p(h_a,m_a) }$
  for $a,b=n+1,\ldots,k$.  They have the property that if $h_a\to 0$ in increasing order
  of $m_a$ (i.e. if $m_b>m_a$, then $h_a\to 0$ followed by $h_b\to 0$) for $a,b$ with
  $\vsigma_a=\vsigma_b=1$ (i.e. $1\le a,b\le n$) and in decreasing order of $m_a$ for $
  \vsigma_a=\vsigma_b=-1$ (i.e. $n+1\le a\le k$) then both $A_{a,b}(\vec{h},\vec{m})$ and
  $B_{a,b}(\vec{h},\vec{m})$ converge to $(-\vsigma_a\vsigma_b)^{m_a\wedge
    m_b+1}\sgn(\vsigma_am_a-\vsigma_bm_b)$.  By the Schur Pfaffian formula \eqref{eq:schur-pf-intro},
  we have $\pf\!\big[A(\vec{h},\vec{m})\big]\pf\!\big[B(\vec{h},\vec{m})\big] = \prod_{1\le
    a<b\le n} A_{a,b}(\vec{h},\vec{m})\prod_{n+1\le a<b\le k} B_{a,b}(\vec{h},\vec{m})$.
  Taking the limit in the given order yields the result.
\end{proof}

Using this result, \eqref{eq:pfaffian1} and the identity
$\tilde\mfh_2(\zp_a,\zmp_a;\np_a,\np_a)=\frac{(1-(\zp_a)^2)^2}{(\tau^{\np_a}-(\zp_a)^2)(\tau^{-\np_a}-(\zp_a)^2)}$
in \eqref{eq:mfZflat} yields (recalling that $\vsigma_a=y_{2\kp+a}$ and $\zp_a=1/\zmp_a$)

\begin{equation}
\begin{split}
\mfZflat(&\ku,\kp)=\quad\quad\,\,\,
\smashoperator{\sum_{\substack{\vec{\vsigma}\in\{-1,1\}^{\ku},\,\vnu\in\zz_{\geq1}^{\ku},\,\vzp\in\zz^{\kp}_{\geq1}\\\sum\nu_a+2\sum\np_a=\bk}}}\qquad\qquad
\itwopii{\kp}\int_{C^{\kp}_{0,1}}\!\!\!d\vzp\,\,\,\,\smashoperator{\prod_{1\leq
    a<b\leq\kp}}\,\,\,\tau^{-2\np_a\np_b}
\,\,\,\smashoperator{\prod_{a\leq\ku,\,b\leq\kp}}\,\,\,\tau^{-\nu_a\np_b}
\,\,\,\smashoperator{\prod_{1\leq a<b\leq\ku}}\,\,\,\tau^{-\inv2\nu_a\nu_b}\\
&\times\mcP\tilde\mff_1(\vzu,\vzp;\vnu,\vnp)\mcP\tmfgu(\vzu;\vnu)\mcP\tmfgp(\vzp;\vnp)
\prod_{a=1}^{\ku}\tfrac12\tau^{-\inv2\nu_a}\zu_a\\
   &\times\prod_{a=1}^{\kp}(-1)^{\np_a}\tau^{-\inv2\np_a(\np_a+1)}
   \frac{\tau^{-\np_a}(1-(\zp_a)^2)^2}{(\tau^{\np_a}-(\zp_a)^2)(\tau^{-\np_a}-(\zp_a)^2)}
   \frac{(1-\tau^{\np_a})^2}{((\zp_a)^2-1)}\\
   &\times\pf\!\left[\frac{\tau^{\frac12m_b}y_b-\tau^{\frac12m_a}y_a}{\tau^{\frac12(m_a+m_b)}y_ay_b-1}\right]_{a,b=1}^{2\kp+\ku}
   \pf\!\left[(-y_{a}y_{b})^{m_a\wedge m_b}\sgn(y_{b}m_b-y_{a}m_a)\right]_{a,b=2\kp+1}^{2\kp+\ku}.
\end{split}\label{eq:Z60s}
\end{equation}
Note that we have removed the restriction that the $\vsigma_a\nu_a$'s be different at this
point. We may do this thanks to the second Pfaffian on the right hand side. In fact, by
Lemma \ref{lem:intpfaffian} this Pfaffian contains a factor
$\prod_{a<b}\sgn(\sigma_bm_b-\sigma_am_a)$, which vanishes if a pair of $\vsigma_a\nu_a$'s
are equal. Now note that
\[\tmfgp(z,n)\tmfgp(1/z,n)\tfrac{\tau^{-n}(1-z^2)}{(\tau^{n}-z^2)(\tau^{-n}-z^2)}=
\tfrac{\pochinf{\!-\tau^{-n/2}z;\tau}\pochinf{\!-\tau^{-n/2}/z;\tau}}{\pochinf{\!-\tau^{n/2}z;\tau}\pochinf{\!-\tau^{n/2}/z;\tau}}
\tfrac{\pochinf{\tau^{1+n}z^2;\tau}\pochinf{\tau^{1+n}/z^2;\tau}}{\pochinf{\tau z^2;\tau}\pochinf{\tau/z^2;\tau}}.\]
With this in mind we introduce the notation 
\begin{equation}
\begin{gathered}
  \mfu(z,n)=(1-\tau^n)\mff_1(\tau^{-\inv2n}z;n)
  \frac{\pochinf{\!-\tau^{-n/2}z;\tau}\pochinf{\tau^{1+n}z^2;\tau}}{\pochinf{\!-\tau^{n/2}z;\tau}\pochinf{\tau z^2;\tau}},\\
  \mfuu(z,n)=\tau^{-\inv2n}z\frac{1-\tau^nz^2}{1-\tau^n},\qquad\quad
  \mfup(z,n)=(-1)^{n}\tau^{-\inv2n}\frac{1}{z^2-1},\\
  \mfuap(z,n)=(-1)^{n}\tau^{-\inv2n}\frac{1+z^2}{z^2-1},
\end{gathered}\label{eq:def-mfu}
\end{equation}
chosen so that
\begin{equation}
(\text{2$^{\rm nd}$ line \eqref{eq:Z60s}})\times(\text{3$^{\rm rd}$ line \eqref{eq:Z60s}})
=\frac1{2^{\ku}}\mcP\mfu(\vzp,\vzu;\vnp,\vnu)\mcP\mfuu(\vzu;\vnu)\prod_{a=1}^{\kp}\mfup(\zp_a,\np_a)\tau^{-\inv2(\np_a)^2}\label{eq:mfuRepl}
\end{equation}
($\mfuap$ is a certain antisymmetrization of $\mfup$, see \eqref{eq:antisymm-mfu2}, and
will be used later on).  Of course, this decomposition is not
uni\-que, but it will turn out to be convenient below.

Observe that, in view of \eqref{eq:defyvars}, all the $y_a$'s live in $C_{0,1}$.
Therefore we may replace the integration in $\vec{z}^{\rm p}$ over $C_{0,1}^{\kp}$ by an
integration in $\vec{y}$ over $C_{0,1}^{2\kp+\ku}$ by introducing suitable delta
functions.  We may similarly sum over $m_1,\dotsc,m_{2\kp+\ku}\in\zz_{\geq1}$ and get rid
of the sum over $\vsigma_1,\dotsc,\vsigma_{\ku}\in\{-1,1\}$. The precise replacement we
are going to use, keeping in mind \eqref{eq:defyvars}, can be expressed as follows: for a
suitable function $f$,
\begin{multline}
  \qquad\qquad
  \smashoperator{\sum_{\substack{\vec{\vsigma}\in\{-1,1\}^{\ku},\,\vnu\in\zz_{\geq1}^{\ku},\,\vzp\in\zz^{\kp}_{\geq1}\\\sum\nu_a+2\sum\np_a=\bk}}}\qquad\qquad
  \itwopii{\kp}\int_{C^{\kp}_{0,1}}d\vzp\,f(\vzp;\vzmp;\vnp;\vnp;\vnu;\vec\vsigma)\\
  =\qquad\smashoperator{\sum_{\substack{m_1,\dotsc,m_{2\kp+\ku}=1,\\m_1+\dotsm+m_{2\kp+\ku}=\bk}}^\infty^\infty}\qquad\,\,\,\itwopii{2\kp+\ku}\int_{C^{2\kp+\ku}_1}d\vec{y}\,
\prod_{a=1}^{\kp}\delta_{y_{2a-1}-\frac1{y_{2a}}}\uno{m_{2a-1}=m_{2a}}\prod_{a=2\kp+a}^{2\kp+\ku}(\delta_{y_a-1}+\delta_{y_a+1})\\
 \times f(y_1,y_3,\dotsc,y_{2\kp-1};y_2,y_4,\dotsc,y_{2\kp};m_1,m_3,\dotsc,m_{2\kp-1};\\
 m_2,m_4,\dotsc,m_{2\kp};m_{2\kp+1},\dotsc,m_{2\kp+\ku};
 y_{2\kp+1},\dotsc,y_{2\kp+\ku}).\label{eq:deltas-repl}
\end{multline}
The proof of this (essentially algebraic) identity is straightforward. Here, in view of the fact that
we are working with contour integrals in the complex plane, our delta functions implicitly
carry a factor of $2\pi\I$ (which is omitted for notational convenience).

\begin{rem}
  The use of delta functions here will should be regarded essentially as a purely
  notational device which will allow us to express our integral kernels in a more
  convenient way. In most parts of the argument it will not cause any difficulty, with the
  exception of the argument following \eqref{eq:pvPfaffPre} which is treated in pages
  \pageref{eq:pvPfaffPre}-\pageref{eq:def-mfv-pre}. In particular, the delta functions
  will be gone by the time we get to our main result, Theorem \ref{thm:fredPf-intro}. The
  reader who feels uncomfortable with this treatment can replace $\delta_z$ by a symmetric
  mollifier $\varphi_\ep$ (so that $\varphi_\ep(x)\longrightarrow\delta_x$ as $\ep\to0$ in
  the sense of distributions), and replace the right hand side of \eqref{eq:deltas-repl}
  by
  \begin{multline}
    \qquad\smashoperator{\sum_{\substack{m_1,\dotsc,m_{2\kp+\ku}=1,\\m_1+\dotsm+m_{2\kp+\ku}=\bk}}^\infty^\infty}
    \qquad\,\,\,\itwopii{2\kp+\ku}\int_{C^{2\kp+\ku}_{0,1}}d\vec{y}\,f(\star)
    \\\times\prod_{a=1}^{\kp}\varphi_{\ep_a}(y_{2a-1}-1/y_{2a})\uno{m_{2a-1}=m_{2a}}
    \prod_{a=2\kp+a}^{2\kp+\ku}(\varphi_{\ep_a}(y_a-1)+\varphi_{\ep_a}(y_a+1))
    ,\label{eq:deltas-repl2}
  \end{multline}
  where $f(\star)$ stands for the factor (with the same arguments) appearing on the last
  two lines of \eqref{eq:deltas-repl}. This is what we do in pages
  \pageref{eq:pvPfaffPre}-\pageref{eq:def-mfv-pre}.
\end{rem}

Using this idea, \eqref{eq:mfuRepl}, and noting that by the definition \eqref{eq:defyvars} of the $m_a$'s in
terms of the $\nu_a$'s and $\np_a$'s
\[\prod_{a=1}^{\kp}\tau^{-\inv2(\np_a)^2}\prod_{1\leq
  a<b\leq\kp}\tau^{-2\np_a\np_b}\prod_{a\leq\ku,\,b\leq\kp}\tau^{-\nu_a\np_b} \prod_{1\leq
  a<b\leq\ku}\tau^{-\inv2\nu_a\nu_b}=\prod_{1\leq a<b\leq2\kp+\ku}\tau^{-\inv2m_am_b},\]
\eqref{eq:Z60s} becomes\footnote{There is an apparent singularity in this formula, because
  $\mfup(y_a,n_a)$ has simple poles at $y_a=\pm1$ and $y_a$ lives in a circle of radius
  1. This is resolved by noting that, in view of \eqref{eq:schur-nl-pf-intro}, the first
  Pfaffian on the third line has a factor
  $(\tau^{m_{2a}/2}y_{2a}-\tau^{m_{2a-1}/2}y_{2a-1})/(\tau^{(m_{2a-1}+m_{2a})/2}y_{2a-1}y_{2a}-1)$,
  which has zeros at these points when $m_{2a-1}=m_{2a}$ and $y_{2a}=1/y_{2a-1}$. In later
  formulas it will be less apparent how these singularities are canceled (see
  e.g. Proposition \ref{prop:firstPfaffian}), but we know that this will be the case by
  virtue of their equivalence to \eqref{eq:Zp60s}.\label{ftn:sing}}
\begin{equation}
\begin{aligned}
  &\mfZflat(\ku,\kp)=\inv{2^{\ku}}\qquad\smashoperator{\sum_{\substack{m_1,\dotsc,m_{2\kp+\ku}=1,\\m_1+\dotsm+m_{2\kp+\ku}=\bk}}^\infty}
\qquad\,\,\,\itwopii{2\kp+\ku}\int_{C^{2\kp+\ku}_{0,1}}d\vec{y}\,\prod_{1\leq a<b\leq2\kp+\ku}\tau^{-\inv2m_am_b}\\
  &\times\prod_{a=1}^{2\kp+\ku}\mfu(y_a,m_a)\prod_{a=2\kp+1}^{2\kp+\ku}\mfuu(y_a,m_a)(\delta_{y_a-1}+\delta_{y_a+1})
  \prod_{a=1}^{\kp}\mfup(y_a,m_a)\delta_{y_{2a-1}-\inv{y_{2a}}}\uno{m_{2a-1}=m_{2a}}\\
  &\times\pf\left[\frac{\tau^{\frac12m_b}y_b-\tau^{\frac12m_a}y_a}{\tau^{\frac12(m_a+m_b)}y_ay_b-1}\right]_{a,b=1}^{2\kp+\ku}
  \pf\left[(-y_ay_b)^{m_a\wedge m_b}\sgn(y_bm_b-y_am_a)\right]_{a,b=2\kp+1}^{2\kp+\ku}.
\end{aligned}\label{eq:Zp60s}
\end{equation}

Now observe that the integration measure
$\sum_{m_1+\dotsm+m_{2\kp+\ku}=\bk}^\infty\int_{C^{2\kp+\ku}_{0,1}}d\vec y$ is symmetric under
exchanging $m_a$'s and $y_a$'s. Let $\pi_{a}\in S_{2\kp}$ denote the transposition of
elements $2a-1$ and $2a$ and let $T_{2\kp}$ be the subgroup of $S_{2\kp}$ generated by
$\{\pi_{a},\,a=1,\dotsc,\kp\}$, which has order $2^{\kp}$. Noting every factor in
\eqref{eq:Zp60s} is invariant under the action of $T_{2\kp}$ on the
$2\kp$-tuple $((m_1,y_1),\dotsc,(m_{2\kp},y_{2\kp}))$ except for the first of the two
Pfaffians, which is antisymmetric under this action, we get
\begin{equation}
\begin{split}
  &\mfZflat(\ku,\kp)=\inv{2^{\ku}}\qquad\smashoperator{\sum_{\substack{m_1,\dotsc,m_{2\kp+\ku}=1,\\m_1+\dotsm+m_{2\kp+\ku}=\bk}}^\infty}
\qquad\,\,\,\itwopii{2\kp+\ku}
  \int_{C^{2\kp+\ku}_{0,1}}d\vec y\,\prod_{1\leq a<b\leq2\kp+\ku}\tau^{-\inv2m_am_b}\\
  &\,\,\,\times\prod_{a=1}^{2\kp+\ku}\mfu(y_a,m_a)\prod_{a=2\kp+1}^{2\kp+\ku}\mfuu(y_a,m_a)(\delta_{y_a-1}+\delta_{y_a+1})
  \\
  &\,\,\,\times\pf\left[(-y_ay_b)^{m_a\wedge
      m_b}\sgn(y_bm_b-y_am_a)\right]_{a,b=2\kp+1}^{2\kp+\ku}
  \pf\left[\frac{\tau^{\frac12m_b}y_b-\tau^{\frac12m_a}y_a}{\tau^{\frac12(m_a+m_b)}y_ay_b-1}\right]_{a,b=1}^{2\kp+\ku}\\
  &\,\,\,\times\frac1{2^{\kp}}\sum_{\pi\in T_{2\kp}}\sgn(\pi)\prod_{a=1}^{\kp}\mfup(y_{\pi(2a-1)},m_{\pi(2a-1)})
  \delta_{y_{\pi(2a-1)}-\inv{y_{\pi(2a)}}}\uno{m_{\pi(2a-1)}=m_{\pi(2a)}}.
\end{split}\label{eq:T2kp}
\end{equation}
Observe that the product of the delta and indicator functions in the last line
is invariant under exchanging any pair of indices of the form $(2a-1,2a)$,
$a=1,\dots,\kp$. On the other hand, we have
\begin{equation}
\mfup(y,m)-\mfup(1/y,m)=(-1)^{m}\tau^{-\inv2m}\frac{1+y^2}{1-y^2}=\mfuap(y,m),\label{eq:antisymm-mfu2}
\end{equation}
where $\mfuap$ was defined in \eqref{eq:def-mfu}, and thus it is not hard to see that
\[
  (\text{last line \eqref{eq:T2kp}})=
  \inv{2^{\kp}}\prod_{a=1}^{\kp}\mfuap(y_{2a-1},m_{2a-1})\delta_{y_{2a-1}-1/y_{2a}}\uno{m_{2a-1}=m_{2a}}.\]
Using this and repeating the procedure we just used of using the symmetry of the
integration measure to sum over permutations, the last formula for $\mfZflat(\ku,\kp)$ can
be rewritten using the replacement
\begin{equation}
  (\text{last line \eqref{eq:T2kp}})
  \leadsto\sum_{\sigma\in S_{2\kp}}\tfrac{\sgn(\sigma)}{2^{\kp}(2\kp)!}
  \prod_{a=1}^{\kp}\mfuap(y_{\sigma(2a-1)},m_{\sigma(2a-1)})
  \delta_{y_{\sigma(2a-1)}-\inv{y_{\sigma(2a)}}}\uno{m_{\sigma(2a-1)}=m_{\sigma(2a)}}.
\end{equation}
Since, by \eqref{eq:antisymm-mfu2}, the matrix
$\big[\mfuap(y_a,y_b;m_a,m_b)\uno{m_{a}=m_{b}}\delta_{y_{a}-1/y_{b}}\uno{a\neq b}\big]_{a,b=1}^{2\kp}$
is skew-symmetric, this sum equals $\frac{\kp!}{(2\kp)!}$ times the Pfaffian of
this matrix. Thus, using in addition the identity
$\prod_{a=1}^{2n}\lambda_a\pf\!\left[A(a,b)\right]_{a,b=1}^{2n}=
\pf\!\left[\lambda_a\lambda_bA(a,b)\right]_{a,b=1}^{2n}$ (see \eqref{eq:diag2Pf}), 
the last sum can be written as
\[\inv{2^{2\kp}}\frac{\kp!}{(2\kp)!}
\pf\!\left[4\mfuap(y_a,m_a)\uno{m_{a}=m_{b}}\delta_{y_{a}-\inv{y_{b}}}\uno{a\neq b}\right]_{a,b=1}^{2\kp}\]
(the introduction of the factor $2^{-2\kp}$ in front of the Pfaffian is for later
convenience), and now using the same Pfaffian identity again we deduce the following

\begin{prop}\label{prop:Zp70s}
  Assume that $\ku$ is even. Then
  \begin{equation}
    \begin{aligned}
      &\mfZflat(\ku,\kp)=\frac{\kp!}{2^{\ku+2\kp}(2\kp)!}\qquad\,\,\smashoperator{\sum_{\substack{m_1,\dotsc,m_{2\kp+\ku}=1,\\m_1+\dotsm+m_{2\kp+\ku}=\bk}}^\infty}
      \qquad\,\,\,\itwopii{2\kp+\ku}\int_{C^{2\kp+\ku}_{0,1}}d\vec y\,\prod_{1\leq a<b\leq2\kp+\ku}\tau^{-\inv2m_am_b}\\
      &\times\prod_{a=1}^{2\kp+\ku}\mfu(y_a,m_a)
      \pf\left[\frac{\tau^{\frac12m_b}y_b-\tau^{\frac12m_a}y_a}{\tau^{\frac12(m_a+m_b)}y_ay_b-1}\right]_{a,b=1}^{2\kp+\ku}
      \pf\!\left[4\mfuap(y_a,m_a)\uno{m_{a}=m_{b}}\delta_{y_{a}-\inv{y_{b}}}\uno{a\neq b}\right]_{a,b=1}^{2\kp}\\
      &\times\pf\big[(-y_ay_b)^{m_a\wedge m_b}\sgn(y_bm_b-y_am_a)\mfuu(y_a,m_a)\mfuu(y_b,m_b)\\
      &\hspace{2.6in}\times(\delta_{y_a-1}+\delta_{y_a+1})(\delta_{y_b-1}+\delta_{y_b+1})\big]_{a,b=2\kp+1}^{2\kp+\ku}.
    \end{aligned}\label{eq:Zp70s}
  \end{equation}
\end{prop}

Note that the product of delta functions in the last line of \eqref{eq:Zp70s} poses no
difficulty when $a\neq b$ (since they involve different variables), while for $a=b$ the
prefactor $\sgn(y_bm_b-y_am_a)$ vanishes and thus the whole factor is interpreted as 0.

Recalling now the definition of
$\nuflat(t)$ in Theorem \ref{thm:flatSeries} we have
\begin{equation}
\nuflat(t)=\sum_{\ku,\kp\geq0,\,\ku+2\kp=k}\inv{2^k(2\kp)!\ku!}(\star),\label{eq:nunuflat}
\end{equation}
where $(\star)$ stands for the expression starting from the first sum in
\eqref{eq:Zp70s}. Note that in this expression the only factor which depends directly on
$\ku$ or $\kp$ (as opposed to $\ku+2\kp=k$) is the second Pfaffian in
\eqref{eq:Zp70s}. Remarkably, this Pfaffian can be combined with the third Pfaffian in
\eqref{eq:Zp70s} and with the sum over $\ku$ and $\kp$ in \eqref{eq:nunuflat}, thus
yielding an expression which only depends on $k$ (since, furthermore, the combinatorial
prefactors also combine into something which only depends on $k$). This is done using Lemma
\ref{lem:resumpf}, and the result is the following: for $\bk\geq k\geq0$, $k$ even,
\begin{equation}
\label{eq:finalPfaffianEvenFormula}
\begin{split}
  &\nuflat(t)=\frac{(-1)^{\frac12k}}{2^{k}k!}\!\!
  \sum_{\substack{m_1,\dotsc,m_{k}=1,\\m_1+\dotsm+m_k=\bk}}^\infty
  \sum_{y_1,\dotsc,y_{k}\in\{-1,1\}}
  \itwopii{k}\int_{C^{k}_{0,1}}d\vec y\quad\,\,\,\,\smashoperator{\prod_{1\leq
      a<b\leq2\kp+\ku}}\quad\,
  \tau^{-\inv2m_am_b}\prod_{a=1}^{k}\mfu(y_a,m_a)\\
  &~~\times\pf\Big[(-y_ay_b)^{m_a\wedge m_b}\sgn(y_bm_b-y_am_a)
    (\delta_{y_a-1}+\delta_{y_a+1})(\delta_{y_b-1}+\delta_{y_b+1})\mfuu(y_a,m_a)\mfuu(y_b,m_b)\\
  &~\hspace{1.4in}+4\mfuap(y_a,m_a)\uno{m_{a}=m_{b}}\delta_{y_{a}-\inv{y_{b}}}\uno{a\neq b}\Big]_{a,b=1}^{k}
  \pf\left[\frac{\tau^{\frac12m_a}y_a-\tau^{\frac12m_b}y_b}{\tau^{\frac12(m_a+m_b)}y_ay_b-1}\right]_{a,b=1}^{k}.
\end{split}
\end{equation}
Notice that in the last Pfaffian we have flipped the sign of the argument,
which yields the additional factor $(-1)^{k/2}$ in front. This will be convenient later.

\subsection{Extension to general $\ku$}\label{sec:genku}

To handle the case when $\ku$ is odd we use a standard trick (which goes back at least to
\cite{deBruijn}) to extend Pfaffians to matrices of odd size (the same idea is used
in \cite{cal-led}). Consider the dummy variables
\[y_{2\kp+\ku+1}=1,\qqand m_{2\kp+\ku+1}=0.\]
A simple computation (together with \eqref{eq:schur-nl-pf-intro}) shows that
\[\prod_{1\leq a<b\leq2\kp+\ku}\frac{y_b-y_a}{y_ay_b-1}=
\prod_{1\leq a<b\leq2\kp+\ku+1}\frac{y_b-y_a}{y_ay_b-1}=\pf\!\left[\frac{y_b-y_a}{y_ay_b-1}\right]_{a,b=1}^{2\kp+\ku+1}\]
and, similarly (using Lemma \ref{lem:intpfaffian}), that
\begin{multline}
  \prod_{2\kp+1\leq a<b\leq2\kp+\ku} (-\vsigma_a\vsigma_b)^{m_a\wedge
  m_b+1}\sgn(\vsigma_am_a-\vsigma_bm_b)\\
 =\pf\!\left[(-\vsigma_a\vsigma_b)^{m_a\wedge m_b}\sgn(\vsigma_bm_b-\vsigma_am_a)\right]_{a,b=2\kp+1}^{2\kp+\ku+1}.
\end{multline}
These formulas can be used directly as replacements for \eqref{eq:pfaffian1} and Lemma
\ref{lem:intpfaffian}, and lead to a formula which reads exactly like \eqref{eq:Zp60s},
except that the matrices appearing inside the two Pfaffians are augmented in this
fashion. This can then be extended directly to \eqref{eq:Zp70s} and further to a version of
\eqref{eq:finalPfaffianEvenFormula}. We leave the details to the reader and simply record
the result: for odd $\ku$ we have
\begin{multline}
\label{eq:finalPfaffianOddFormula}
  \nuflat(t)=\frac{(-1)^{\frac12(k+1)}}{2^{k}k!}\sum_{\substack{m_1,\dotsc,m_{k}=1,\\m_1+\dotsm+m_k=\bk}}^\infty\sum_{y_1,\dotsc,y_{k}\in\{-1,1\}}
  \itwopii{k}\int_{C^{k}_{0,1}}d\vec y\,\,\,\quad\smashoperator{\prod_{1\leq
      a<b\leq2\kp+\ku}}\quad\,\,
  \tau^{-\inv2m_am_b}\\
  \times\prod_{a=1}^{k}\mfu(y_a,m_a)\pf\!\left[\begin{smallmatrix}A & U\\-U^{\sf T} & 0\end{smallmatrix}\right]
  \pf\!\left[\begin{smallmatrix}B & \uno{}\\-\uno{}^{\sf T} & 0\end{smallmatrix}\right],
\end{multline}
where $A$ and $B$ are, respectively, the matrices appearing inside the first and second
Pfaffians in \eqref{eq:finalPfaffianEvenFormula}, $U$ is the vector with entries
$U_a=\mfuu(-1,m_a)\delta_{y_a+1}-\mfuu(1,m_a)\delta_{y_a-1}$ and $\uno{}$ represents a vector of ones of size $\ku$
(which is formally obtained from \eqref{eq:finalPfaffianEvenFormula} by augmenting the
matrices in the Pfaffians using the dummy variables and computing their last row and column).

Now  Lemma \ref{lem:pfaffianblockform} with $V= \uno{}$ together with the identity
$(-1)^{\frac12k(k+1)}=(-1)^{\frac12k}\uno{k\in2\zz}+(-1)^{\frac12(k+1)}\uno{k\in2\zz+1}$
allow us to write a version of \eqref{eq:finalPfaffianEvenFormula} and
\eqref{eq:finalPfaffianOddFormula} which is valid for $k$ even and odd:
  
\begin{prop}\label{prop:firstPfaffian}
  For any $0\leq k\leq\bk$ we have
  \begin{equation}
    \nuflat(t)=\frac{(-1)^{\frac12k(k+1)}}{2^{k}k!}\sum_{\substack{m_1,\dotsc,m_{k}=1,\\m_1+\dotsm+m_k=\bk}}^\infty
    \int_{C^{k}_{0,1}}\!d\vec y\,\prod_{1\leq
      a<b\leq2\kp+\ku}\tau^{-\inv2m_am_b}\prod_{a=1}^{k}\mfu(y_a,m_a)\pf\!\left[\Kmu(\vec
      y;\vec m)\right]\label{eq:genPfaffianFormula}
  \end{equation}
  where $\nuflat$ was defined in Proposition \ref{prop:x-limit} and
  \[\Kmu(\vec y;\vec m)=\!\left[\begin{matrix} \big[\Kmu_{1,1}(y_a,y_b;m_a,m_b)\big]_{a,b=1}^k &
        \big[\Kmu_{1,2}(y_a,y_b;m_a,m_b)\big]_{a,b=1}^k \\
        -\big[\Kmu_{1,2}(y_b,y_a;m_b,m_a)\big]_{a,b=1}^k &
        \big[\Kmu_{2,2}(y_a,y_b;m_a,m_b)\big]_{a,b=1}^k \end{matrix}\right]\]
with\footnote{In order to avoid an even heavier notation, we are slightly abusing it in the definition of
  $\Kmu_{1,1}(y_a,y_b;m_a,m_b)$, which depends on $a$ and $b$ explicitly in addition to
  $y_a$, $y_b$, $m_a$ and $m_b$. This will become irrelevant in later version of our
  formulas, see footnote \ref{ft:indic} in page \pageref{ft:indic}.}
\begin{equation}
  \begin{split}
    \Kmu_{1,1}(y_a,y_b;m_a,m_b)&=4\mfuap(y_a,m_a)\uno{m_{a}=m_{b}}\delta_{y_{a}-\inv{y_{b}}}
    \uno{a\neq b}\\
    &\hspace{-0.95in}+(-y_ay_b)^{m_a\wedge
      m_b}\sgn(y_bm_b-y_am_a)\mfuu(y_a,m_a)\mfuu(y_b,m_b)
    (\delta_{y_a-1}+\delta_{y_a+1})(\delta_{y_b-1}+\delta_{y_b+1}),\\
    \Kmu_{1,2}(y_a,y_b;m_a,m_b)&=\mfuu(-1,m_a)\delta_{y_a+1}-\mfuu(1,m_a)\delta_{y_a-1},\\
    \Kmu_{2,2}(y_a,y_b;m_a,m_b)&=\frac{\tau^{\frac12m_a}y_a-\tau^{\frac12m_b}y_b}{\tau^{\frac12(m_a+m_b)}y_ay_b-1},
  \end{split}\label{eq:Ks}
\end{equation}
where $\mfu$, $\mfuap$, and $\mfuu$ are given in \eqref{eq:def-mfu}.
\end{prop}

This result together with \eqref{eq:xlimit}, \eqref{eq:stpt} and Proposition
\ref{prop:x-limit} yield the following explicit formula for the exponential moments of
flat ASEP: for any $\bk\in\zz_{\geq0}$ we have, with $\Kmu$ as in Proposition \ref{prop:firstPfaffian},
  \begin{multline}\label{eq:flatASEPmoments}
    \ee^{\rm flat}\!\left[\tau^{\inv2\bk
        h(t,0)}\right]=\bk\taufac\sum_{k=0}^{\bk}\frac{(-1)^{\frac12k(k+1)}}{2^{k}k!}\sum_{\substack{m_1,\dotsc,m_{k}=1,\\m_1+\dotsm+m_k=\bk}}^\infty
    \int_{C^{k}_{0,1}}d\vec y\,\\\times\prod_{1\leq
      a<b\leq k}\tau^{-\inv2m_am_b}\prod_{a=1}^{k}\mfu(y_a,m_a)\pf\!\left[\Kmu(\vec
      y;\vec m)\right].
  \end{multline}

At this point we can get rid of the delta functions inside $\Kmu$, by rewriting
\eqref{eq:flatASEPmoments} in such a way that we can take the $y_a$ integrations inside
the Pfaffian. To this end, let $p_a=\frac{1-\tau^{m_a/2}y_a}{1+\tau^{m_a/2}y_a}$, so that
$\frac{\tau^{m_a/2}y_a-\tau^{m_b/2}y_b}{\tau^{(m_a+m_b)/2}y_ay_b-1}=\frac{p_a-p_b}{p_a+p_b}$. Note
that $\Re(p_a)>0$. The key will be to use the identity
\begin{equation}\label{eq:signId}
  \frac{p_a-p_b}{p_ap_b(p_a+p_b)}=\int_{(\rr_{\geq0})^2}d\lambda_a\,d\lambda_b\,e^{-\lambda_ap_a-\lambda_bp_b}\sgn(\lambda_b-\lambda_a).
\end{equation}
In order to use it, we first use the Pfaffian identity \eqref{eq:diag2Pf} to rewrite \eqref{eq:flatASEPmoments} as
\begin{equation}
  \begin{split}
    &\ee^{\rm
      fl}\!\left[\tau^{\inv2mh(t,0)}\right]=m\taufac\sum_{k=0}^m\inv{2^kk!}\sum_{\substack{m_1,\dotsc,m_{k}=1,\\
      m_1+\dotsm+m_k=m}}^\infty
    \itwopii{k}\int_{C^{k}_{0,1}}\!\!\!d\vec y\,
    \prod_{1\leq a<b\leq k}\tau^{-\inv2m_am_b}\prod_a\mfu(y_a,m_a)p_a\\
    &\hspace{0.5in}\times(-1)^{\frac12k(k+1)}\pf\!\left[\begin{matrix} \big[
        \Kmu_{1,1}(y_a,y_b;m_a,m_b)\big]_{a,b=1}^k &
        \big[\tfrac{1}{p_b}\Kmu_{1,2}(y_a,y_b;m_a,m_b)\big]_{a,b=1}^k \\
        -\big[\tfrac1{p_a}\Kmu_{1,2}(y_b,y_a;m_b,m_a)\big]_{a,b=1}^k &
        \big[\tfrac{p_a-p_b}{p_ap_b(p_a+p_b)}\big]_{a,b=1}^k \end{matrix}\right].
  \end{split}\label{eq:pf4}
\end{equation}
Now using \eqref{eq:signId} and Lemma \ref{lem:pfaffInt}, which is a certain Pfaffian version
of the Andr\'eief (or generalized Cauchy-Binet) identity \cite{andreief} (more precisely, we use a version of the lemma
with the matrices inside the two Pfaffians transposed), with respect to the integration in
the $\lambda_a$'s, with $\phi_a(\lambda_a)=e^{-\lambda_ap_a}$, we get
\begin{equation}
  \begin{split}
  \ee^{\rm flat}&\!\left[\tau^{\inv2mh(t,0)}\right]=m\taufac\sum_{k=0}^m\inv{k!}\sum_{\substack{m_1,\dotsc,m_{k}=1\\
      m_1+\dotsm+m_k=m}}^\infty\itwopii{k}\int_{C^{k}_{0,1}}\!\!\!d\vec y\,
  \int_{(\rr_{\geq0})^k}d\vec\lambda\\
  &\quad\times\prod_{a=1}^{k}\mfu(y_a,m_a)p_ae^{-\lambda_ap_a}\,\,\,\,\smashoperator{\prod_{1\leq a<b\leq k}}\,\,\,\,\tau^{-\inv2m_am_b}\\
  &\quad\times\frac{(-1)^{\frac12k(k+1)}}{2^k}\pf\!\left[\begin{matrix} \big[
      \Kmu_{1,1}(y_a,y_b;m_a,m_b)\big]_{a,b=1}^k &
      \big[\Kmu_{1,2}(y_a,y_b;m_a,m_b)\big]_{a,b=1}^k \\
      -\big[\Kmu_{1,2}(y_b,y_a;m_b,m_a)\big]_{a,b=1}^k &
      \big[\sgn(\lambda_b-\lambda_a)\big]_{a,b=1}^k \end{matrix}\right].
\end{split}
\label{eq:pvPfaffPre}
\end{equation}

We would like next to interchange the $y_a$ and $\lambda_a$ integrals. The difficulty we
face is that the $y_a$ contours pass through singularities of the factors $\mfuap(y_a,m_a)$
in the integrand, and after interchanging the integrals the zeros of the second Pfaffian
in \eqref{eq:T2kp} (see footnote \ref{ftn:sing} in page \pageref{ftn:sing}) are not there
any more cancel them. This means in particular that the $y_a$ integrals that
result after the interchange will have to be regarded as principal value integrals. In
order to justify the interchange of limits let us restrict ourselves to the case
$k=2$. The general case can be justified similarly but the notation becomes much heavier.

When $k=2$ we may use the Pfaffian formula
\[\pf\!\left[\begin{smallmatrix}
    0 & a & b  & c \\
    -a & 0 & d & e \\
    -b & -d & 0 & f \\
    -c & -e & -f & 0\\
  \end{smallmatrix}\right]=af-be+dc\]
to write the double integral in \eqref{eq:pvPfaffPre} as
\begin{multline}
  \frac{-1}{4\twopii{2}}\int_{C^{2}_{0,1}}\!\!\!d\vec y\,
  \int_{(\rr_{\geq0})^2}d\vec\lambda\,\mfu(y_1,m_1)\mfu(y_2,m_2)p_1p_2e^{-\lambda_1p_1-\lambda_2p_2}
  \tau^{-\inv2m_1m_2}\\
  \times\Big[\Kmu_{1,1}(y_1,y_2;m_1,m_2)\sgn(\lambda_2-\lambda_1)
    -\Kmu_{1,2}(y_1,y_1;m_1,m_1)\Kmu_{1,2}(y_2,y_2;m_2,m_2)\\
    +\Kmu_{1,2}(y_1,y_2;m_1,m_2)\Kmu_{1,2}(y_2,y_1;m_2,m_1)\Big].
\end{multline}
The last two terms in the bracket pose no difficulty. In fact, recalling that
$\Kmu_{1,2}(y,y';m,m')$ only depends on $y$ and $m$, each of the two terms involves
products of delta functions depending on different variables, and one can check directly that the
$y_1,y_2$ integration (which corresponds to summing the values at $\pm1$) can be done
before or after the $\lambda_1,\lambda_2$ integration.

Now consider the first term in the bracket above and call $I_1$ the associated
integral. Integrating $\lambda_1,\lambda_2$ first and using \eqref{eq:signId} yields
\[I_1=\frac{-1}{4\twopii{2}}\int_{C^{k}_{0,1}}\!\!\!d\vec y\,\mfu(y_1,m_1)\mfu(y_2,m_2)\tau^{-\inv2m_1m_2}
\frac{p_1-p_2}{p_1+p_2}\Kmu_{1,1}(y_1,y_2;m_1,m_2).\] $\Kmu_{1,1}$ is defined in
\eqref{eq:Ks} and has two terms. Call $I_{1,1}$ and $I_{1,2}$ the integrals associated to
each (so that $I_1=I_{1,1}+I_{1,2}$). As before, the second integral poses no difficulty
for the interchange of $\lambda_a$ and $y_a$ integrals, so let us focus on $I_{1,1}$.
Here, thanks to the prefactor $\frac{p_1-p_2}{p_1+p_2}$, the integrand has no
singularities on the $y_a$ contours and thus the integration simply yields
\begin{equation}
I_{1,1}=\frac{-1}{2\pi\I}\int_{C_{0,1}}dy\,\mfu(y,m_1)\mfu(\tfrac1{y},m_1)
\mfuap(y,m_1)\mfuap(\tfrac1{y},m_1)\frac{p(y,m_1)-p(\frac1{y},m_1)}{p(y,m_1)+p(\frac1{y},m_1)}
\tau^{-\inv2m_1^2}\uno{m_1=m_2}\label{eq:I11}
\end{equation}
with $p(y,m)=\frac{1-\tau^{m/2}y}{1+\tau^{m/2}y}$.

Next we let let $\wt I_{1,1}$ be the analogous integral where we now integrate first in
$\vec y$. In this case the integrand has singularities along the $\vec y$ contour,
so we will have to regard this as a principal value integral while dealing with the delta
functions at the same time. More precisely, letting $C^\delta_{0,1}$ be portion of the contour $C_{0,1}$ which lies
outside of the circles of radius $\delta$ centred at $\pm1$, and letting $\varphi_\ep$ be
a suitable mollifier, we \emph{define} $\wt
I_{1,1}$ as
\begin{multline}\label{eq:wtI11}
  \wt I_{1,1}=\lim_{\ep\to0}\lim_{\delta_1,\delta_2\to0}\frac{-1}{\twopii{2}}
  \int_{(\rr_{\geq0})^2}d\vec\lambda\,\,\int_{C^{\delta_1}_{0,1}\times C^{\delta_2}_{0,1}}\!\!\!d\vec y\,
  \mfu(y_1,m_1)\mfu(y_2,m_1)\mfuap(y_1,m_1)\mfuap(y_2,m_1)\\
  \times \tau^{-\inv2m_1^2}p_1p_2e^{-\lambda_1p_1-\lambda_2p_2}\sgn(\lambda_2-\lambda_1)\varphi_\ep(y_1-1/y_2).
\end{multline}
Using Cauchy's Theorem, each of the $\delta_a\to0$ limits equals the integral over a
circle $\wt C_{0,1}$ of radius slightly smaller than 1 plus half the residues of the
integrand at $y_a=\pm1$. Using this to compute the two limits (one after the other) leads,
after some calculation, to
\begin{align}
  \wt I_{1,1}&=-\lim_{\ep\to0}\frac1{\twopii{2}}
  \int_{(\rr_{\geq0})^2}d\vec\lambda\,\,\int_{(\wt C_{0,1})^2}\!\!\!d\vec y\,\,\tau^{-\inv2m_1^2}
  \mfu(y_1,m_1)\mfu(y_2,m_1)\mfuap(y_1,m_1)\mfuap(y_2,m_1)\\
  &\hspace{1.4in}\times
  p_1p_2e^{-\lambda_1p_1-\lambda_2p_2}\sgn(\lambda_2-\lambda_1)\varphi_\ep(y_1-1/y_2)\\
  &\quad-\lim_{\ep\to0}\sum_{\sigma\in\{-1,1\}}\frac{1}{2\pi\I}\int_{(\rr_{\geq0})^2}d\vec\lambda\int_{\wt
    C_{0,1}}dy\,\tau^{-\inv2m_1^2}
   \mfu(y,m_1)\mfu(\sigma,m_1)\mfuap(y,m_1)\sigma_2(-1)^{m_1}\\
   &\hspace{0.8in}\times\tau^{-\tfrac12m_1}p(y,m_1)p(\sigma_2,m_1)e^{-\lambda_1p(y,m_1)-\lambda_2p(\sigma_2,m_1)}
   \sgn(\lambda_2-\lambda_1)\varphi_\ep(y-\sigma)\\
   &\quad-\lim_{\ep\to0}\frac{1}4\sum_{\sigma_1,\sigma_2\in\{-1,1\}}\int_{(\rr_{\geq0})^2}d\vec\lambda\,\tau^{-\inv2m_1^2}
   \mfu(\sigma_1,m_1)\mfu(\sigma_2,m_1)\sigma_1\sigma_2\tau^{-m_1}\\
   &\hspace{0.8in}\times p(\sigma_1,m_1)p(\sigma_2,m_1)e^{-\lambda_1p(\sigma_1,m_1)-\lambda_2p(\sigma_2,m_1)}
   \sgn(\lambda_2-\lambda_1)\varphi_\ep(\sigma_1-\sigma_2).
\end{align}
At this point we can perform the $\lambda_a$ integrals using \eqref{eq:signId} (note
that in the first two terms there is now no difficulty in interchanging the $\lambda_a$
and $y_a$ integrations). Starting with the last one, it yields a factor
$\frac{p(\sigma_1,m_1)-p(\sigma_2,m_1)}{p(\sigma_1,m_1)+p(\sigma_2,m_1)}$; when
$\sigma_1=\sigma_2$ this factor vanishes, while when $\sigma_1\neq\sigma_2$ the limit
$\lim_{\ep\to0}\varphi_{\ep}(\sigma_1-\sigma_2)$ vanishes. This means that the third limit
above is zero. One can check similarly that the second limit vanishes, by performing the
$\lambda_a$ integrations to obtain a factor $\frac{p(\sigma,m_1)-p(y,m_1)}
{p(\sigma,m_1)+p(y,m_1)}$ and then computing the $\ep\to0$ limit to yield evaluation at
$y=\sigma$, for which this factor vanishes. Performing the $\lambda_a$ integrals now in
the first term and taking $\ep\to0$ shows that $\wt I_{1,1}$ equals the right hand side of
\eqref{eq:I11} with the contour $C_{0,1}$ on the right hand side replaced by $\wt
C_{0,1}$. But since the integrand has no singularities along $C_{0,1}$, Cauchy's Theorem
implies that $I_{1,1}=\wt I_{1,1}$ as desired.

The conclusion of all this (and its extension to all $k\geq1$) is that we may interchange
the $\lambda_a$ and $y_a$ integrations in \eqref{eq:pvPfaffPre}, provided we mollify
the delta functions and we interpret the $y_a$ integrals as principal value integrals
as done in \eqref{eq:wtI11}. Now observe that in the Pfaffian appearing in
\eqref{eq:pvPfaffPre}, the $(1,2)$ entry only depends on $y_a$, the $(2,1)$ entry only
depends on $y_b$, and the $(2,2)$ entry depends on neither of the two. Hence we can use
Lemma \ref{lem:pfaffInt} again, this time to bring the $y_a$ integrals inside the
Pfaffian. Let us first define
\begin{align}
    \mfv(\lambda_a,y_a,m_a)&=p_ae^{-\lambda_ap_a}\mfu(y_a,m_a)\\
    &=\frac{1}{y_a}(1-\tau)^{m_a}\tau^{m_a/2}\frac{1-\tau^{m_a/2}y_a}{1+\tau^{m_a/2}y_a}
    e^{-\lambda_a\frac{1-\tau^{m_a/2}y_a}{1+\tau^{m_a/2}y_a}+t\Big[\frac{1}{1+\tau^{-m_a/2}y_a}-\frac1{1+\tau^{m_a/2}y_a}\Big]}\\
    &\hspace{1.6in}\times\frac{\pochinf{-\tau^{-m_a/2}y;\tau}}{\pochinf{-\tau^{m_a/2}y_a;\tau}}
    \frac{\pochinf{\tau^{1+m_a}y^2;\tau}}{\pochinf{\tau y_a^2;\tau}}
\label{eq:def-mfv-pre}
\end{align}
where we have used the definition of $\mfu$ in \eqref{eq:def-mfu} (and of $\mff_1$ in
\eqref{eq:germans-hf}), and observe that the middle line of \eqref{eq:pvPfaffPre} can be
written as $\tau^{-m^2/4}\prod_a\tau^{\inv4m_a^2}\mfv(\lambda_a,y_a,m_a)$ (recall that
$\sum_am_a=m$). Define also\footnote{The notation $\,\pvint$ in the second line of this
  formula indicates that this is a Cauchy principal value integral.}
\begin{equation}
  \label{eq:hatKmu}
  \begin{split}
    \KASEP_{1,1}(\lambda_a,\lambda_b;m_a,m_b)&=\frac1{2\twopii{2}}\pvint_{\!\!\!C_{0,1}^2}dy\,dy'\,\tau^{\inv4(m_a^2+m_b^2)}\mfv(\lambda_a,y,m_a)\mfv(\lambda_b,y',m_b)\\
    &\hspace{3in}\times\Kmu_{1,1}(y,y';m_a,m_b)\\
    &=\uno{m_a=m_b}\inv{\pi\I}\pvint_{\!\!\!\!C_{0,1}}dy\,\tau^{\inv2m_a^2}\mfv(\lambda_a,y,m_a)\mfv(\lambda_b,1/y,m_b)\mfuap(y,m_a)\\
    &\qquad+\frac12\sum_{\sigma,\sigma'\in\{-1,1\}}(-\sigma\sigma')^{m_a\wedge
      m_b+1}\sgn(\sigma'm_b-\sigma m_a)\\
    &\hspace{1.2in}\times\tau^{\inv4(m_a^2+m_b^2)-\inv2(m_a+m_b)}\mfv(\lambda_a,\sigma,m_a)\mfv(\lambda_b,\sigma',m_b),\\
    \KASEP_{1,2}(\lambda_a,\lambda_b;m_a,m_b)&=\frac1{4\pi\I}\int_{C_{0,1}}dy\,\tau^{\inv4m_a^2}\mfv(\lambda_a,y,m_a)\Kmu_{1,2}(y,-;m_a,-)\\
    &=-\inv2\sum_{\sigma\in\{-1,1\}}\tau^{\inv4m_a^2-\inv2m_a}\mfv(\lambda_a,\sigma,m_a),\\
    \KASEP_{2,2}(\lambda_a,\lambda_b;m_a,m_b)&=\frac12\sgn(\lambda_b-\lambda_a)
  \end{split}
\end{equation}
where $\mfuap$ is given in \eqref{eq:def-mfu}, $\mfv$ is given in
\eqref{eq:def-mfv-pre} (the $-$'s in $\Kmu_{1,2}$ denote the fact that this kernel does
not depend on those arguments), and we have used the fact that
for $\sigma\in\{-1,1\}$ the function $\mfuu(\sigma,m)$ (defined in \eqref{eq:def-mfu}) satisfies
$\mfuu(\sigma,m)=\tau^{-\frac m2}\sigma\frac{1-\tau^n\sigma^2}{1-\tau^n}=\sigma\tau^{-\frac m2}$.
Then, by virtue of Lemma \ref{lem:pfaffInt} again, we get
\begin{multline}
  \ee^{\rm flat}\!\left[\tau^{\inv2mh(t,0)})\right]=m\taufac\tau^{-\inv4m^2}\sum_{k=0}^m\inv{k!}\sum_{\substack{m_1,\dotsc,m_{k}=1,\\
      m_1+\dotsm+m_k=m}}^\infty\int_{(\rr_{\geq0})^k}d\vec\lambda\,
  \\\times(-1)^{\frac12k(k+1)}\pf\!\left[\begin{smallmatrix}
      \big[\KASEP_{1,1}(\lambda_a,\lambda_b;m_a,m_b)\big]_{a,b=1}^k &
      \big[\KASEP_{1,2}(\lambda_a,\lambda_b;m_a,m_b)\big]_{a,b=1}^k \\
      -\big[\KASEP_{1,2}(\lambda_b,\lambda_a;m_b,m_a)\big]_{a,b=1}^k &
      \big[\KASEP_{2,2}(\lambda_a,\lambda_b;m_a,m_b)\big]_{a,b=1}^k \end{smallmatrix}\right].\label{eq:pvPfaffian}
\end{multline}
(Note that we have used the fact that $\pf[aA]=a^n\pf[A]$ for a skew-symmetric
$2n\times2n$ matrix $A$ and a scalar $a$.  See \eqref{eq:detPf}.)

The formulas for $\KASEP_{1,1}$ and $\KASEP_{1,2}$ follow directly\footnote{\label{ft:indic}Observe that
  we have dropped the factor $\uno{a\neq b}$ in the second term of the second equality in
  the definition of $\KASEP_{1,1}$. We may do this because this term vanishes when
  $\lambda_a=\lambda_b$ and $m_a=m_b$, as can be checked using the antisymmetry of
  $(m,m')\longmapsto\sgn(m-m')$.} from the definitions of $\mfv$, $\Kmu_{1,1}$ and
$\Kmu_{1,2}$, with the nuisance that we have to justify writing a principal value integral
in $\KASEP_{1,1}$. This involves writing out explicitly the result of the interchange of
the $\lambda_a$ and $y_a$ integrations in \eqref{eq:pvPfaffPre} using mollifiers and
principal value integrals in a similar way as \eqref{eq:wtI11}, then moving the $y_a$ integrals
inside the Pfaffian, and finally showing that the limits involving the mollification and
the principal value integrals can be taken back inside the Pfaffian to yield the kernel
$\KASEP$. The argument is very similar to the one we used to justify the interchange of
limits itself, so we omit it.

The factor $(-1)^{\frac12k(k+1)}$ in front of the Pfaffian in \eqref{eq:pvPfaffian} now allows us to write the
formula in the form of a Fredholm Pfaffian. In fact, let $\KASEP$ be the $2\times2$
matrix kernel
\begin{equation}
  \KASEP(\lambda_a,\lambda_b;m_a,m_b)=\left[\begin{matrix}\KASEP_{1,1}(\lambda_a,\lambda_b;m_a,m_b) &
    \KASEP_{1,2}(\lambda_a,\lambda_b;m_a,m_b) \\
    -\KASEP_{1,2}(\lambda_b,\lambda_a;m_b,m_a) &
    \KASEP_{2,2}(\lambda_a,\lambda_b;m_a,m_b)\end{matrix}\right].\label{eq:hatKmatrix}
\end{equation}
The matrix $\big[\KASEP(\lambda_a,\lambda_b;m_a,m_b)\big]_{a,b=1}^k$ differs from the one inside
the above Pfaffian by a permutation of the rows and columns which has sign
$(-1)^{k(k-1)/2}$, and thus we get
 \begin{equation}
   \ee^{\rm flat}\!\left[\tau^{\inv2mh(t,0)}\right]=m\taufac\tau^{-\inv4m^2}\sum_{k=0}^m\frac{(-1)^k}{k!}
   \quad\smashoperator{\sum_{\substack{m_1,\dotsc,m_{k}=1,\\
         m_1+\dotsm+m_k=m}}^\infty}\,\,
   \quad\int_{(\rr_{\geq0})^k}d\vec\lambda\,\pf\!\left[\KASEP(\lambda_a,\lambda_b;m_a,m_b)\right]_{a,b=1}^k.\label{eq:ys-in}
\end{equation}
This gives Theorem \ref{prop:secondPfaffian}. 

As a corollary we get a bound on the moments of $\tau^{\inv2h(t,0)}$, which will be useful
later on when we form a generating function for $\tau^{\inv2h(t,0)}$:

\begin{cor}\label{cor:momentBd}
  There is a $c>0$ such that  for all $m\in\zz_{\geq1}$,
  \begin{equation}
    \ee^{\rm flat}\!\left[\tau^{\inv2mh(t,0)}\right]\leq c\hspace{0.1em}m\taufac\tau^{-\inv4m^2}\label{eq:momentBd}
  \end{equation}
\end{cor}

We note that for fixed $\tau$ we have $m\taufac\leq c(1-\tau)^{-m}$ (this follows from the
$q$-analogue of Stirling's formula for the $q$-Gamma function), and hence the above moment
bound is actually of the form $c\hspace{0.1em}\tau^{-\inv4m^2-O(m)}$. Before getting
started with the proof we need to establish an estimate on certain ratios of
$q$-Pochhammer symbols.

\begin{lem}\label{lem:mfgbd}
  For any $\eta>\inv8$ and $\tau\in(0,1)$ we have
  \[\sup_{n\in\zz_{\geq1},\,|z|=1}\tau^{\eta
    n^2}\frac{\pochinf{-\tau^{-n/2}z;\tau}}{\pochinf{-\tau^{n/2}z;\tau}}
  \frac{\pochinf{\tau^{1+n}z^2;\tau}}{\pochinf{\tau z^2;\tau}}<\infty.\]
\end{lem}

\begin{proof}
  By definition of the $q$-Pochhamer symbol,
  \[\tfrac{\pochinf{-\tau^{-n/2}z;\tau}}{\pochinf{-\tau^{n/2}z^2;\tau}}
  \tfrac{\pochinf{\tau^{1+n}z^2;\tau}}{\pochinf{\tau z^2;\tau}} = \textstyle\prod\nolimits_{\ell=0}^{n-1}
  \tfrac{1+\tau^{\ell-n/2}z}{1-\tau^{\ell+1}z^2}.\]
  It is not hard to see that the absolute value of each factor is maximized for $z$ in the
  unit circle at $z=1$, whence
  \[\abs{\tfrac{\pochinf{-\tau^{-n/2}z;\tau}}{\pochinf{-\tau^{n/2}z^2;\tau}}
  \tfrac{\pochinf{\tau^{1+n}z^2;\tau}}{\pochinf{\tau z^2;\tau}}}\leq
  \tfrac{\prod_{\ell=0}^{n-1}(1+\tau^{\ell-n/2})}{\prod_{\ell=1}^{n}(1-\tau^{\ell})}
  \leq\tfrac1{\pochinf{\tau;\tau}}\textstyle\prod\nolimits_{\ell=0}^{n-1}(1+\tau^{\ell-n/2}).\]
  Expanding the last product and bounding each term by the biggest one, which is given by
  $\prod_{\ell=0}^{\lfloor n/2\rfloor}\tau^{\ell-n/2}\approx\tau^{-n^2/8}$ results in the
  bound
  \[\abs{\tfrac{\pochinf{-\tau^{-n/2}z;\tau}}{\pochinf{-\tau^{n/2}z^2;\tau}}
  \tfrac{\pochinf{\tau^{1+n}z^2;\tau}}{\pochinf{\tau z^2;\tau}}}
  \leq\tfrac{C}{\pochinf{\tau;\tau}}2^n\tau^{-n^2/8}\]
  for some $C>0$, which implies the result.
\end{proof}

\begin{proof}[Proof of Corollary \ref{cor:momentBd}]
  Recall that $\mfv(\lambda_a,y,m_a)$ contains a factor $e^{-\lambda_ap_a}$ with
  $p_a=\frac{1-\tau^{m_a/2}y}{1+\tau^{m_a/2}y}$ and note that $\Re(p_a)\geq c_1$ for some
  $c_1>0$, uniformly in $y\in C_{0,1}$ and $m_a\in\zz_{\geq1}$. Having chosen $c_1$ in
  this way, we use \eqref{eq:diag2Pf} to factor out
  $\prod_{a}e^{-\lambda_ap_a}\tau^{m_a^2/16}$ from the Pfaffian:
  \[\pf\!\left[\KASEP(\lambda_a,\lambda_b;m_a,m_b)\right]_{a,b=1}^\ell
  =\prod_{a=1}^\ell e^{-c_1\lambda_a}\tau^{\inv{16}m_a^2}\pf\!\left[{\KASEP}'(\lambda_a,\lambda_b;m_a,m_b)\right]_{a,b=1}^\ell\]
  where \[{\KASEP}'(\lambda_a,\lambda_b;m_a,m_b)=\left[
    \begin{smallmatrix}
      e^{c_1(\lambda_a+\lambda_b)}\tau^{-(m_a^2+m_b^2)/16}\Kmu_{1,1}(\lambda_a,\lambda_b;m_a,m_b) &
      e^{c_1\lambda_a}\tau^{-m_a^2/16}\Kmu_{1,2}(\lambda_a,\lambda_b;m_a,m_b)\\
      -e^{c_1\lambda_b}\tau^{-m_b^2/16}\Kmu_{1,2}(\lambda_b,\lambda_a;m_b,m_a) &
      \Kmu_{2,2}(\lambda_a,\lambda_b;m_a,m_b)
    \end{smallmatrix}\right].\]
  The point of rewriting $\KASEP$ in this manner is that the entries of
  $\KASEP(\lambda_1,\lambda_2;m_1,m_2)$ are uniformly bounded for
  $\lambda_1,\lambda_2\geq0$ and $m_1,m_2\in\zz_{\geq1}$, say by a constant $c_2>0$. This
  can be checked straightforwardly for most of the factors involved, except for two minor
  issues to address. First, $\mfv(\lambda_a,y,m_a)$ contains a factor of the form
  $g(y,m_a):=\frac{\pochinf{-\tau^{-m_a/2}y;\tau}}{\pochinf{-\tau^{m_a/2}y;\tau}}
  \frac{\pochinf{\tau^{1+m_a}y^2;\tau}}{\pochinf{\tau y^2;\tau}}$, which is not bounded in
  $m_a$. Nevertheless, $\mfv(\lambda_a,y,m_a)$ appears next to $\tau^{m_a^2/4}$, which
  together with the factor $\tau^{-m_a^2/16}$ introduced above gives a factor
  $\tau^{3m_a^2/16}$. Thanks to Lemma \ref{lem:mfgbd} this factor is enough to balance the
  growth of $g(y,m_a)$. The other term which is not seen directly to be bounded is the one
  coming from $\Kmu_{1,1}$ involving the principal value integral
  $\uno{m_a=m_b}\inv{2\pi\I}\pvint_{\!C_{0,1}\!\!}dy\,\mfv(\lambda_a,y,m_a)\mfv(\lambda_b,1/y,m_a)\mfuap(y,m_a)$.
  But in this case the uniform bound follows from the fact that the integrand can be expanded around $\pm1$ as
  $\frac{1}{1\mp y}+\mathcal{O}(1)$.

  Using the fact that $\KASEP$ is uniformly bounded, the identity $\pf[A]^2=\det[A]$
  and Hadamard's bound we deduce that
  $\big|\!\pf\!\big[\KASEP(\lambda_a,\lambda_b;m_a,m_b)\big]_{a,b=1}^k\big|\leq
  k^{k/2}c_2^k\prod_a e^{-c_1\lambda_a}\tau^{m_a^2/16}$, and thus
  \begin{equation}
    \left|\int_{(\rr_{\geq0})^k}d\vec\lambda\,\pf\!\big[\KASEP(\lambda_a,\lambda_b;m_a,m_b)\big]_{a,b=1}^k\right|
    \leq(c_2/c_1)^kk^{k/2}\prod_a\tau^{\inv{16}m_a^2}.\label{eq:pfbd}
  \end{equation}

  Now we use this bound in \eqref{eq:ys-in} to deduce that
  \[\ee^{\rm flat}\!\left[\tau^{\inv2mh(t,0)}\right]\leq
    m\taufac\hspace{0.1em}
    \tau^{-\inv4m^2}\sum_{k=0}^m\frac{1}{k!}\qquad\smashoperator{\sum_{\substack{m_1,\dotsc,m_{k}=1,\\
          m_1+\dotsm+m_k=m}}^\infty}\,\,\quad\prod_{a}\tau^{\inv{16}m_a^2}c^kk^{k/2}
    \leq c'm\taufac\hspace{0.1em}\tau^{-\inv4m^2}\]
  for some $c,c'>0$ which are independent of $m$ (using Stirling's formula).
\end{proof}

\section{Generating function}\label{sec:generatingfunction}

We are in position now to turn this into a formula for a certain generating function of
flat ASEP. We will consider the following generalized $q$-exponential function: for fixed
$\xi\in\cc$ with $|\xi|\leq1$,
\begin{equation}
  \label{eq:gen-q-exp}
  \exp_q(x;\xi)=\sum_{k=0}^\infty\inv{k\qfac}\xi^{k(k-1)}x^k.
\end{equation}
Note that for $|\xi|<1$ and $x\in\cc$ this is function is analytic in both variables. On
the other hand, for fixed $\xi$ with $|\xi|=1$, this function is analytic in $x$ for
$|x|<1$. Choosing $\xi$ to be 1, $q^{1/2}$ and $q^{1/4}$ yields, respectively,
$e_q(x)$, $E_q(x)$ and $\exp_q(x)$ (although note that, for $\xi=1$ it only yields
$e_q(x)$ restricted to $|x|<1$), see \eqref{eq:standardqexps} and \eqref{eq:symmqexp}.

The next result gives a formula for the expectation of
$\exp_\tau(\zeta\tau^{\inv2h(t,0)};\xi)$ for $|\xi|\leq\tau^{1/4}$. The restriction on
$|\xi|$ comes from the fact that we will use our moment bound \eqref{eq:momentBd}.

\begin{thm}\label{thm:taulaplflat-pre}
  Let $\xi,\zeta\in\cc$ and assume that either $|\xi|<\tau^{1/4}$ or $|\xi|=\tau^{1/4}$
  and $|\zeta|<\tau^{1/4}$. Then
  \begin{multline}\label{eq:taulaplflat-pre}
    \ee^{\rm flat}\!\left[\exp_\tau(\zeta\tau^{\inv2h(t,0)};\xi)\right]
    =\sum_{k=0}^{\infty}(-1)^{k}\frac{1}{k!}\sum_{m_1,\dotsc,m_{k}=1}^\infty\itwopii{k}
    \int_{(\rr_{\ge 0})^k}d\vec\lambda\,\\\times\zeta^{\sum_am_a}\tau^{-\frac14(\sum_am_a)^2}\xi^{(\sum_am_a)^2-\sum_am_a}
    \pf\!\left[\KASEP(\lambda_a,\lambda_b;m_a,m_b)\right]_{a,b=1}^k
  \end{multline}
  with $\KASEP$ as in \eqref{eq:hatKmu}. Moreover, the series on the right hand side is absolutely convergent.
\end{thm}

\begin{proof}
  Using the definition of $\exp_\tau(\cdot;\xi)$ and interchanging sum and
  expectation we have
  \begin{equation}
    \ee^{\rm flat}\!\left[\exp_\tau(\zeta\tau^{\inv2h(t,0)};\xi)\right]
    =\sum_{k=0}^\infty\inv{k\taufac}\zeta^k\xi^{k(k-1)}\ee^{\rm
      flat}\!\left[\tau^{\inv2kh(t,0)}\right].\label{eq:momexptau}
  \end{equation}
  Here, in order to justify the use of Fubini's Theorem, we are using Corollary
  \ref{cor:momentBd}, which gives $\ee^{\rm flat}\!\left[\tau^{\inv2kh(t,0)}\right]\leq
  ck\taufac\tau^{-k^2/4}$, and implies that the double sum on the right hand side is
  absolutely summable under the stated conditions on $\xi$ and $\zeta$.

  Now rewrite the right hand side of \eqref{eq:ys-in} as
  \begin{equation}
    k\taufac\sum_{\ell=0}^\infty\frac{(-1)^\ell}{\ell!}\quad\smashoperator{\sum_{\substack{m_1,\dotsc,m_{\ell}=1,\\
          m_1+\dotsm+m_\ell=k}}^\infty}\,\,\quad
    \tau^{-\inv4(\sum_am_a)^2}Q(\vec{m})
  \end{equation}
  with $Q(\vec m)=\int_{(\rr_{\geq0})^\ell}d\vec\lambda\,\pf\!\big[\KASEP(\lambda_a,\lambda_b;m_a,m_b)\big]_{a,b=1}^\ell$.
  To get a formula for the right hand side of \eqref{eq:momexptau} we multiply this by
  $\inv{k\taufac}\zeta^{k}\xi^{k(k-1)}$, sum over $k=0,\dotsc,N$ and then take
  $N\to\infty$ to get
  \[\ee^{\rm flat}\!\left[\exp_\tau(\zeta\tau^{\inv2h(t,0)};\xi)\right]
    =\lim_{N\to\infty}\sum_{\ell=0}^\infty\frac{(-1)^\ell}{\ell!}\quad\smashoperator{\sum_{\substack{m_1,\dotsc,m_{\ell}=1,\\
          m_1+\dotsm+m_\ell\leq
          N}}^\infty}\,\,\quad\zeta^{\sum_am_a}\tau^{-\frac14(\sum_am_a)^2}\xi^{(\sum_am_a)^2-\sum_am_a}Q(\vec m).\]
  We claim that the limit can be taken inside the first sum, which yields
    \[\ee^{\rm flat}\!\left[\exp_\tau(\zeta\tau^{\inv2h(t,0)};\xi)\right]
    =\sum_{\ell=0}^\infty\frac{(-1)^\ell}{\ell!}\quad\smashoperator{\sum_{m_1,\dotsc,m_{\ell}=1}^\infty}
    \,\,\quad\zeta^{\sum_am_a}\tau^{-\frac14(\sum_am_a)^2}\xi^{(\sum_am_a)^2-\sum_am_a}Q(\vec
    m)\] and proves the result. To see that the limit can be taken inside it is enough to
    show that the right hand side above is absolutely summable. But this follows from
    \eqref{eq:pfbd}, which gives $|Q(\vec m)|\leq
    c^\ell\ell^{\ell/2}\tau^{\sum_am_a^2/16}$. In fact, using this bound we get
  \begin{multline}
    \sum_{\ell=0}^\infty\frac{1}{\ell!}\quad\smashoperator{\sum_{m_1,\dotsc,m_{\ell}=1}^\infty}
    \,\,\quad|\zeta|^{\sum_am_a}\tau^{-\frac14(\sum_am_a)^2}|\xi|^{(\sum_am_a)^2-\sum_am_a}|Q(\vec
    m)|\\
    \leq\sum_{\ell=0}^\infty\frac{c^\ell \ell^{\ell/2}}{\ell!}\quad\smashoperator{\sum_{m_1,\dotsc,m_{\ell}=1}^\infty}
    \,\,\quad|\zeta|^{\sum_am_a}\tau^{-\frac14(\sum_am_a)^2+\inv{16}\sum_am_a^2}|\xi|^{(\sum_am_a)^2-\sum_am_a}\label{eq:unifbdexptau}
  \end{multline}
  which is finite, as desired, by Stirling's formula and our assumptions on $\xi$ and
  $\zeta$.
\end{proof}

We focus now on the case $\xi=\tau^{1/4}$ of Theorem \ref{thm:taulaplflat-pre}, which yields the $\exp_\tau$-Laplace
transform. In this case the factor $\xi^{(\sum_am_a)^2}$ in front of the Pfaffian in
\eqref{eq:taulaplflat-pre} cancels exactly with $\tau^{-(\sum_am_a)^2/4}$, and
hence the sums in $m_a$ can be brought inside the Pfaffian. To this end, and in view of
the definition of $\KASEP$ in \eqref{eq:hatKmu}, we define, for $\zeta\in\cc$,
\begin{equation}
  \wtKASEPz(\lambda_a,\lambda_b)=\left[\begin{matrix}\wtKASEPz_{1,1}(\lambda_a,\lambda_b) &
    \wtKASEPz_{1,2}(\lambda_a,\lambda_b) \\
    -\wtKASEPz_{1,2}(\lambda_b,\lambda_a) &
    \wtKASEPz_{2,2}(\lambda_a,\lambda_b)\end{matrix}\right],\label{eq:KASEPmatrix}
\end{equation}
with
\begin{equation}\label{eq:wtKASEP}
  \begin{split}
    \wtKASEPz_{1,1}(\lambda_a,\lambda_b)&=\sum_{m=1}^\infty\inv{\pi\I}\pvint_{\!\!\!\!C_{0,1}}dy\,\tau^{\inv2m^2}\zeta^{2m}
    \mfv(\lambda_a,y,m)\mfv(\lambda_b,1/y,m)
    \mfuap(y,m)\\
    &\qquad\qquad+\frac12\sum_{m,m'=1}^\infty\sum_{\sigma,\sigma'\in\{-1,1\}}(-\sigma\sigma')^{m\wedge m'+1}
    \sgn(\sigma'm'-\sigma m)\\
    &\hspace{1.3in}\times\tau^{\inv4(m^2+(m')^2)-\inv2(m+m')}\zeta^{m+m'}\mfv(\lambda_a,\sigma,m)\mfv(\lambda_b,\sigma',m'),\\
    \wtKASEPz_{1,2}(\lambda_a,\lambda_b)&=-\inv2\sum_{m=1}^\infty\sum_{\sigma\in\{-1,1\}}\tau^{\inv4m^2-\inv2m}\zeta^m
    \mfv(\lambda_a,\sigma,m),\\
    \wtKASEPz_{2,2}(\lambda_a,\lambda_b)&=\frac12\sgn(\lambda_b-\lambda_a),
  \end{split}
\end{equation}
where  $\mfuap$, and $\mfuu$ are given in \eqref{eq:def-mfu} and $\mfv$ is given in \eqref{eq:def-mfv-pre}.
Then using Theorem \ref{thm:taulaplflat-pre} and Lemma \ref{lem:pfaffInt} we deduce the
following

\begin{thm}\label{thm:exptau-pre}
  For $|\zeta|<\tau^{1/4}$,
  \begin{align}
    \ee\!\left[\exp_\tau\!\big(\zeta\tau^{\inv2h(t,0)}\big)\right]&=\sum_{k\geq0}\inv{k!}
    \int_{(\rr_{\geq0})^k}d\vec\lambda\,
    (-1)^{k}\pf\!\left[\wtKASEPz(\lambda_a,\lambda_b)\big]_{a,b=1}^k\right]_{a,b=1}^k\\
    &=\pf\!\left[J-\wtKASEPz\right]_{L^2([0,\infty))}
    \label{eq:exptautr}
  \end{align}
\end{thm}

\label{disc:airy}
Finally we are ready to prove Theorem \ref{thm:fredPf-intro}, for which we need to check
that the kernel given in \eqref{eq:wtKASEP} coincides with the one given in
\eqref{eq:kernelBessel}.  The proof amounts to what can be considered as an ASEP version
of the ``Airy trick'' commonly used in the physics literature. In fact, the representation
of $\mfv(\lambda,y,m)$ through an (inverse) double-sided Laplace transform (see
\eqref{eq:DLapInv} below) is an exact analog of the identity
\begin{equation}
e^{m^3/3}=\int_{-\infty}^\infty dw\Ai(w)e^{w m}\label{eq:airyTrick}
\end{equation}
for $m>0$ used in \cite{cal-led} (see
(42) and the identity right before (137) in that paper). A crucial difference, though, is
that in \cite{cal-led} (as in other situations were the ``trick'' is applied), the authors
are dealing with a divergent series, and the Airy trick is used to resum it through an
illegal interchange of sum and integration, whereas in our case such an interchange will
be fully justified. 

\begin{proof}[Proof of Theorem \ref{thm:fredPf-intro}]
  For simplicity we will write $K$ instead of $\wtKASEPz$ throughout this proof. The only
  two pieces of the kernel for which we need to show that the two definitions coincide are
  $K_{1,1}$ and $K_{1,2}$, and we will begin with the latter. Throughout the rest of the
  proof we fix $\beta>0$ as in the Section \ref{sec:fred-pf-intro}.

  For $\zeta\notin\rr_{>0}$ define
  \begin{equation}
    \label{eq:psiASEPtext}
    \psiASEP(s,\lambda,y;\zeta)=(-\zeta)^s\tau^{-\beta s^2-\inv4s}
    \frac{\pochinf{\tau y^2;\tau}}{(1-\tau)^s\pochinf{\tau^{1+s}y^2;\tau}}\mfv(\lambda,y,s).
  \end{equation}
  Recalling the definition of $\mfv$ in \eqref{eq:def-mfv-pre}, this definition coincides
  with the one given in \eqref{eq:psiASEP-intro}. Moreover, for $m\in\zz_{\geq1}$ and
  $\sigma\in\{-1,1\}$, since
  $\frac{\pochinf{\tau;\tau}}{(1-\tau)^m\pochinf{\tau^{1+m};\tau}}=m\taufac$, $\psiASEP$
  satisfies
  \begin{equation}\label{eq:relPsiMfv}
    \tau^{\inv4m^2-\inv2m}\zeta^m\mfv(\lambda,\sigma,m)  =
    (-1)^m\frac{\tau^{(\inv4+\beta)m^2-\inv4m}}{m\taufac}\,\psiASEP(m,\lambda,\sigma;\zeta).
  \end{equation}
  Therefore
  \begin{equation}\label{eq:stop}
    K_{1,2}(\lambda_a,\lambda_b)=-\inv2\sum_{m=1}^\infty\sum_{\sigma\in\{-1,1\}}(-1)^m
    \frac{\tau^{(\inv4+\beta)m^2-\inv4m}}{m\taufac}\,\psiASEP(m,\lambda_a,\sigma;\zeta).
  \end{equation}

  Now we introduce the inverse double-sided Laplace transform of $\psi$ (in $s$), given by
  \begin{equation}
    \cpsiASEP(\omega,\lambda,y;\zeta)=\inv{2\pi\I}\int_{\I\rr}ds\,e^{s\omega}\psiASEP(s,\lambda,y;\zeta).
  \end{equation}
  This integral converges thanks to the factor $\tau^{-\beta s^2}$ which we introduced in
  $\psi$, and we have the inversion formula (see \cite{widder}, Theorem VI.19a)
  \begin{equation}\label{eq:DLapInv}
    \psiASEP(m,\lambda,y;\zeta)=\int_{-\infty}^\infty
    d\omega\,e^{-m\omega}\cpsiASEP(\omega,\lambda,y;\zeta).
  \end{equation}
  Substituting \eqref{eq:DLapInv} into the right hand side of \eqref{eq:stop} yields
  \[K_{1,2}(\lambda_a,\lambda_b) = -\inv2\sum_{m=1}^\infty\sum_{\sigma\in\{-1,1\}}
  \frac{\tau^{(\frac14+\beta)m^2-\inv4m}}{m\taufac} \int_{-\infty}^\infty
  d\omega\,e^{-m\omega}\cpsiASEP(\omega,\lambda_a,\sigma;\zeta).\]

  In order to finish the verification of the identity for $K_{1,2}$ all that remains is to
  interchange the order of the summation in $m$ and the integral in $\omega$, which, as
  required in view of \eqref{eq:cpsiASEP-intro}, \eqref{eq:DefF} and
  \eqref{eq:kernelBessel}, leads to
  \[\KASEP_{1,2}(\lambda_a,\lambda_b)=\frac12\sum_{\sigma\in\{-1,1\}}\,\int_{-\infty}^\infty
  d\omega\, F_3(e^{-\omega}) \cpsiASEP(\omega,\lambda_a,\sigma;\zeta)\] with
  \[F_3(z) = - \sum_{m=1}^\infty\frac{\tau^{(\frac14+\beta)m^2-\inv4m}}{m\taufac}(-z)^m.\] The application
  of Fubini's theorem will be justified once we check that
  \begin{align}
    \label{eq:checkFubini}
    \int_{-\infty}^\infty d\omega\, \sum_{m=1}^\infty
    \bigg|\sigma\frac{\tau^{(\frac14+\beta)m^2-\inv4m}}{m\taufac}
    e^{-m\omega}\cpsiASEP(\omega,\lambda,\sigma;\zeta)\bigg|< \infty.
  \end{align}
  We first observe that $\psiASEP(s,\lambda,\sigma;\zeta)=\tau^{-\beta
    s^2}\zeta^sG(\tau^s,\lambda,\sigma)$, where $G$ is a continuous function, whose
  arguments vary over compact domains (recall that $s\in\I\rr$). Therefore there is some
  $C>0$ such that
  \begin{equation}
    \abs{\psiASEP(s,\lambda,\sigma;\zeta)}  \leq C\tau^{-\beta s^2}|\zeta^s|.
  \end{equation}
  Let us write $a=(\frac14+\beta)\abs{\log(\tau)}$ and $b=\beta\abs{\log(\tau)}$. Then we
  have, using the last inequality,
  \begin{equation}\label{eq:BoundPsiASEP}
    |\cpsiASEP(\omega,\lambda,\sigma;\zeta)|\leq C\int_{\I\rr} ds \,e^{bs^2+s\omega }
    =C\sqrt{\frac{\pi}{b}}e^{-\omega^2/4b},
  \end{equation}
  and thus the left hand side of \eqref{eq:checkFubini} can be bounded by a constant
  multiple of
  \[\sqrt{\frac{\pi}{b}}\int_{-\infty}^\infty d\omega\,\sum_{m=1}^\infty
  e^{-am^2-(1/4+\omega)m-\omega^2/4b} \leq C'\frac{1}{\sqrt{ab}} \int_{-\infty}^\infty
  d\omega\,e^{(\omega+1/4)^2/4a-\omega^2/4b}<\infty\] for some $C'>0$, where the last
  integral is finite since $0<b<a$.

  The argument for $K_{1,1}$ is similar. Using \eqref{eq:psiASEPtext} and the definition of
  $\mfuap$ in \eqref{eq:def-mfu} we may write
  \begin{align}
    K_{1,1}(\lambda_a,\lambda_b)
    =\sum_{m=1}^\infty&\,\inv{\pi\I}\pvint_{\!\!\!\!C_{0,1}}dy\,(-1)^{m}(1-\tau)^{2m}
    \tau^{(\frac12+2\beta)m^2-m}\frac{1+y^2}{y^2-1}\\
    &\hspace{0.4in}\times\frac{\pochinf{\tau^{1+m}y^2;\tau}\pochinf{\tau^{1+m}/y^2;\tau}}
    {\pochinf{\tau y^2;\tau}\pochinf{\tau/y^2;\tau}}\psiASEP(m,\lambda_a,y;\zeta)\psiASEP(m,\lambda_b,\tfrac1y;\zeta)\\
    &+\frac12\sum_{m,m'=1}^\infty\sum_{\sigma,\sigma'\in\{-1,1\}}(-1)^{m+m'}(-\sigma\sigma')^{m\wedge
      m'+1}
    \sgn(\sigma'm'-\sigma m)\\
    &\qquad\qquad\times\frac{\tau^{(\inv4+\beta)(m^2+(m')^2)-\inv4(m+m')}}{m\taufac\, m'\taufac}
    \psiASEP(m,\lambda_a,\sigma;\zeta)\psiASEP(m',\lambda_b,\sigma';\zeta).
  \end{align}
  Plugging in the Laplace transform formula \eqref{eq:DLapInv} for $\psi$, we would like
  to interchange the integration in $\omega$ and $\omega'$ with the summation in $m$ and
  $m'$ as in the $K_{1,2}$ case. This leads to
  \begin{multline}
    K_{1,1}(\lambda_a,\lambda_b) =\int_{\rr^2} d\omega\,d\omega'\, \frac1{\pi i}
    \pvint_{\!\!\!\!C_{0,1}}dy\,\cpsiASEP(\omega,\lambda_a,y;\zeta)\cpsiASEP(\omega',\lambda_b,\tfrac1y;\zeta)F_1(e^{-\omega-\omega'},y)\\
    \qquad\qquad+\frac12\sum_{\sigma,\sigma'\in\{-1,1\}}\int_{\rr^2}d\omega\,d\omega'\,
    \cpsiASEP(\omega,\lambda_a,\sigma;\zeta) \cpsiASEP(\omega',\lambda_b,\sigma';\zeta)
    F_2(e^{-\omega},e^{-\omega'};\sigma,\sigma'),
  \end{multline}
  with
  \begin{align}
    F_1(z,y) & = \sum_{m=1}^\infty
    (-z)^{m}(1-\tau)^{2m}\tau^{(\frac12+2\beta)m^2-m}\frac{1+y^2}{y^2-1}
    \frac{\pochinf{\tau^{1+m}y^2;\tau}\pochinf{\tau^{1+m}/y^2;\tau}}
    {\pochinf{\tau y^2;\tau}\pochinf{\tau/y^2;\tau}}\\
    F_2(z_1,z_2;\sigma_1,\sigma_2) & = \sum_{m_1,m_2=1}^\infty (-z_1)^{m_1} (-z_2)^{m_2}
    (-\sigma_1\sigma_2)^{m_1\wedge m_2+1} \sgn(\sigma_2m_2-\sigma_1 m_1)\\
    &\hspace{2.8in}\times\frac{\tau^{(\inv4+\beta)(m_1^2+m_2^2)-\inv4(m_1+m_2)}}{(m_1)\taufac
      \,(m_2)\taufac},
  \end{align}
  which is exactly what we need. Hence all that remains is to justify the application of
  Fubini's Theorem. To this end, in the case of the first summand it is enough to check
  that
  \begin{multline}
    \int_{\rr^2}
    d\omega_1\,d\omega_2\,\bigg|\inv{2\pi\I}\pvint_{\!\!\!\!C_{0,1}}dy\,\frac{1+y^2}{y^2-1}
    \cpsiASEP(\omega_1,\lambda_a,y;\zeta)\cpsiASEP(\omega_2,\lambda_b,\tfrac1{y};\zeta)\bigg|\\
    \times\sum_{m=1}^\infty\bigg|e^{-(\omega_1+\omega_2)m}\frac{\tau^{(\frac12+2\beta)
        m^2-m}}{(m\taufac)^2}\bigg|<\infty.\label{eq:K11check}
  \end{multline}
  The $y$ contour passes through singularities of the factor $\frac{1+y^2}{y^2-1}$ at
  $y=\pm1$ and thus, as a principal value integral, it equals half the residues of the
  integrand at those points plus the same integral with the contour replaced by a circle
  $\wt C_{0,1}$ of radius slightly smaller than 1. Now the residues of
  $\frac{1+y^2}{y^2-1}$ at $y=\pm1$ equal $\pm1$ so the residues of the integrand at these
  points equal
  $\pm\cpsiASEP(\omega_1,\lambda_a,\pm1;\zeta)\cpsiASEP(\omega_2,\lambda_b,\pm1;\zeta)$.
  On the other hand, the estimate \eqref{eq:BoundPsiASEP} still holds (with a different
  constant) with $\sigma$ on the left hand side replaced by $y$, uniformly in $y\in\wt
  C_{0,1}\cup\{-1,1\}$. Therefore
  \begin{multline}
    \bigg|\inv{2\pi\I}\pvint_{\!\!\!\!C_{0,1}}dy\,\frac{1+y^2}{y^2-1}
    \cpsiASEP(\omega_1,\lambda_a,y;\zeta)\cpsiASEP(\omega_2,\lambda_b,\tfrac1{y};\zeta)\bigg|
    \leq Ce^{-\omega^2/4b}+Ce^{-\omega^2/4b}\int_{\wt
      C_{0,1}}dy\left|\frac{1+y^2}{y^2-1}\right|
  \end{multline}
  where, we recall, $b=\beta\abs{\log(\tau)}$ (and later
  $a=(\frac14+\beta)\abs{\log(\tau)}$). The last integral is finite, and thus this shows
  that the left hand side of \eqref{eq:K11check} is bounded above by a constant multiple
  of
  \begin{align}
    \int_{\rr^2}d\omega_1\,d\omega_2&\,e^{-(\omega_1^2+\omega_2^2)/4b}\sum_{m=1}^\infty
    e^{-(\omega_1+\omega_2+1)m}\frac{e^{-2a m^2}}{(m\taufac)^2}
    \leq\int_{\rr^2}d\omega_1\,d\omega_2\,e^{(\omega_1+\omega_2+1)^2/8a-(\omega_1^2+\omega_2^2)/4b},
  \end{align}
  which as before is finite thanks to the fact that $a>b$. This yields \eqref{eq:K11check}
  and finishes the justification of the desired identity for the first summand of
  $K_{1,1}$. The justification for the second summand is entirely analogous to the
  argument for $\wtKASEP_{1,2}$ above, and thus we omit it. This concludes out proof of
  the equality between the kernels given in \eqref{eq:kernelBessel} and
  \eqref{eq:wtKASEP}.
\end{proof}

\section{Moment formulas for flat KPZ/stochastic heat equation}\label{sec:app-bose}

In this section we will describe moment formulas for the solution of the KPZ equation with
flat initial condition which are obtained in an analogous way as those for ASEP. Versions
of these formulas already appeared in the physics literature in \cite{cal-led}.

The one-dimensional \emph{Kardar-Parisi-Zhang (KPZ) ``equation"} is given by
\begin{equation}
\partial_t h =\tfrac12\partial_x^2 h -\tfrac12 [(\partial_x h)^2 -\infty]  +\xi.
\end{equation}
where $\xi$ is a space-time white noise. This SPDE is ill-posed as written but, at least
on the torus, it can be made sense of by a renormalization procedure introduced by
M. Hairer in \cite{hairer,hairerReg}. His solutions coincide with the Cole-Hopf solution
obtained by setting $h(t,x)=-\log Z(t,x)$, where $Z$ is the unique solution to the
(well-posed) \emph{stochastic heat equation} (SHE)
\begin{equation}\label{eq:SHE}
\partial_t Z = \tfrac12\partial_x^2 Z +\xi Z.
\end{equation}
The formulas which we will obtain will actually be for the moments $\ee^{\rm
  flat}[Z(t,0)^m]$ of the SHE with flat initial data, which means
\[Z(0,x)=1\qquad\text{or}\qquad h(t,x)=0.\] Of course, this means that we are getting
formulas for the exponential moments $\ee^{\rm flat}[e^{-mh(t,0)}]$ of flat KPZ. One can also
think of this in terms of the solution of the \emph{delta Bose gas} with flat initial data, which is the solution
$v(t;\vec x)$ of the following system of equations, where we write $W_k=\{\vec
x\in\rr^k\!:x_1<x_2<\dotsm<x_k\}$ (see \cite{borCor,oqr-half-flat} for more details):
\begin{enumerate}[label=(\arabic*)]
\item For $\vec x\in W_k$,
  \[\p_tv(t;\vec x)=\tfrac12\Delta v(t;\vec x),\]
  where the Laplacian acts on $\vec{x}$.
\item For $\vec x$ on the boundary of $W_k$, with $x_a=x_{a+1}$,
  \[(\p_{x_a}-\p_{x_{a+1}}-1)v(t,\vec x)=0.\]
\item For $\vec x\in W_k$,
  \[\lim_{t\to0}v(t;\vec x)=1.\]
\end{enumerate}
It is widely accepted in the physics literature that if $Z(t,x)$ is a solution of the
SHE, then $v(t;\vec x)=\ee[Z(t,x_1)\dotsm Z(t,x_k)]$ is a solution of the delta Bose gas.  This
fact is proved in \cite{mqrScaling}, where it is also shown that there is at most one
solution. 

In \cite{oqr-half-flat} we obtained an explicit formula for the moments of the SHE started
with half-flat initial condition by taking a weakly asymmetric limit of the half-flat ASEP
moment formula (Theorem \ref{thm:main-hf} above). The same formula was obtained in
\cite{oqr-half-flat} directly for the delta Bose gas with a slightly more general initial
condition using the same method which led in that paper to the moment formulas for
half-flat ASEP.

\begin{prop}
  For $\theta>0$ define (with $x$ appearing $k$ times in the argument)
  \begin{equation}
    \label{Eq:DBGHF-shifted-2}
    v^{\thfl}_\theta(t;x,\dotsc,x)=2^kk!\sum_{\ell=0}^k\inv{\ell!}\sum_{\substack{n_1,\dotsc,n_\ell\ge 1\\n_1+\dotsm+n_\ell=k}}\inv{(2\pi\I)^\ell}
    \idotsint_{(\alpha+\I\rr)^k}d\vec w\,I_\theta(\vec w;\vec n)
  \end{equation}
  with $\alpha>0$ and with $I_\theta$ given by
  \begin{multline}\label{eq:Itheta}
    I_\theta(\vec w,\vec n)=2^k\prod_a\frac{\Gamma(2w_a)}{n_a\Gamma(2w_a+n_a)}
    e^{t\big[\inv{12}n_a^3-\inv{12}n_a+n_a(w_a-\theta)^2\big]+xn_a(w_a-\theta)}\\
    \times\prod_{a<b}\frac{\Gamma(w_a+w_b+\inv2(n_a-n_b))\Gamma(w_a+w_b-\inv2(n_a-n_b))}
    {\Gamma(w_a+w_b-\inv2(n_a+n_b))\Gamma(w_a+w_b+\inv2(n_a+n_b))}\\
    \times\frac{(w_a-w_b+\inv2(n_a-n_b))(w_a-w_b-\inv2(n_a-n_b))}{(w_a-w_b)(w_a-w_b+\inv2(n_a+n_b))}.
  \end{multline}
  Then the solution $v(t;\vec x)$ of the delta Bose gas with the initial data in (3)
  replaced by the tilted half-flat initial condition $\lim_{t\to0}v(t;\vec
  x)=\prod_ae^{-\theta x_a}\uno{x_a\geq0}$ coincides with $v^{\thfl}_\theta(t;x,\dotsc,x)$
  when $\vec x=(x,\dotsc,x)$ (with $x\in\rr$ repeated $k$ times). Moreover, if $Z(t,x)$ is
  the solution of the SHE \eqref{eq:SHE} with half-flat initial condition
  $Z(0,x)=\uno{x\geq0}$, then (denoting by $\ee^\thfl$ the expectation with this initial
  condition)
  \begin{equation}\label{Eq:DBGHF-shifted-2-SHE}
    \ee^\thfl[Z(t,x)^k]=v^{\thfl}_0(t;x,\dotsc,x).
  \end{equation}
\end{prop}

This corresponds to Proposition 5.3 of \cite{oqr-half-flat}. As in the ASEP case, we have
used the Cauchy determinant identity \eqref{eq:cauchydet} to expand the determinant which
appears in that formula.

Just as \eqref{Eq:DBGHF-shifted-2-SHE} was obtained in \cite{oqr-half-flat} as a weakly
asymmetric limit of \eqref{eq:stpt1}, one can take a weakly asymmetric limit of
\eqref{eq:ys-in} to derive a moment formula for the SHE with flat initial
condition $Z(0,x)=1$. Alternatively, one can take the $x\to\infty$ limit directly at the
level of \eqref{Eq:DBGHF-shifted-2}-\eqref{Eq:DBGHF-shifted-2-SHE} to obtain such a
formula. In fact, by statistical translation invariance of the white noise, $Z(t,x)$ with
half-flat initial condition $Z(0,y)=\uno{y\geq0}$ has the same distribution as $Z(t,0)$
with initial condition $Z(0,y)=\uno{y\geq-x}$, and thus computing the $x\to\infty$ limit
of \eqref{Eq:DBGHF-shifted-2} leads to the moments of $Z(t,0)$ with flat initial
condition. This is the approach that we will take here (although we will not provide all
the details). As for the case of ASEP, the calculations involved in the computation of
this limit are quite involved, but they are completely analogous to those of Sections
\ref{sec:flat-lim} and \ref{sec:pfaffian}, as we describe next.

Note that the dependence on $x$ of the right hand side of \eqref{eq:Itheta} is only through
factors of the form $e^{xw_an_a}$, which suggests that we should deform the $w_a$ contours
to a region where $\Re(w_a)<0$. More precisely, we will shift the $w_a$ contours to
$-\delta+\I\rr$ for some small $\delta>0$. The contour deformation involves crossing many
poles, and a careful analysis shows that the pole structure is entirely analogous to the
one described in Section \ref{sec:poles} for ASEP. Now the unpaired poles occur at
$w_a=0$, coming from the factor $\Gamma(2w_a)$, while the paired poles occur at $w_a=-w_b$
in the case $n_a=n_b$, coming from Gamma factors in the numerator of the second line of
\eqref{eq:Itheta}. Crucially, as in the ASEP computation (though in a slightly simpler way), the
dependence on $x$ is gone in all the residues that one has to compute as the contours are
deformed.  Since the free variables will now be integrated on a contour $-\delta+\I\rr$
with $\delta>0$, these factors will vanish as $x\to\infty$, and as a result the limiting
formula will be written only in terms of unpaired and paired variables, analogously to
what was done in Section \ref{sec:flat-lim} for ASEP. Instead of going through the whole
argument again, we will only state the result. Introduce the index set
\[\bar\Lambda^m_{k_1,k_2}=\left\{(\vnu,\vnp)\in\zz_{\geq1}^{k_1}\times\zz_{\geq_1}^{k_2}:\,
  \sum_{a=1}^{k_1}\nu_a+2\sum_{a=1}^{k_2}\np_a=\bk\right\}.\] Then
\begin{equation}
  \label{eq:flat-DBG}
  \ee^{\rm flat}[Z(t,0)^m]=2^mm!\sum_{k=0}^{m}\bnuflat(t)
\end{equation}
with
\[\bnuflat(t)=\sum_{\substack{\ku,\kp\geq0\\\ku+2\kp=k}}\inv{\ku!2^{\kp}\kp!}\bmfZflat(\ku,\kp),\]
where
\begin{equation}\label{eq:bmfZflat}
\begin{split}
&\bmfZflat(\ku,\kp)=\quad\quad\smashoperator{\sum_{(\vnu,\vnp)\in\bar\Lambda^{\bk}_{\ku,\kp}}}\quad\quad
\itwopii{\kp}\int_{(\I\rr)^{\kp}}d\vzp\,\prod_{a\leq\kp,\,b\leq\ku}\frac{(\np_a-\nu_b)^2-4(\zp_a)^2}{(\np_a+\nu_b)^2-4(\zp_a)^2}\\
&\qquad\times\prod_{a=1}^{\kp}\prod_{j=1}^{\np_a}\inv{(\zp_a)^2-j^2}
\prod_{1\leq a<b\leq\kp}\frac{(\np_a-\np_b)^2-4(\zp_a-\zp_b)^2}{(\np_a+\np_b)^2-4(\zp_a-\zp_b)^2}
 \frac{(\np_a-\np_b)^2-4(\zp_a+\zp_b)^2}{(\np_a+\np_b)^2-4(\zp_a+\zp_b)^2}\\
&\qquad\times
\inv{2^{\ku}}\prod_{a=1}^{\ku}\inv{\nu_a!}e^{\inv{12}t((\nu_a)^3-\nu_a)}
\prod_{a=1}^{\kp}\inv{\np_a}e^{t\big[\inv6(\np_a)^3-\inv6\np_a+2\np_a(\zp_a)^2\big]}\prod_{1\leq a<b\leq
  \ku}(-1)^{\nu_a\wedge\nu_b}\frac{|\nu_a-\nu_b|}{\nu_a+\nu_b}.
\end{split}
\end{equation}
This formula should be compared with the moment formula for flat ASEP provided in Theorem
\ref{thm:flatSeries} (in fact, \eqref{eq:flat-DBG} can be obtained as a weakly asymmetric
limit of that formula). Note that in the above discussion the $\zp_a$ contours where
deformed to $-\delta+\I\rr$, but here they have been replaced by $\I\rr$; we may do this
here (i.e. take $\delta\to0$) because the resulting integrand has no poles on
$z_a\in\I\rr$ (this analogous to what was done for ASEP, see the paragraph following
\eqref{eq:mfZflat}).

It is not hard to check that \eqref{eq:bmfZflat} corresponds to (108) in \cite{cal-led}
(with $\ku=M$, $\kp=N$ and setting $s=0$ in their formula, which comes from a scaling
factor which we have not included at this point; note also that in \cite{cal-led} the
replacement $t=4\lambda^3$ has been performed). Their heuristic derivation essentially
consists in studying the poles which arise from taking $\theta\to0$ at the same time as
$x\to\infty$ directly in \eqref{Eq:DBGHF-shifted-2} with $\alpha=\theta$. Although their
procedure is different and not fully justified, the algebraic structure is 
similar. In terms of \eqref{Eq:DBGHF-shifted-2}, they start by shifting all the $w_a$ variables by $\theta$,
which results in the contours turning into $\I\rr$. As they take $\theta\to0$, some of the
factors in the integrand now pass through singularities, which give rise to residues which
are in correspondence with the ones arising from our paired and unpaired poles. Then they
argue that taking $x\to\infty$ at the same time as $\theta\to0$ results in only these
types of terms remain in the limit, leading to the formula.

The final step is to find a Pfaffian representation for the moment formula
\eqref{eq:flat-DBG}. This is done again in an entirely analogous fashion to what was done
for ASEP in Sections \ref{sec:kueven} and \ref{sec:genku} (alternatively, and as we
mentioned, one can take a weakly asymmetric limit of the moment formula in Proposition
\ref{prop:secondPfaffian}, which leads to the same expression). Let us skip the details
and just write the result:
\begin{equation}
  \label{eq:flat-DBG-Pf}
  \ee^{\rm flat}[Z(t,0)^m]=m!\sum_{k=0}^{m}\frac{(-1)^k}{k!}
  \!\!\!\sum_{\substack{m_1,\dotsc,m_k=1\\m_1+\dotsm+m_k=m}}^\infty\int_{(\rr_{\geq0})^k}d\vec\lambda\,
  \pf\!\left[\KCDRP(\lambda_a,\lambda_b;m_a,m_b)\right]_{a,b=1}^k
\end{equation}
with $\KCDRP(\lambda_1,\lambda_2;m_1,m_2)=\left[
  \begin{smallmatrix}
    \KCDRP_{1,1}(\lambda_1,\lambda_2;m_1,m_2) & \KCDRP_{1,2}(\lambda_1,\lambda_2;m_1,m_2)\\
    -\KCDRP_{1,2}(\lambda_2,\lambda_1;m_2,m_1) & \KCDRP_{2,2}(\lambda_1,\lambda_2;m_1,m_2)
  \end{smallmatrix}\right]$ and
\begin{align}
  \KCDRP_{1,1}(\lambda_1,\lambda_2;m_1,m_2)&=\uno{m_1=m_2}(-1)^{m_1+1} \frac{1}{8\pi\I}
  \pvint_{\!\!\!\I\rr} dy\,\frac1{y}\frac{\Gamma(2y)}{\Gamma(m+2y)}\frac{\Gamma(-2y)}{\Gamma(m-2y)}\\
  &\hspace{1.8in}\times e^{[\frac1{96}m_1^3+\frac1{8}m_1y^2]t-\frac14(m_1-2y)\lambda_1-\frac14(m_1+2y)\lambda_2}\\
  &\hspace{-0.4in}+\frac1{32}\frac1{(m_1-1)!}\frac1{(m_2-1)!}(-1)^{m_1\wedge
    m_2}\sgn(m_1-m_2)\,
  e^{\frac1{192}(m_1^3+m_2^3)t-\frac14(m_1\lambda_1+m_2\lambda_2)}\\
  \KCDRP_{1,2}(\lambda_1,\lambda_2;m_1,m_2) &=-\frac18\frac1{(m_1-1)!}\,
  e^{\frac1{192}m_1^3t-\frac14m_1\lambda_1}\\
  \KCDRP_{2,2}(\lambda_1,\lambda_2;m_1,m_2) & = \tfrac12\sgn(\lambda_2-\lambda_1).
\end{align}
This formula should be compared with the the $m$-th term of the (divergent) series written
for $\sum_{m\geq0}\frac{\zeta^n}{n!}\ee^{\rm flat}[Z(t,0)^m]$ (for a specific choice of $\zeta$) in (147) of \cite{cal-led}\footnote{There are two main
  differences between \eqref{eq:flat-DBG-Pf} and the $n$-th term of the formula written in
  \cite{cal-led}, besides some simple scaling and renaming of variables. First, in
  their formula they have applied the Airy trick \eqref{eq:airyTrick} in order to rewrite
  expressions involving $e^{\frac1{192}m^3t}$ in terms of integrals of Airy
  functions. This can be done for \eqref{eq:flat-DBG-Pf} without any trouble (we have not
  done it because it leads to more complicated formulas). Second, one can check that
  obtaining their formula from ours involves (formally) integrating $\KCDRP_{1,1}$ in
  $\lambda_1$ and $\lambda_2$, integrating $\KCDRP_{1,1}$ in $\lambda_1$ and
  differentiating it in $\lambda_2$, and differentiating $\KCDRP_{2,2}$ in both $\lambda_1$
  and $\lambda_2$. This can be justified (formally) using Lemma \ref{lem:pfaffInt}
  (twice). This difference stems from the delta Bose gas analog of the step performed at the
  ASEP level after \eqref{eq:signId}, which is not done in \cite{cal-led}.}.

\appendix

\section{Formal GOE asymptotics for flat ASEP}
\label{sec:goe-limit}

In this section we will provide a non-rigorous argument that shows that, under the scaling
specified in \eqref{eq:exptrick}, the right hand side of \eqref{eq:mainresult} recovers the
GOE Tracy-Widom distribution.

More precisely, we perform a formal asymptotic analysis of  
 $\pf\!\big[J-\wtKASEPz\big]_{L^2([0,\infty))}$ with 
$\zeta=-\tau^{-t/4+t^{1/3}r/2}$ as $t\to \infty$ and obtain
$\pf\!\left[J-K_r\right]_{L^2([0,\infty))}$,
with $K_r$ defined as in \eqref{eq:Kr}. This is the content of Section \ref{sec:limComp},
and should be interpreted as evidence for the
\emph{conjecture}
\[\lim_{t\to\infty}\pp^{\rm
  flat}\!\left(\frac{h(t/(q-p),t^{2/3}x)-\frac12t}{t^{1/3}}\geq-r\right)=\pf\!\left[J-K_r\right]_{L^2([0,\infty))}.\]
The issues discussed in Remark \ref{remark2.2}.vii preclude directly making this part of
the asymptotics rigorous.  However, even the formal critical point analysis has a
complicated algebraic structure which needs to be explained. We will then show,
rigorously, in Section \ref{sec:GOEPfaffian}, that the right hand side recovers the
Tracy-Widom GOE distribution:
\begin{equation}
  \label{eq:pf-GOE}
  \pf\!\left[J-K_r\right]_{L^2([0,\infty))}=F_{\rm GOE}(r).
\end{equation}

\subsection{Mellin-Barnes representation}

There are two alternative ways to derive (formally) the $t\to\infty$ asymptotics of our
formula. The first one consists in starting directly with the formula given in Theorem
\ref{thm:fredPf-intro}, which is written in a form that lends itself readily to critical point
analysis (this is thanks to the rewriting of \eqref{eq:wtKASEP} through
\eqref{eq:DLapInv}, which is analogous to the physicists' Airy trick, see the discussion in page \pageref{disc:airy}
preceding the proof of the theorem). The difficulty with this
approach is that it involves delicate asymptotics of the functions $F_1$, $F_2$ and
$F_3$ (defined in \eqref{eq:DefF}) which are as poorly behaved as  $\exp_\tau$.

A different argument relies on using a Mellin-Barnes representation (see \cite{borCor},
which introduced this idea in this setting) for the sums (in $m$ and $m'$) which appear in
\eqref{eq:wtKASEP}. The advantage of this approach is that the critical point analysis
turns out to be relatively simple and one is not faced with the asymptotics
of the functions $F_1$, $F_2$ and $F_3$. The difficulty, on the other hand, is that
implementing the Mellin-Barnes representation turns out to be much harder than usual, due
to the existence of a large number of additional poles (on top of the ones poles at the
integers which are inherent in the representation) which have to be accounted for. The
fact that all these additional poles end up cancelling is far from trivial and serves as
another indication of the remarkable structure behind the flat ASEP formulas. Partly
because of this reason, this is the approach we will follow presently.

In order to derive the Mellin-Barnes representation of the kernel $\wtKASEPz$ in
\eqref{eq:wtKASEP} we will need to use the following generalization of the
Mellin--Barnes representation formula which appears e.g. in \cite{borCor}:

\begin{lem}
  \label{Lem:MBResult2}
  Let $C_{1,2,\ldots}$ be a negatively oriented contour enclosing all positive
  integers (e.g. $C_{1,2,\dotsc}=\inv2+\I\rr$ oriented with increasing imaginary
  part). Let $g$ be a meromorphic function and let $\mathbb{A}$ be the set of all poles of
  $g$ lying to the right of $C_{1,2\dotsc}$. Assume that
  $\mathbb{A}\cap\zz_{\geq1}=\emptyset$.  Then for $\zeta\in\cc\setminus\rr_{>0}$ with
  $|\zeta|<1$ we have
  \begin{align}
    \label{Eq:MBResult}
    \sum_{n=1}^\infty g(n)\zeta^n=\frac1{2\pi\I}\int_{C_{1,2,\dotsc}}\!ds\,\frac{\pi}{\sin(-\pi
      s)}(-\zeta)^sg(s)-\sum_{w\in\mathbb{A}}\frac{\pi}{\sin(-\pi w)}(-\zeta)^w\Res_{s=w}g(\tau^s)
  \end{align}
  provided that the left hand side converges and that there exist closed contours $C_k,\
  k\in\nn$ enclosing the positive integers from $1$ to $k$ and such that the integral of
  the integrand on the right hand side over the symmetric difference of $C_{1,2,\ldots}$
  and $C_k$ goes to zero as $k\to\infty$.
\end{lem}

The proof is elementary and only uses the fact that $\pi/\sin(-\pi s)$ has poles at each $s=k\in\zz$ with residue equal to
$(-1)^{k+1}$. The (trivial) novelty of this formula is the appearance of a second family
of poles at $\mathbb{A}$.

In order to apply this result to our formula we need to find a different
expression for $\wtKASEPz_{1,1}$, which as written contains factors of the form
$\uno{m_a=m_b}$, $(-1)^{m_a\wedge m_b}$ and $(-1)^{m_a}$ (appearing inside the definition
of $\mfuap(y,m_a)$) which are not suitable for a 
representation in terms of contour integrals. The following result can be proved in a similar way to Lemma
\ref{lem:mfeu}:

\begin{lem}\label{lem:mfeuGammas}
  Given any $m_1\neq m_2\in\zz$ with the same sign,
  \begin{multline}
    \lim_{\eta\to0}\frac{\Gamma(\frac12(m_1-m_2)+\eta)\Gamma(\frac12(m_2-m_1)+\eta)}{\Gamma(\frac12(m_1+m_2)+\eta)\Gamma(-\frac12(m_1+m_2)+\eta)}
    \frac{m_2-m_1+2\eta}{m_1+m_2+2\eta}=(-1)^{m_1\wedge m_2}\sgn(m_2-m_1).
  \end{multline}
\end{lem}

(One can choose slightly different versions of the left hand side, but this form will turn
out to be convenient later). As a consequence, if we define
\begin{equation}
  \label{eq:defmfhh}
  \mfh_\eta(\sigma_1,\sigma_2;m_1,m_2)=
  \begin{dcases*}
    \sigma_2\frac{\Gamma(\frac12(m_1-m_2)+\eta)\Gamma(\frac12(m_2-m_1)+\eta)}{\Gamma(\frac12(m_1+m_2)+\eta)\Gamma(-\frac12(m_1+m_2)+\eta)}
    \frac{m_2-m_1+2\eta}{m_1+m_2+2\eta}
    & if $\sigma_1=\sigma_2$\\
    \sigma_2 & if $\sigma_1\neq \sigma_2$,
  \end{dcases*}
\end{equation}
then we may replace every factor of the form $(-\sigma_a\sigma_b)^{m_a\wedge
  m_b}\sgn(\sigma_bm_b-\sigma_am_a)$ in our formula by
$\lim_{\eta\to0}\mfh_\eta(\sigma_a,\sigma_b;m_a,m_b)$. Something similar can be done
for the product of $\uno{m_a=m_b}$ and the factor $(-1)^{m_a}$ appearing in $\mfuap(y_a,m_a)$, replacing
it by $\lim_{\eta\to0}\mfs_\eta(m_a,m_b)$ with
\begin{equation}
  \label{eq:def-mfs-delta}
  \mfs_\eta(m_1,m_2)=\sin(\tfrac12\pi(2m_1+1))\frac{\sin(\pi(m_1-m_2+\eta))}{\pi(m_1-m_2+\eta)}.
\end{equation}
As we will see in the
next result, the $\eta\to0$ limit can be taken outside the sums and integrals. We state it
at the level of the formula for the $\exp_\tau(x;\xi)$ transform:

\begin{prop}\label{prop:hlem}
  For every $\xi,\zeta\in\cc$ with $|\xi|<\tau^{1/4}$ or $|\xi|=\tau^{1/4}$ and $|\zeta|<\tau^{1/4}$ we have
  \begin{multline}\label{eq:taulaplflat-delta}
    \ee^{\rm flat}\!\left[\exp_\tau(\zeta\tau^{\inv2h(t,0)};\xi)\right]
    =\lim_{\eta\to0}\sum_{k=0}^{\infty}(-1)^{k}\frac{1}{k!}\sum_{m_1,\dotsc,m_{k}=1}^\infty
    \int_{(\rr_{\ge 0})^k}d\vec\lambda\,\\\times\zeta^{\sum_am_a}\tau^{-\frac14(\sum_am_a)^2}\xi^{(\sum_am_a)^2-\sum_am_a}
    \pf\!\left[\KASEPeta(\lambda_a,\lambda_b;m_a,m_b)\right]_{a,b=1}^k
  \end{multline}
  where $\KASEPeta$ is the $2\times 2$ skew-symmetric matrix kernel defined by
  \begin{equation}
    \label{eq:KASEP2}
    \begin{gathered}
      \begin{aligned}
        &\KASEPeta_{1,1}(\lambda_a,\lambda_b;m_a,m_b)=\mfs_\eta(m_a,m_b)\frac{1}{\pi\I}
        \pvint_{\!\!\!C_{0,1}}\!\!\!\!dy\,\tau^{\inv2m_a^2}\mfv(\lambda_a,y,m_a)\mfv(\lambda_b,\tfrac{1}{y},m_a)\tmfuap(y,m_a)\\
        &\qquad+\frac{1}2\sum_{\sigma,\sigma'\in\{-1,1\}}\mfh_\eta(\sigma,\sigma';m_a,m_b)\tau^{\inv4(m_a^2+m_b^2)}
        \mfv(\lambda_a,\sigma,m_a)\mfv(\lambda_b,\sigma',m_b)\mfuu(\sigma,m_a)\mfuu(\sigma',m_b),
      \end{aligned}\\[3pt]
     \KASEPeta_{1,2}=\KASEP_{1,2},\quad\qqand\quad
     \KASEPeta_{2,2}=\KASEP_{2,2},
    \end{gathered}
  \end{equation}
  where
  \begin{equation}
    \label{eq:def-tmfuap}
    \tmfuap(y,m)=(-1)^m\mfuap(y,m)=\tau^{-\inv2n}\frac{1+y^2}{y^2-1}
  \end{equation}
  and $\mfuu$ and $\mfv$ were defined respectively in \eqref{eq:def-mfu} and \eqref{eq:def-mfv-pre}.
\end{prop}

\begin{proof}
  Consider the kernel $\KASEPzero$ defined as $\lim_{\eta\to0}\KASEPeta$. Then it is
  straightforward to check from Theorem \ref{thm:taulaplflat-pre}, using Lemma \ref{lem:mfeuGammas} and the fact that
  $\mfs_\eta(m_1,m_2)\longrightarrow(-1)^{m_1}\uno{m_1=m_2}$ as $\eta\to0$ for $m_1,m_2\in\zz$, that
    \begin{multline}\label{eq:taulaplflat-delta2}
    \ee^{\rm flat}\!\left[\exp_\tau(\zeta\tau^{\inv2h(t,0)};\xi)\right]
    =\sum_{k=0}^{\infty}(-1)^{k}\frac{1}{k!}\sum_{m_1,\dotsc,m_{k}=1}^\infty\itwopii{k}
    \int_{(\rr_{\ge 0})^k}d\vec\lambda\,\\\times\zeta^{\sum_am_a}\tau^{-\inv4(\sum_am_a)^2}\xi^{(\sum_am_a)^2-\sum_am_a}
    \pf\!\left[\KASEPzero(\lambda_a,\lambda_b;m_a,m_b)\right]_{a,b=1}^k.
  \end{multline}
  To justify taking the $\eta\to0$ limit outside use the
  Dominated Convergence Theorem together with an argument similar to the one used in the
  proof of Theorem \ref{thm:taulaplflat-pre} to show that the series on the right hand side of
  \eqref{eq:taulaplflat-delta} is absolutely summable.
\end{proof}

Note that, since $\mfs_\eta(m_a,m_b)$ is replacing the factor $(-1)^{m_a}\uno{m_a=m_b}$,
the first term of $\KASEPeta_{1,1}$ in Proposition \ref{prop:hlem} can be rewritten as
\[\mfs_\eta(m_a,m_b)\zeta^{-m_a-m_b}\frac{1}{\pi\I}
\int_{C_{0,1}}dy\,\int_{0}^\infty
dv\,\zeta^{2m_a}\tau^{\inv2m_a^2}\mfv(\lambda_a,y,m_a)\mfv(\lambda_b,1/y,m_a)\tmfuap(y,m_a)\]
(as can be checked directly in the proof). Using this replacement, setting
$\xi=\tau^{1/4}$ in Proposition \ref{prop:hlem} and bringing the sums in $m_a$ inside the
Pfaffian as in Section \ref{sec:generatingfunction} yields the following formula:
\begin{equation}\label{eq:limEta}
\ee^{\rm flat}\!\left[\exp_\tau\!\big(\zeta\tau^{\inv2h(t,0)}\big)\right]
=\lim_{\eta\to0}\pf\!\left[J-\wtKASEPeta\right]
\end{equation}
with (we omit here the dependence of the kernel on $\zeta$)
\begin{equation}
  \label{eq:wtKASEPeta}
  \begin{split}
      \wtKASEPeta_{1,1}(\lambda,\lambda')&=\sum_{m,m'=1}^\infty\mfs_\eta(m,m')\zeta^{-m-m'}\frac{1}{\pi\I}
      \pvint_{\!\!\!C_{0,1}}\!\!\!\!dy\,\zeta^{2m}\tau^{\inv2m^2}\mfv(\lambda,y,m)\mfv(\lambda',\tfrac{1}{y},m)\tmfuap(y,m)\\
      &\hspace{0.4in}+\tfrac{1}2\sum_{m,m'=1}^\infty\sum_{\sigma,\sigma'\in\{-1,1\}}\mfh_\eta(\sigma,\sigma';m,m')
      \zeta^{m+m'}\tau^{\inv4(m^2+(m')^2)}\\
      &\hspace{2.2in}\times\mfv(\lambda,\sigma,m)\mfv(\lambda',\sigma',m')\mfuu(\sigma,m)\mfuu(\sigma',m'),\\
    \wtKASEPeta_{1,2}(\lambda,\lambda')&=-\tfrac12\sum_{m=1}^\infty\sum_{\sigma\in\{-1,1\}}\sigma\zeta^{m}\tau^{\inv4m^2}
    \mfv(\lambda,\sigma,m)\mfuu(\sigma,m),\\
    \wtKASEPeta_{2,2}(\lambda,\lambda')&=\tfrac12\sgn(\lambda'-\lambda).
  \end{split}
\end{equation}

Up to here we have proceeded rigorously. In what follows we will proceed with a formal
asymptotic analysis.  The first step, which we undertake in the rest of this subsection,
is to use a Mellin-Barnes representation (Lemma \ref{Lem:MBResult2}, however, without
verifying the necessary decay conditions) to argue formally that the right hand side of
\eqref{eq:limEta} yields
\begin{equation}
\pf\!\left[J-\wtKASEPzero\right]\label{eq:wtKASEP0}
\end{equation}
with
\begin{equation}
  \label{eq:wtKASEP0def}
  \begin{split}
    \wtKASEPzero_{1,1}(\lambda,\lambda')&=\itwopii{2}\int_{(c+\I\rr)^2}ds\,ds'\,\frac{\pi^2}{\sin(-\pi
      s)\sin(-\pi s')}\mfs_0(s,s')\\
    &\hspace{1.4in}\times\zeta^{-s-s'}\frac{1}{\pi\I}
    \pvint_{\!\!\!C_{0,1}}\!\!\!\!dy\,(-\zeta)^{2s}\tau^{\inv2s^2}\mfv(\lambda,y,s)\mfv(\lambda',\tfrac{1}{y},s)\tmfuap(y,s)\\
    &\hspace{0.4in}+\inv{2\twopii2}\int_{(c+\I\rr)^2}ds\,ds'\,\frac{\pi^2}{\sin(-\pi
      s)\sin(-\pi s')}\sum_{\sigma,\sigma'\in\{-1,1\}}\mfh_0(\sigma,\sigma';s,s')\\
    &\hspace{1.2in}\times(-\zeta)^{s+s'}\tau^{\inv4(s^2+(s')^2)}\mfv(\lambda,\sigma,s)\mfv(\lambda',\sigma',s')\mfuu(\sigma,s)\mfuu(\sigma',s'),\\
    \wtKASEPzero_{1,2}(\lambda,\lambda')&=-\frac1{4\pi\I}\int_{c+\I\rr}ds\,\frac{\pi}{\sin(-\pi
      s)}\sum_{\sigma\in\{-1,1\}}\sigma(-\zeta)^{s}\tau^{\inv4s^2}
    \mfv(\lambda,\sigma,s)\mfuu(\sigma,s),\\
    \wtKASEPzero_{2,2}(\lambda,\lambda')&=\tfrac12\sgn(\lambda'-\lambda),
  \end{split}
\end{equation}
where $\mfs_0(s,s')=\lim_{\eta\to0}\mfs_\eta(s,s')$ and
$\mfh_0(\sigma,\sigma';s,s')=\lim_{\eta\to0}\mfh_\eta(\sigma,\sigma';s,s')$. Note that
this formula involves using first the Mellin-Barnes representation and then computing the
limit $\eta\to0$. As we will see below, it is only in the limit $\eta\to0$ that the
additional poles which arise in the Mellin-Barnes representation cancel.

In view of \eqref{eq:limEta}, and taking the $\eta\to0$ limit back inside the Pfaffian,
our goal is to argue that $\lim_{\eta\to0}\wtKASEPeta$ is given by $\wtKASEPzero$.
Consider first applying the Mellin-Barnes representation to the kernel
$\wtKASEPeta_{1,2}$. One can check that in this case the only poles of the integrand (in
$s$) are the ones occurring at $s\in\zz_{\geq1}$ coming from $\pi/\sin(-\pi s)$, and thus
the Mellin-Barnes representation can be applied without additional difficulties (in other
words the set $\mathbb{A}$ in Lemma \ref{Lem:MBResult2} is empty in this case). Doing this
and taking $\eta\to0$ yields $\wtKASEPzero_{1,2}$. Note also that $\wtKASEPeta_{2,2}$ does
not depend on $\eta$, and in fact it equals $\wtKASEPzero_{2,2}$, so the equality is direct
in this case.

It remains to handle $\wtKASEPeta_{1,1}$, which is where the additional poles in the
Mellin-Barnes representation arise. Write for convenience
\[\wtKASEPeta_{1,1}=L^\eta_1+L^\eta_2\]
where $L^\eta_1$ and $L^\eta_2$ correspond to each of the two terms appearing in the
definition of $\wtKASEPeta_{1,1}$. For $L^\eta_1$ one checks again that the only poles
occur at $s,s'\in\zz_{\geq1}$ so the Mellin-Barnes representation can be applied as
before, yielding
\begin{multline}\label{eq:L1}
  L^\eta_1=\itwopii{2}\int_{(c+\I\rr)^2}ds\,ds'\,\frac{\pi^2}{\sin(-\pi
    s)\sin(-\pi s')}\mfs_\eta(s,s')\\
  \times(-\zeta)^{-s-s'}\frac{1}{\pi\I}
  \pvint_{\!\!\!C_{0,1}}\!\!\!\!dy\,\zeta^{2s}\tau^{\inv2s^2}\mfv(\lambda,y,s)\mfv(\lambda',\tfrac{1}{y},s)\tmfuap(y,s)
\end{multline}
for $c\in(0,1)$. Taking $\eta\to0$ will yield the desired representation.

Now write $L^\eta_2$ as
\begin{equation}\label{eq:L2def}
  L^\eta_2=\tfrac{1}2\sum_{m,m'=1}^\infty\sum_{\sigma,\sigma'\in\{-1,1\}}\zeta^{m+m'}\mfh_\eta(m,m';\sigma,\sigma')I(m,m';\sigma,\sigma')
\end{equation}
with
$I(m,m';\sigma,\sigma')=\tau^{\inv4(m^2+(m')^2)}\mfv(\lambda,\sigma,m)\mfv(\lambda',\sigma',m')\mfuu(\sigma,m)\mfuu(\sigma',m')$.
Our goal is to get
\begin{multline}\label{eq:L2}
  \lim_{\eta\to0}L^\eta_2=\inv{2\twopii2}\int_{(c+\I\rr)^2}\!\!ds\,ds'\,\frac{\pi^2(-\zeta)^{s+s'}}{\sin(-\pi
    s)\sin(-\pi
    s')}\sum_{\sigma,\sigma'\in\{-1,1\}}\!\!\mfh_0(s,s';\sigma,\sigma')I(s,s';\sigma,\sigma').
\end{multline}
This case is more involved, due to the factor $\mfh_\eta(s,s';\sigma,\sigma')$. One can
check that $I(s,s';\sigma,\sigma')$ is analytic away from $s,s'\in\zz_{\geq1}$, and in
fact the same holds for $\mfh_\eta(s,s';\sigma,\sigma')$ when $\sigma\neq\sigma'$.  Now
consider the case $\sigma=\sigma'$. Let us first fix $m$ and consider the effect of
applying Lemma \ref{Lem:MBResult2} to the $m'$ sum in \eqref{eq:L2def}.  We need to
analyze the poles (in $s'$) of
\[\mfh_\eta(m,s';\sigma,\sigma)=
\sigma\frac{\Gamma(\inv2(m-s')+\eta)\Gamma(\inv2(s'-m)+\eta)}{\Gamma(\inv2(m+s')+\eta)\Gamma(-\inv2(m+s')+\eta)}
\frac{s'-m+2\eta}{m+s'+2\eta}.\] The numerator has singularities when $s'=m+2\ell+2\eta$
and $s'=m-2\ell-2\eta$ for $\ell\in\zz_{\geq0}$. The first type of singularity is
removable, because at that point the factor $1/\Gamma(-\inv2(m+s')+\eta)$ evaluates to
$1/\Gamma(-m-\ell)=0$. The second type of singularity is, on the other hand, a pole, with
residue
$2\sigma\frac{(-1)^\ell}{\ell!}\frac{\Gamma(\ell+2\eta)}{\Gamma(m-\ell)\Gamma(\ell-m+2\eta)}
\frac{\ell+2\eta}{\ell-m}$, and which belongs to the deformation region (namely
$\{z\in\cc\!:\Re(z)\geq\inv2\}$) only when $\ell<m/2$ (assuming that $\eta$ is small). In
view of this, Lemma \ref{Lem:MBResult2} suggests that
\begin{align}
  L^\eta_2&=\sum_{m=1}^\infty\inv{4\pi\I}\int_{c+\I\rr}ds'\,\frac{\pi}{\sin(-\pi s')}
  \sum_{\sigma,\sigma'\in\{-1,1\}}\zeta^m(-\zeta)^{s'}\mfh_\eta(\sigma,\sigma';m,s')I(\sigma,\sigma';m,s')\\
  &\qquad-\sum_{m=1}^\infty\sum_{\ell=0}^{\lfloor
    m/2\rfloor}\sum_{\sigma\in\{-1,1\}}\sigma
  \frac{(-1)^\ell}{\ell!}\frac{\Gamma(\ell+2\eta)}{\Gamma(m-\ell)\Gamma(\ell-m+2\eta)}\frac{\ell-2\eta}{\ell-m}
  \zeta^m(-\zeta)^{m-2\ell-2\eta}\\
  &\hspace{2in}\times\frac{\pi}{\sin(\pi(2\eta+2\ell-m))}I(\sigma,\sigma';m,m-2\ell-2\eta)\\
  &:=L^\eta_{2,1}-L^\eta_{2,2}.
\end{align}
Here we are choosing the contour $C_{1,2,\dotsc}$ in Lemma \ref{Lem:MBResult2} to be
$c+\I\rr$. Now we need to apply the Mellin-Barnes representation to the
remaining sum (in $m$) appearing in $L^\eta_{2,1}$, regarding $s'$ as fixed. As before
most of the factors making up the integrand are analytic, except for
$\mfh_\eta(s,s';\sigma,\sigma')$ when $\sigma=\sigma'$, in which case it reads
\[\mfh_\eta(s,s';\sigma,\sigma)=\sigma\frac{\Gamma(\inv2(s-s')+\eta)\Gamma(\inv2(s'-s)+\eta)}
{\Gamma(\inv2(s+s')+\eta)\Gamma(-\inv2(s+s')+\eta)} \frac{s'-s+2\eta}{s+s'+2\eta}.\] The
numerator has singularities when $s=s'-2\ell-2\eta$ and $s=s'+2\ell+2\eta$ for
$\ell\in\zz_{\geq0}$. Note that the singularities of the first type are never in the
deformation region, so we are only left with the second type of singularities, which lie
in the deformation region for all $\ell\geq0$, with residue
$2\sigma\frac{(-1)^{\ell}}{\ell!}\frac{\Gamma(\ell+2\eta)}{\Gamma(s'+\ell+2\eta)\Gamma(-s'-\ell)}
\frac{\ell}{s'+\ell+2\eta}$. Reasoning as before we get
\begin{align}
  L^\eta_{2,1}&=\sum_{\sigma,\sigma'\in\{-1,1\}}\inv{2\twopii2}\int_{(c+\I\rr)^2}ds\,ds'\,\frac{\pi^2(-\zeta)^{s+s'}}{\sin(-\pi
    s)\sin(-\pi s')}
  \mfh_\eta(s,s';\sigma,\sigma')I(s,s';\sigma,\sigma')\\
  &\qquad-\sum_{\ell\geq0}\sum_{\sigma\in\{-1,1\}}\inv{2\pi\I}\int_{c+\I\rr}ds'\,\sigma\frac{(-1)^{\ell}}{\ell!}\frac{\Gamma(\ell+2\eta)}
  {\Gamma(s'+\ell+2\eta)\Gamma(-s'-\ell)}\frac{\ell}{s'+\ell+2\eta}\\
  &\hspace{0.6in}\times\frac{\pi^2(-\zeta)^{2s'+2\ell+2\eta}}{\sin(-\pi(s'+2\ell+2\eta))\sin(-\pi s')}I(s'+2\ell+2\eta,s';\sigma,\sigma)\\
  &:=L^\eta_{2,3}-L^\eta_{2,4}.
\end{align}

In terms of the above kernels we have
$\wtKASEPeta_{1,2}=L^\eta_1+L^\eta_{2,3}-L^\eta_{2,2}-L^\eta_{2,4}$. Observe now, using
\eqref{eq:L1} and the last equation, that the limit as $\eta\to0$ of
$L^\eta_1+L^\eta_{2,3}$ yields exactly $\wtKASEPzero_{1,1}$ (given in
\eqref{eq:wtKASEP0def}), so all that remains to show is that
$L^\eta_{2,2}+L^\eta_{2,4}\longrightarrow0$ as $\eta\to0$. For the first term, noting that the zero of
the sine in the denominator cancels with a zero coming from one of the Gamma functions, so
that
$\frac{\pi}{\sin(\pi(2\eta+2\ell-m))\Gamma(\ell-m+2\eta)}\longrightarrow(-1)^{\ell}(m-\ell)!$,
we have:
\begin{equation}
  \label{eq:L22lim}
  \begin{split}
    \lim_{\eta\to0}L^\eta_{2,2}&=-\sum_{m=1}^\infty\sum_{\ell=0}^{\lfloor m/2\rfloor}
    \sum_{\sigma\in\{-1,1\}}\sigma(-1)^m\zeta^{2m-2\ell}I(\sigma,\sigma;m,m-2\ell)\\
    &=-\sum_{\ell=0}^{\infty}\sum_{m=0}^\infty
    \sum_{\sigma\in\{-1,1\}}\sigma(-1)^m\zeta^{2m+2\ell}I(\sigma,\sigma;m+2\ell,m)
  \end{split}
\end{equation}
where, for convenience, in the second equality we have added the term with $\ell=m=0$
which is 0 anyway because $\mfuu(\sigma,0)$ (which appears in $I(\sigma,\sigma;0,0)$) is
so.  On the other hand we have
\begin{multline}
  \lim_{\eta\to0}L^\eta_{2,4}=\sum_{\ell\geq0}\sum_{\sigma\in\{-1,1\}}\inv{2\pi\I}\int_{c+\I\rr}ds'\,\sigma\frac{(-1)^{\ell}}{\ell!}\frac{\Gamma(\ell)}
  {\Gamma(s'+\ell)\Gamma(-s'-\ell)}\frac{\ell}{s'+\ell}\\
  \times\frac{\pi^2(-\zeta)^{2s'+2\ell}}{\sin(-\pi(s'+2\ell))\sin(-\pi
    s')}I(s'+2\ell,s';\sigma,\sigma).
\end{multline}
The $s'$ integral can be computed in terms of the residues of the integrand for
$\Re(s')\geq\inv2$. There is a double zero in the denominator coming from the sine factors
when $s'=m'\in\zz_{\geq1}$, but one of them is canceled by the zero of
$1/\Gamma(-s'-\ell)$ at these points, resulting in a simple pole, with
$\Res_{s'=m}\pi^2/[\sin(-\pi(s'+2\ell))\sin(-\pi
s')\Gamma(-s'-\ell)]=(-1)^{m'+\ell}(m'+\ell)!$. We get
\[\lim_{\eta\to0}L^\eta_{2,4}=\sum_{\ell\geq0}\sum_{m'\geq1}\sum_{\sigma\in\{-1,1\}}\,
\sigma(-1)^{m'}\zeta^{2m'+2\ell}\tau^{\inv4((m'+2\ell)^2+(m')^2)}
I(\sigma,\sigma;m'+2\ell,m').\] As before we may add the term with $\ell=m'=0$, and now
comparing with \eqref{eq:L22lim} we see that
$\lim_{\eta\to0}(L^\eta_{2,2}+L^\eta_{2,4})=0$ as desired, which finishes our derivation
of \eqref{eq:wtKASEP0}.

\subsection{Computation of the limit}\label{sec:limComp}

We are finally in position to compute the $t\to\infty$ asymptotics of the flat ASEP
distribution function. We take $\zeta=-\tau^{-\inv4t+\inv2t^{1/3}r}$ and, in view
\eqref{eq:wtKASEP0}, we need to compute the limit of $\pf[J-\wtKASEPzero]$, with
$\wtKASEPzero$ given in \eqref{eq:wtKASEP0def}.
We will only provide a formal critical point derivation of the limit. To that end we study first the factors
in $\wtKASEPzero(s,s')$ which depend on $t$. They come from the products of the form
$\tau^{(-\inv4t+\inv2t^{1/3}r)s}\mfv(\lambda,y,s)$, and are given by
\[\exp\!\left(t\left[\tfrac{1}{1+\tau^{-s/2}y}-\tfrac{1}{1+\tau^{s/2}y}-\tfrac14\log(\tau)s\right]+\tfrac12t^{1/3}rs\log(\tau)\right)\]
(this factor appears twice in $\wtKASEPzero_{1,1}(\lambda,\lambda';s,s')$, once with
respect to $s$ and once with respect to $s'$, and once in
$\wtKASEPzero_{1,2}(\lambda,\lambda';s,s')$). Consider now the function
$f(s,y)=\frac{1}{1+\tau^{-s/2}y}-\frac{1}{1+\tau^{s/2}y}-\inv4\log(\tau)s$. One checks
that $f$ has a critical point at $(0,1)$ and, moreover, that the Hessian of this function
vanishes at this point. This suggests that we should use a $t^{-1/3}$ scaling around these
points. Explicitly, we will rescale as follows:
\[y=1+t^{-1/3}\ty,\qquad s=\tfrac{-1}{\log(\tau)}t^{-1/3}\ts\qqand
\lambda=t^{1/3}(\tl-2r)\]
(and similarly for $\lambda'$ and $s'$). The $\lambda$ scaling is so that $\lambda\frac{1-\tau^{s/2}y}{1+\tau^{s/2}y}$ is
order 1 (the shift by $2r$ is for convenience). The $\lambda$ change of variables produces
a $t^{1/3}$ in front of the kernel $\wtKASEPzero$, and using \eqref{eq:diag2Pf} this can
be rewritten as
\begin{equation}
t^{1/3}\wtKASEPzero(\lambda,\lambda')=\left[
  \begin{smallmatrix}
    t^{2/3}\wtKASEPzero_{1,1}(\lambda,\lambda') & t^{1/3}\wtKASEPzero_{1,2}(\lambda,\lambda')\\
    -t^{1/3}\wtKASEPzero_{1,2}(\lambda',\lambda) & \wtKASEPzero_{2,2}(\lambda,\lambda')
  \end{smallmatrix}\right].\label{eq:kernelSc}
\end{equation}
With this scaling we have 
\begin{gather}
  t\left[\tfrac{1}{1+\tau^{-s/2}y}-\tfrac1{1+\tau^{s/2}y}
    -\tfrac14\log(\tau)s\right]+\tfrac12rt^{1/3}\log(\tau)s
  \approx\tfrac1{192}\ts^3+\tfrac1{16}\ts\ty^2,\\
  \tfrac{1-\tau^{s/2}y}{1+\tau^{s/2}y}\approx
  \tfrac14(\ts-2\ty)t^{-1/3},\quad\tau^{-s/2}\tfrac{1+y^2}{y^2-1}\,dy\approx\tfrac{1}{\ty}\,d\ty,
  \quad t^{-1/3}\tfrac{\pi}{\sin(-\pi s)}\,ds\approx-\tfrac1{\ts}\,d\ts,\\
  \sin(\tfrac{\pi}{2}(2s_1+1))\tfrac{\sin(\pi(s_1-s_2))}{\pi(s_1-s_2)}\approx1,\qquad
  \tfrac{\Gamma(\inv2(s_1-s_2))\Gamma(\inv2(s_2-s_1))}{\Gamma(\inv2(s_1+s_2))\Gamma(-\inv2(s_1+s_2))}\tfrac{s_2-s_1}{s_1+s_2}
  \approx\tfrac{\ts_1+\ts_2}{\ts_2-\ts_1}
\end{gather}
and, furthermore, using Lemma A.1 in \cite{oqr-half-flat},
\[\tfrac{\pochinf{-\tau^{-s/2}y;\tau}}{\pochinf{-\tau^{s/2}y;\tau}}\approx1,
\qqand\tfrac{\pochinf{\tau^{1+s}y^2;\tau}}{\pochinf{\tau y^2;\tau}} \approx1.\] Using
these asymptotics and the definition of $\mfv$ in \eqref{eq:def-mfv-pre}, we get
\begin{align}
  \mfv(\lambda,y,s)&\approx\tfrac14t^{-1/3}(\ts-2\ty)
  e^{\frac1{192}\ts^3+\frac1{16}\ts\ty^2-\frac14(\ts-2\ty)\tl+\inv2\ty r}\\
  &=-t^{-1/3}\p_{\lambda}
  e^{\frac1{192}\ts^3+\frac1{16}\ts\ty^2-\frac14(\ts-2\ty)\tl+\inv2\ty r}.
\end{align}
When $y$ is replaced by $1/y$ the above asymptotics hold with $\ty$
replaced by $-\ty$. At the same time, if $y=1$ then the asymptotics for $\mfv$ hold with
$\ty=0$ while if $y=-1$ the right hand side should be replaced by 0. The above asymptotics
also give, in view of the definition of $\mfuu$ in \eqref{eq:def-mfu} and of $\tmfuap$ in \eqref{eq:def-tmfuap},
\[\mfuu(\pm1,s)\approx1\qqand\tmfuap(y,s)dy\approx\tfrac{1}{\ty}d\ty.\]
Using all this in \eqref{eq:wtKASEP0def}, choosing the $s$ and $s'$ contours to be
$\inv2t^{-1/3}+\I\rr$, and keeping \eqref{eq:kernelSc} in mind we deduce that
\begin{equation}\label{eq:Kbar}
  \lim_{t\to\infty}\wtKASEPzero=\bar K
\end{equation}
with
\begin{equation}
  \begin{split}
    \bar K_{1,1}(\tl,\tl')&\approx\p_{\tl}\p_{\tl'}\tfrac1{\pi\I}\pvint_{\!\!\!\!\I\rr}d\ty\,\tfrac1{\twopii2}\pvint_{\!\!\!\!(\inv2+\I\rr)^2}d\ts\,d\ts'\,\tfrac1{\ts\ts'}
    \tfrac{1}{\ty}e^{\frac1{96}\ts^3+\frac1{8}\ts\ty^2-\frac14\ts(\tl+\tl')+\inv2(\tl-\tl')\ty}\\
    &\qquad\qquad+\p_{\tl}\p_{\tl'}\tfrac1{\twopii2}\pvint_{\!\!\!\!(\inv2+\I\rr)^2}d\ts\,d\ts'\,\tfrac{\ts+\ts'}{2\ts\ts'(\ts'-\ts)}e^{\frac1{192}(\ts^3+\ts'^3)-\inv{4}(\ts\tl+\ts'\tl')},\\
  \bar K_{1,2}(\tl,\tl')&\approx-\tfrac12\p_{\tl}\tfrac1{2\pi\I}\int_{\inv2+\I\rr}d\ts\,e^{\inv{192}\ts^3-\inv4\ts\tl},\\
  \bar K_{2,2}(\tl,\tl')&=\tfrac12\sgn(\tl'-\tl).
\end{split}
\end{equation}

What remains is to rewrite $\bar K$ in a more convenient way (this part bears some
similarities with Appendix K of \cite{cal-led}). Let us write
$\bar{K}_{1,1}=\bar{K}_{1,1}^1+\bar{K}_{1,1}^2$.  Note that $\bar{K}^1_{1,1}$ depends on
$\ts'$ only through the integral $\inv{2\pi\I}\pvint_{\!\inv2+\I\rr}d\ts'\,\inv{\ts'}$
which, as a principal value integral, equals $\inv2$. Then, using the definition of the
Airy function,
\begin{equation}
\Ai(z)=\inv{2\pi\I}\int_{c+\I\rr}du\,e^{\inv3u^3-uz}\label{eq:airy}
\end{equation}
for $c>0$, we get (removing the tildes)
\begin{align}
  \bar{K}_{1,1}^1(\lambda_1,\lambda_2)&=\p_{\lambda_1}\p_{\lambda_2}\tfrac1{\pi\I}\pvint_{\!\!\!\!\I\rr}dy\,\tfrac1{2\pi\I}\int_{c+\I\rr}ds_1
   \,\tfrac{1}{2s_1y}e^{\frac1{96}s_1^3+\frac1{8}s_1y^2-\frac14(s_1-2y)\lambda_1-\inv4(s_1+2y)\lambda_2}\\
   &=\p_{\lambda_1}\p_{\lambda_2}\tfrac1{2\pi\I}\pvint_{\!\!\!\!\I\rr}dy\int_0^\infty\!\!du\,\tfrac1{2\pi\I}\int_{c+\I\rr}\!\!ds_1
   \,\tfrac{1}{y}e^{\frac1{96}s_1^3+\frac1{8}s_1y^2-\frac14(s_1-2y)\lambda_1-\inv4(s_1+2y)\lambda_2-s_1u}\\
   &=\p_{\lambda_1}\p_{\lambda_2}\tfrac{1}{2\pi\I}\pvint_{\!\!\!\!\rr}dy\,\int_0^\infty du
   \,\tfrac{1}{y}\Ai(2^{2/3}(\tfrac12(\lambda_1+\lambda_2)+y^2)+u)e^{\I(\lambda_2-\lambda_1)y}\\
   &=\p_{\lambda_2}\tfrac{1}{2^{4/3}\pi\I}\pvint_{\!\!\!\!\rr}dy\,\int_0^\infty du
   \,\tfrac{1}{y}\Ai'(2^{2/3}(\tfrac12(\lambda_1+\lambda_2)+y^2)+u)e^{\I(\lambda_2-\lambda_1)y}\\
   &\qquad\qquad-\p_{\lambda_2}\tfrac{1}{2\pi}\pvint_{\!\!\!\!\rr}dy\int_0^\infty du
   \Ai(2^{2/3}(\tfrac12(\lambda_1+\lambda_2)+y^2)+u)e^{\I(\lambda_2-\lambda_1)y}\\
   &=-\p_{\lambda_2}\tfrac{1}{2^{7/3}\pi\I}\int_{\lambda_1-\lambda_2}^{\lambda_2-\lambda_1}d\eta\,\pvint_{\!\!\!\!\rr}dy
   \,\Ai(2^{2/3}(\tfrac12(\lambda_1+\lambda_2)+y^2))e^{\I\eta y}\\
   &\qquad\qquad-\p_{\lambda_2}\tfrac{1}{2\pi}\pvint_{\!\!\!\!\rr}dy\int_0^\infty du
   \Ai(2^{2/3}(\tfrac12(\lambda_1+\lambda_2)+y^2+2^{-2/3}u))e^{\I(\lambda_2-\lambda_1)y}.
 \end{align}
 As in Appendix J of \cite{cal-led}, we will use the following identity, proved in \cite{valleeSoaresIzarra}:
 \[\tfrac1{2^{1/3}\pi}\int_{-\infty}^{\infty}dy\,\Ai\!\big(2^{2/3}(y^2+\tfrac12(a+b))\big)e^{\I
   (a-b)y}=\Ai(a)\Ai(b).\]
 It implies that
 \begin{equation}
 \begin{aligned}
  \bar{K}_{1,1}^1(\lambda_1,\lambda_2)&=-\p_{\lambda_2}\tfrac1{4}\int_{\lambda_1-\lambda_2}^{\lambda_2-\lambda_1}d\eta
   \Ai\!\big(\tfrac12(\lambda_1+\lambda_2+\eta)\big)\Ai\!\big(\tfrac12(\lambda_1+\lambda_2-\eta)\big)\\
   &\hspace{1.4in}-\p_{\lambda_2}\tfrac{1}{2^{2/3}}\int_0^\infty du\Ai(\lambda_1+2^{-2/3}u)\Ai(\lambda_2+2^{-2/3}u)\\
   &=-\tfrac12\int_{\lambda_1}^{\lambda_2}d\eta\Ai'(\eta)\Ai(\lambda_1+\lambda_2-\eta)-\tfrac12\Ai(\lambda_1)\Ai(\lambda_2)-\p_{\lambda_2}\K(\lambda_1,\lambda_2),
 \end{aligned}\label{eq:barK11}
\end{equation}
where $\K$ was defined in \eqref{eq:airyKernel}.
 
Now we turn to $\bar{K}^2_{1,1}$, which is given by
\begin{multline}
  \bar{K}^2_{1,1}(\lambda_1,\lambda_2)=\p_{\lambda_1}\p_{\lambda_2}\tfrac1{2\twopii{2}}\pvint_{\!\!\!\!(c+\I\rr)^2}\!\!ds_1\,ds_2
  \tfrac{s_1+s_2}{s_1s_2(s_2-s_1)}e^{\inv{192}(s_1^3+s_2^3)-\inv4(s_1\lambda_1+s_2\lambda_2)}\\
  =\p_{\lambda_1}\p_{\lambda_2}\tfrac1{2\twopii{2}}\int_{2c+\I\rr}\!\!ds_1\int_{c+\I\rr}\!\!ds_2
  \tfrac{1}{s_2(s_2-s_1)}e^{\inv{192}(s_1^3+s_2^3)-\inv4(s_1\lambda_1+s_2\lambda_2)}-(\lambda_1\longleftrightarrow\lambda_2)
\end{multline}
where we have shifted the $s_2$ contour to $2c+\I\rr$ for convenience and where
$(\lambda_1\longleftrightarrow\lambda_2)$ denotes the same as the previous expression with
$\lambda_1$ and $\lambda_2$ interchanged. The first term on the right hand side can be
rewritten as
\begin{align}
  &-\p_{\lambda_1}\p_{\lambda_2}\tfrac12\tfrac1{\twopii{2}}\int_{2c+\I\rr}\!\!ds_1\int_{c+\I\rr}\!\!ds_2
  \int_{0}^\infty du_1\int_0^{\infty} du_2\,
  e^{\inv{192}(s_1^3+s_2^3)-\inv4(s_1\lambda_1+s_2\lambda_2)-(s_1-s_2)u_1-s_2u_2}\\
  &\qquad=-\p_{\lambda_1}\p_{\lambda_2}\tfrac12\tfrac1{\twopii{2}}\int_{2c+\I\rr}\!\!ds_1\int_{c+\I\rr}\!\!ds_2
  \int_{0}^\infty dv_1\int_{-v_1}^{\infty} dv_2\,e^{\inv{192}(s_1^3+s_2^3)-\inv4(s_1\lambda_1+s_2\lambda_2)-s_1v_1-s_2v_2}\\
  &\qquad=-8\,\p_{\lambda_1}\p_{\lambda_2}\int_0^{\infty}dv_1\int_{-v_1}^\infty
  dv_2\Ai(\lambda_1+4v_1)\!\Ai(\lambda_2+4v_2)\\
  &\qquad=\tfrac12\int_0^{\infty}dv\Ai'(\lambda_1+v)\Ai(\lambda_2-v)=\tfrac12\int_{\lambda_1}^{\infty}dv\Ai'(v)\Ai(\lambda_1+\lambda_2-v),
\end{align}
where we have used again \eqref{eq:airy}. Subtracting the term with $\lambda_1$ and $\lambda_2$ flipped we get
  \[\bar{K}_{1,1}^2(\lambda_1,\lambda_2)=\tfrac12\int_{\lambda_1}^{\lambda_2}\,d\omega\Ai'(\omega)\Ai(\lambda_1+\lambda_2-\omega).\]
Note that this term cancels the first term on the right hand side of \eqref{eq:barK11},
and thus we obtain
\begin{equation}
  \label{eq:K11}
  \bar{K}_{1,1}(\lambda_1,\lambda_2)=-\p_{\lambda_2}\K(\lambda_1,\lambda_2)-\tfrac12\Ai(\lambda_1)\Ai(\lambda_2)
  =\tfrac12(\p_{\lambda_1}-\p_{\lambda_2})\K(\lambda_1,\lambda_2),
\end{equation}
where the second equality follows by integration by parts.

$\bar{K}_{1,2}$ is simpler to obtain: using \eqref{eq:airy} one more time,
\[\bar{K}_{1,2}(\lambda_1,\lambda_2)=-\tfrac1{16\pi\I}\int_{c+\I\rr}ds\,e^{\inv{192}s^3-\inv4s\lambda_1}
   =-\tfrac12\Ai(\lambda_1).\]
 This, together with \eqref{eq:K11}, gives
 \[\bar K(\lambda_1,\lambda_2)=\left[
    \begin{matrix}
      \tfrac12(\p_{\lambda_1}-\p_{\lambda_2})\K(\lambda_1,\lambda_2) & -\tfrac12\Ai(\lambda_1)\\
      \tfrac12\Ai(\lambda_2) & \tfrac12\sgn(\lambda_2-\lambda_1)
    \end{matrix}\right].\]
  In view of \eqref{eq:wtKASEP0}, \eqref{eq:Kbar} and the definition of $K_r$ in
  \eqref{eq:Kr} (shifting $\lambda_1\mapsto\lambda_1+r$, $\lambda_2\mapsto\lambda_2+r$
  in the Pfaffian as well), this finishes our formal asymptotic analysis of
  $\pf\!\big[J-\wtKASEPz\big]$, yielding $\pf\!\big[J-K_r]$.

 \subsection{Fredholm Pfaffian formula for GOE}\label{sec:GOEPfaffian}

 All that remains is to complete our formal derivation of the GOE asymptotics for flat
 ASEP is to verify that the limiting Fredholm Pfaffian is indeed a formula for the
 Tracy-Widom GOE distribution. This fact can be proved rigorously:
 
\begin{prop}\label{prop:GOEPfaffian}
  \begin{equation}
    \label{eq:GOE-pfaffian}
    F_{\rm GOE}(r)=\pf\!\left[J-K_r\right]_{L^2([0,\infty))},
  \end{equation}
  with
  \[K_r(\lambda_1,\lambda_2)=\left[
    \begin{matrix}
     \tfrac12(\p_{\lambda_1}-\p_{\lambda_2})\K(\lambda_1+r,\lambda_2+r) & -\tfrac12\Ai(\lambda_1+r)\\
      \tfrac12\Ai(\lambda_2+r) & \tfrac12\sgn(\lambda_2-\lambda_1)
    \end{matrix}\right].\]
\end{prop}

A very similar formula for $F_{\rm GOE}$ appears in \cite{ferrariPolyGOE}. Verifying that
the Fredholm Pfaffian on the right hand side of \eqref{eq:GOE-pfaffian} is convergent is
not hard. In fact, it is enough to note that the kernel $K_r$ is uniformly bounded for
$\lambda_1,\lambda_2\geq 0$, and then expand the Pfaffian as a Fredholm series and use the
identity $\pf[A]^2=\det(A)$ together with Hadamard's bound. 

In view of the discussion about Fredholm Pfaffians contained in Appendix \ref{sec:fredPf}
(see in particular \eqref{eq:pfConjugate}), it would be nice if one could find a
symplectic $2\times2$-matrix kernel $M$ so that $M^{\sf T}K_rM$ is trace
class. Unfortunately, it does not seem like such a kernel exists. More precisely, at least
in the Fredholm determinant case one is usually led to consider multiplication operators
for conjugation (see \eqref{eq:conjugate}). In the symplectic case such an operator would
be of the form $M\binom{f_1}{f_2}(x)=\binom{\phi(x)f_1(x))}{\phi(x)^{-1}f_2(x)}$ for some
(non-vanishing) function $\phi$. But it is proved in page 82 of \cite{deiftGioev} that
there is no choice of $\phi$ for which $M^{\sf T}K_rM$ is trace class (the underlying
problem is the singularity at the diagonal of $\sgn(\lambda_2-\lambda_1)$). As mentioned
at the end of Appendix \ref{sec:fredPf}, we can still use the relation between Fredholm
Pfaffians and Fredholm determinants since we know that the Fredholm Pfaffian in
\eqref{eq:GOE-pfaffian} and the associated Fredholm determinant (see below) define
absolutely convergent series. This is what we will do in the proof, which will then be
devoted to showing that the resulting Fredholm determinant yields $F_{\rm GOE}$. The proof
follows the arguments of Section 7 of \cite{cal-led} relatively closely.

\begin{proof}[Proof of Proposition \ref{prop:GOEPfaffian}]
  For convenience will omit the subscript $r$ in $K_r$ during this proof. We will also
  omit the subscripts from Fredholm determinants and Pfaffians, which are
  always computed on $L^2([0,\infty))$ or $L^2([0,\infty))\otimes L^2([0,\infty))$.

  We have shown already that the Fredholm Pfaffian series for $\pf[J-K]$ is absolutely
  convergent. An identical argument shows that $\det[I+JK]$ satisfies the same, and thus
  Proposition \ref{prop:fredDetPf} implies that
  \begin{equation}
      \pf[J-K]^2=\det[I+JK]=\det\!\left[
      \begin{smallmatrix}
        I-K^\mathsf{T}_{1,2} & K_{2,2}\\-K_{1,1} & I-K_{1,2}
      \end{smallmatrix}\right].
    \label{eq:pf2}
  \end{equation}  
  In view of \eqref{eq:K11}, let us write
  \[K_{1,1}=K_{1,1}^a+K_{1,1}^b\] with
  \[K_{1,1}^a(\lambda_1,\lambda_2)=-\p_{\lambda_2}\K(\lambda_1+r,\lambda_2+r)\qand
  K_{1,1}^2(\lambda_1,\lambda_2)=-\tfrac12\Ai(\lambda_1+r)\Ai(\lambda_2+r).\] Note that
  $K_{1,1}^b$ is a symmetric, rank-one kernel. We claim that $K_{1,1}$ can be replaced by
  $K_{1,1}^a$ on the right hand side of \eqref{eq:pf2}:
  \begin{equation}
    \pf[J-K]^2=\det\!\left[
      \begin{smallmatrix}
        I-K^\mathsf{T}_{1,2} & K_{2,2}\\-K^a_{1,1} & I-K_{1,2}
      \end{smallmatrix}\right].\label{eq:pf3}
  \end{equation}
  It is enough to check this at the level of the finite dimensional determinants which appear in
  Fredholm determinant series \eqref{eq:fredDet} for $\det[I+JK]$, in which
  case it follows from the following general fact: if $A,B,C,U$ are $n\times n$ (real)
  matrices with $B$ and $C$ skew-symmetric and $U$ symmetric and rank-one, then the
  matrices
  \begin{equation}
    \left[
      \begin{matrix}
        A & B \\ C+U & A^{\mathsf{T}}
      \end{matrix}\right]\qqand\left[
      \begin{matrix}
        A & B \\ C & A^{\mathsf{T}}
      \end{matrix}\right]\label{eq:det-lem}
  \end{equation}
  have the same eigenvalues (and eigenvectors) To see this, suppose that $(v_1,v_2)$ is an
  eigenvector of $\left[
    \begin{smallmatrix}
      A & B \\ C+U & A^{\mathsf{T}}
    \end{smallmatrix}\right]$, where $v_i\in \rr^n$.  We claim that $Uv_1=0$, which
  implies that $(v_1,v_2)$ is also an eigenvector of $\left[
    \begin{smallmatrix}
      A & B \\ C & A^{\mathsf{T}}
    \end{smallmatrix}\right]$ with the same eigenvalue, say $\lambda$. In fact, we have
  $Av_1+Bv_2= \lambda v_1$ and $(C+U) v_1 + A^\mathsf{T} v_2 = \lambda v_2$.  Now test the
  first equation on the right with $v_2$ and the second on the left with $v_1$ to get
  $\langle Av_1,v_2\rangle + \langle Bv_2,v_2\rangle = \lambda \langle v_1,v_2\rangle$ and
  $\langle v_1, (C+U) v_1\rangle + \langle v_1, A^\mathsf{T}T v_2\rangle = \lambda\langle
  v_1,v_2\rangle$.  By skew-symmetry we have $\langle Bv_2,v_2\rangle = \langle v_1, C
  v_1\rangle =0$, so we conclude from the two equations that $\langle v_1, U
  v_1\rangle=0$.  Now a symmetric rank one matrix must be a multiple of a projection onto
  a vector $w$ and we conclude that $\langle v_1, w\rangle =0$ which further implies that
  $Uv_1=0$. This proves our claim that the matrices in \eqref{eq:det-lem} have the same
  eigenvalues, and thus \eqref{eq:pf3}.

   Since $K_{1,2}$ is rank-one, the kernel inside the determinant on the right hand side of
  \eqref{eq:pf3} is a rank-two perturbation of $\left[\begin{smallmatrix} I &
      K_{2,2}\\-K^a_{1,1} & I
    \end{smallmatrix}\right]$.
  Now recall that if $B$ is a rank-one kernel then
  \begin{equation}
    \det[I+A+B]=\det[I+A]\big(1+\tr[(I+A)^{-1}B]\big)\label{eq:rk-1}
  \end{equation}
  (this will be used later on), as long as $I+A$ is invertible. A similar formula holds
  for rank-two perturbations, but in our case the skew-symmetry of $K^a_{1,1}$ and
  $K_{2,2}$ yields the nicer formula
  \begin{multline}
    \label{eq:pdrk1}
      \pf[J-K]^2=\det\!\left[I+\left[
          \begin{smallmatrix}
            0 & K_{2,2} \\ -K^a_{1,1} & 0
          \end{smallmatrix}\right]-\left[
          \begin{smallmatrix}
            K_{1,2}^{\sf T} & 0 \\ 0 & K_{1,2}
          \end{smallmatrix}\right]\right]\\
        =\det\!\left[I+K^a_{1,1}K_{2,2}\right]
        \left(1-\tr\!\left[(I+K^a_{1,1}K_{2,2})^{-1}K_{1,2}\right]\right)^2.
  \end{multline}
  Of course, a necessary condition for this identity to hold is that $I+K^a_{1,1}K_{2,2}$
  be invertible, but this is true thanks to \eqref{eq:shMo}. The identity follows by
  approximation from the following matrix identity, which is (167) in \cite{cal-led}: if
  $A$ and $B$ are skew-symmetric $n\times n$ matrices, $R=uv^{\sf T}$ for some
  $u,v\in\rr^n$, $Q= \left[
          \begin{smallmatrix}
          R & 0 \\ 0 & R^T
          \end{smallmatrix}\right]$ and $P =\left[
          \begin{smallmatrix}
           I  & B \\ A & I
          \end{smallmatrix}\right]$, then
   \begin{equation}
    \det(P+Q) = \det(P)(1+ \langle u, (I-AB)^{-1}v\rangle )^2,\label{eq:cal-led-rk-2}
   \end{equation}
   as long as $I-AB$ is invertible. Since this matrix identity is not completely obvious,
   and no proof is given in \cite{cal-led}, let us pause for a moment to prove it.

   Note first that, since $I-AB$ is invertible by assumption (and then, taking transpose
   and by skew-symmetry of $A$ and $B$, so is $I-BA$), we have
   $P^{-1} = \left[
          \begin{smallmatrix}
           (I -BA)^{-1} & 0 \\ 0 &  (I -AB)^{-1}
          \end{smallmatrix}\right]\left[
          \begin{smallmatrix}
           I  & -B \\ -A & I
          \end{smallmatrix}\right]$.
    On the other hand, we have
   $\det(P+Q) =\det(P)\det(I+P^{-1}Q)$, so it suffices to show that $\det(I +
   P^{-1}Q)= (1+ \langle u, (I-AB)^{-1}v\rangle )^2$. Using the formula for $P^{-1}$ we get $(P^{-1}Q) \left(
          \begin{smallmatrix}
           w_1 \\ w_2
          \end{smallmatrix}\right) =  \left(\begin{smallmatrix}
           \langle v,w_1\rangle (I-BA)^{-1} u \\    \langle u,w_2\rangle (I-AB)^{-1} v
          \end{smallmatrix}\right)  - \left(\begin{smallmatrix}
           \langle u,w_2\rangle (I-BA)^{-1} Bv \\    \langle v,w_1\rangle (I-AB)^{-1} Au
          \end{smallmatrix}\right)$. 
  Now since $A$ and $B$ are skew-symmetric, we have (taking transpose) that $\langle u,
  (I-AB)^{-1}Au\rangle =  -\langle u, A(I-BA)^{-1}u\rangle$. But
  $(I-AB)^{-1}A=A(I-BA)^{-1}$, as can be checked easily by multiplying both sides by
  $I-AB$, so the previous identity implies that $\langle u, (I-AB)^{-1}Au \rangle = 0$. Now it is easy to see that $P^{-1}Q$
  is rank-two, with eigenfunctions $\left(
          \begin{smallmatrix}
           (I-BA)^{-1} u \\ (I-AB)^{-1} Au
          \end{smallmatrix}\right)$ and $\left(
          \begin{smallmatrix}
           (I-BA)^{-1}Bv \\ (I-AB)^{-1} v
          \end{smallmatrix}\right)$, both with eigenvalue $\langle u,
        (I-AB)^{-1}v\rangle$.  This implies that $\det(I+P^{-1}Q)=(1+ \langle u,
        (I-AB)^{-1}v\rangle )^2$ as desired, and \eqref{eq:cal-led-rk-2} follows.

Going back to \eqref{eq:pdrk1}, observe that
  \begin{equation}\label{eq:airyKernelPert}
    \begin{aligned}
    K^a_{1,1}K_{2,2}(\lambda_1,\lambda_2) &= -\tfrac12\int_{0}^\infty
    d\xi\,\p_\xi\K(\lambda_1+r,\xi+r)\sgn(\lambda_2-\xi)\\
    &=-\K(\lambda_1+r,\lambda_2+r)+\tfrac12\K(\lambda_1+r,r).
  \end{aligned}
  \end{equation}
  Writing
  \begin{equation}
    \label{eq:K2r}
    B_{r}(\lambda_1,\lambda_2)=\Ai(\lambda_1+\lambda_2+r)\qqand \bar\delta_0=\delta_0\otimes\uno{}
  \end{equation}
  (here $\uno{}$ is the function which is identically equal to 1), observing that
  $\K(\cdot+r,\cdot+r)=B_r^2$ and using the last formula we may write
  \begin{equation}
    K^a_{1,1}K_{2,2}=-B_r^2+\tfrac12B_r^2\bar\delta_0\qqand K_{1,2}=-\tfrac12B_{r}\bar\delta_0.
  \end{equation}
  The kernel $\tfrac12B_r^2\bar\delta_0$ is rank-one, and thus by \eqref{eq:rk-1} we have
  \begin{equation}
    \det[I+K^a_{1,1}K_{2,2}]=\det[I-B_r^2](1+\tfrac12\tr[(I-B_r^2)^{-1}B_r^2\bar\delta_0]).\label{eq:1stdet}
  \end{equation}
  On the other hand, using again the fact that $K^a_{1,1}K_{2,2}$ is a rank-one
  perturbation of $-B_r^2$, the Sherman-Morrison formula gives
  \begin{equation}
    (I+K^a_{1,1}K_{2,2})^{-1}=(I-B_r^2)^{-1}
    -\tfrac12\tfrac{(I-B_r^2)^{-1}B_r^2\bar\delta_0(I-B_r^2)^{-1}}{1+\frac12\tr[(I-B_r^2)^{-1}B_r^2\bar\delta_0]},\label{eq:shMo}
  \end{equation}
  and thus
  \[1+\tr\big[(I+K^a_{1,1}K_{2,2})^{-1}K_{1,2}\big] =1-\tfrac12\tr[(I-B_r^2)^{-1}B_r\bar\delta_0]
  +\tfrac14\tfrac{\tr[(I-B_r^2)^{-1}B_r^2\bar\delta_0(I-B_r^2)^{-1}B_r\bar\delta_0]}
  {1+\frac12\tr[(I-B_r^2)^{-1}B_r^2\bar\delta_0]}.\] Now writing $B_r\bar\delta_0$ as $\psi\otimes\uno{}$
  (where $\psi(\lambda)=\Ai(\lambda+r)$), $A=(I-B_r^2)^{-1}$ and $B=B_r\bar\delta_0$, one checks
  that
  \[\tr[AB(\psi\otimes\uno{})A(\psi\otimes\uno{})]=\tr[AB(\psi\otimes\uno{})]\tr[A(\psi\otimes\uno{})].\] Using this fact, the above identity
  yields
  \[1+\tr\big[(I+K^a_{1,1}K_{2,2})^{-1}K_{1,2}\big]
  =\tfrac{1-\frac12\tr[(I-B_r^2)^{-1}B_r\bar\delta_0]
    +\frac12\tr[(I-B_r^2)^{-1}B_r^2\bar\delta_0]}{1+\frac12\tr[(I-B_r^2)^{-1}B_r^2\bar\delta_0]}
  =\tfrac{1-\frac12\tr[(I+B_r)^{-1}B_r\bar\delta_0]}{1+\frac12\tr[(I-B_r^2)^{-1}B_r^2\bar\delta_0]}.\]
  Using this and \eqref{eq:1stdet} in \eqref{eq:pdrk1} we deduce that
  \[\pf[J-K]=\det[1-B_r^2]\frac{(1-\frac12\tr[(I+B_r)^{-1}B_r\bar\delta_0])^2}{1+\frac12\tr[(I-B_r^2)^{-1}B_r^2\bar\delta_0]}.\]
  Now let $\varphi_\ep(x)=\left(\frac{2}{\pi\ep}\right)^{1/2}e^{-x^2/(2\ep)}$ for $\ep>0$, so that
  \[\tr[(I+B_r)^{-1}B_r\bar\delta_0]=\lim_{\ep\to0}\tr[(I+B_r)^{-1}B_r(\varphi_\ep\otimes\uno{})].\]
  Since $\tr[\varphi_\ep\otimes\uno{}]=1$ we have
  \[1-\tfrac12\tr[(I+B_r)^{-1}B_r\bar\delta_0]=\tfrac12+\tfrac12\lim_{\ep\to0}\tr[(I+B_r)^{-1}(\varphi_\ep\otimes\uno{})]
  =\tfrac12+\tfrac12\tr[(I+B_r)^{-1}\bar\delta_0].\]
  In a similar way we get
  \[1+\tfrac12\tr[(I-B_r^2)^{-1}B_r^2\bar\delta_0]=\tfrac12+\tfrac12\tr[(I-B_r^2)^{-1}\bar\delta_0]
  =\tfrac12+\tfrac14\tr[(I-B_r)^{-1}\bar\delta_0]+\tfrac14\tr[(I+B_r)^{-1}\bar\delta_0].\]
  We deduce from the above identities that
  \[\pf[J- K]=\det[1-B_r^2]\frac{(1+\tr[(I+B_r)^{-1}\bar\delta_0])^2}
  {2+\tr[(I-B_r)^{-1}\bar\delta_0]+\tr[(I+B_r)^{-1}\bar\delta_0]}.\]  
  Now we use the identity
  \[\frac{\det[I-B_r]}{\det[I+B_r]}=\tr[(I+B_r)^{-1}\bar\delta_0],\]
  proved in \cite{ferrariSpohn}, and the fact, noted in \cite{cal-led}, that the same
  identity holds (with identical proof) if $B_r$ is replaced by $-B_r$. The consequence is
  that
  \begin{equation}
    \pf[J-K]^2=\det[1-B^2_{r}]\frac{\big(1+\frac{\det[I-B_r]}{\det[I+B_r]}\big)^2}
    {2+\frac{\det[I-B_r]}{\det[I+B_r]}+\frac{\det[I+B_r]}{\det[I-B_r]}}=\det[I-B_r]^2.\label{eq:pf-det}
  \end{equation}
  Since, by \cite{ferrariSpohn}, $\det[I-B_r]_{L^2([0,\infty))}=F_{\rm GOE}(r)$, this
  proves \eqref{eq:GOE-pfaffian} up to sign.
 
  In order to determine the sign the basic idea is to argue by continuity and compare the
  two sides in the limit $r\to\infty$ (in the remainder of the proof we will reintroduce
  the subscript in $K_r$). In order use continuity we will take advantage of the fact
  that, although $K_r$ is not trace class, it is easy to turn its Fredholm Pfaffian into
  that of a Hilbert-Schmidt operator (for the definition see \cite{quastelRem-review}). In
  fact, defining a multiplication operator
  $M\binom{f_1}{f_2}(x)=\binom{\phi(x)f_1(x))}{\phi(x)^{-1}f_2(x)}$ with
  $\phi(x)=(1+x^2)$, it is not hard to check, using the fact that $|\!\Ai(x)|\leq
  c\hspace{0.01em}e^{-2x^{3/2}/3}$ for $x\geq0$, that $M^{\sf T}K_rM$ is Hilbert-Schmidt
  (see Example 2 in Section 2 of \cite{quastelRem-review} for the proof of a similar
  fact). Thus we may use the notion of regularized Pfaffians introduced in Section
  \ref{sec:fredPf} (see \eqref{eq:regPf}). The idea is to study
  \begin{equation}
    \begin{aligned}
      \pf_2[J-M^{\sf T}K_rM]&=\pf\!\big[e^{-\inv2JM^{\sf T}K_rM}(J+JM^{\sf
        T}K_rMJ)e^{-\inv2M^{\sf T}K_rMJ}]\big]\\
      &=e^{\tr[\phi^{-1}(K_r)_{1,2}\phi]}\pf[J+JM^{\sf
        T}K_rMJ].
    \end{aligned}\label{eq:pf2GOE}
  \end{equation}
  The first equality is by definition of the regularized determinant. For the second one,
  where $\phi$ and $\phi^{-1}$ denote the corresponding multiplication operators, we are
  using the fact that the operator appearing in the exponent is trace class (see
  Example 3 in Section 2 of \cite{quastelRem-review} for the proof of a similar fact) to
  show that $\det[e^{-\inv2M^{\sf T}K_rMJ}]$ equals the exponential; the identity follows
  now from \eqref{eq:pfDetConj} (the fact that $\pf[J+JM^{\sf T}K_rMJ]$ defines a
  convergent series follows in the same way as for $\pf[J-K_r]$).

  A similar argument to the one that shows that $M^{\sf T}K_rM$ is Hilbert-Schmidt shows
  that $M^{\sf T}K_rM$ is continuous in $r$ in Hilbert-Schmidt norm, which implies by
  \eqref{eq:pfRegCont} that the left hand side of \eqref{eq:pf2GOE} is continuous in
  $r$. Moreover, similar arguments show that $\phi^{-1}(K_r)_{1,2}\phi$ is continuous in
  $r$ in trace class norm, and thus the exponential in \eqref{eq:pf2GOE} is continuous in
  $r$. We deduce that $\pf[J+JM^{\sf T}K_rMJ]=\pf[J-K_r]$ is continuous in $r$ and thus,
  since $F_{\rm GOE}(r)>0$ for all $r$, this and \eqref{eq:pf-det} imply that there is a
  $\sigma\in\{-1,1\}$ such that $\pf[J-K_r]=\sigma F_{\rm GOE}(r)$ for all
  $r$. Using again \eqref{eq:pfRegCont} and the fact that $\lim_{r\to\infty}M^{\sf
    T}K_rM=M^{\sf T}K_\infty M$ in Hilbert-Schmidt norm with
  $K_\infty(\lambda_1,\lambda_2)=\left[\begin{smallmatrix} 0 & 0 \\ 0 &
      \inv2\sgn(\lambda_2-\lambda_1)\end{smallmatrix}\right]$, we get
  $\lim_{r\to\infty}\pf[J-M^{\sf T}K_rM]=\pf[J-M^{\sf T}K_\infty M]=1$. Since
  $\lim_{r\to\infty}F_{\rm GOE}(r)$ is also 1, this implies that $\sigma=1$ and
  finishes the proof of \eqref{eq:pf2GOE}.
\end{proof}

\section{Fredholm Pfaffians}\label{sec:linalg}

The purpose of this section is to provide a brief discussion about the theory of Fredholm
Pfaffians. These were introduced by Rains in \cite{rainsCorr}, who stated its main
properties. Some of the issues discussed in Section \ref{sec:fredPf} are not
directly motivated by what is used in the main text, but we chose to include them in the
hope that they will help clarifying the notion of a Fredholm Pfaffian and its relation to
the Fredholm determinant.

\subsection{Pfaffians}
\label{sec:pfaffians}

Let $A$ be a $2n\times 2n$ skew-symmetric matrix. By basic spectral theory the eigenvalues
of $A$ come in pairs $\pm\I\lambda$ (for real $\lambda$), which then implies that $\det(A)$ can be written as
the square of a polynomial in the entries of $A$. The \emph{Pfaffian} of $A$ is defined
(up to sign) to be this polynomial, and it has the following explicit form:
\begin{equation}\label{eq:DefPf}
	\pf\!\left(A\right)=\frac1{2^n n!} \sum_{\sigma\in S_{2n}}\sgn(\sigma) \prod_{j=1}^n A_{\sigma(2j-1),\sigma(2j)}.
\end{equation}
When $A$ has odd size the same arguments imply that $A$ has $0$ as an eigenvalue, which
leads to define $\pf(A)=0$. The fact that the definition of the Pfaffian through \eqref{eq:DefPf} satisfies
\begin{equation}
  \label{eq:detPf}
  \pf(A)^2=\det(A)
\end{equation}
is not at all obvious; we refer the reader to Section 3 of \cite{deiftGioev} which 
contains three different proofs.

Another basic fact about Pfaffians (a proof of which can also be found in
\cite{deiftGioev}) is the following: Given two matrices $A$, $B$ of size $k\times k$, with
$k$ even and $A$ skew-symmetric,
\begin{equation}
  \label{eq:detPf}
  \pf(BAB^{\sf T})=\det(B)\pf(A).
\end{equation}
A particular case of this identity, which we use repeatedly in Section \ref{sec:pfaffian}
and below, is the following. Let $A$, $B$, $U$ be given matrices of size $k\times k$, with $A$ and $B$ skew-symmetric, and
$D_1$, $D_2$ be diagonal matrices of the same size. Then
  \begin{equation}
    \prod_{a=1}^k(D_1)_{a,a}(D_2)_{a,a}
    \pf\!\left[\left(\begin{matrix}A&U\\ -U&B\end{matrix}\right)\right]
    =\pf\!\left[\left(\begin{matrix} D_1&0\\ 0&D_2\end{matrix}\right)
      \left(\begin{matrix}A&U\\ -U&B\end{matrix}\right)
      \left(\begin{matrix} D_1&0\\ 0&D_2\end{matrix}\right)\right].\label{eq:diag2Pf}
  \end{equation}
  
We collect here some further facts about Pfaffians of skew-symmetric matrices which were
used in the main text. The first one was used to extend certain formulas from an even to
an odd number of variables in Section \ref{sec:genku}, while the next two involve
certain integration identities which were useful in the derivation of the Pfaffian
formulas in Sections \ref{sec:pfaffian} and \ref{sec:generatingfunction}.

\begin{lem}\label{lem:pfaffianblockform}
  Fix $k\in\zz_{\geq0}$ and consider two skew-symmetric $k\times k$ matrices
  $A$ and $B$, two (column) vectors $U$, $V$ of size $k$, and define the $k\times k$
  matrix $D=UV^{\sf T}$. Then
  \begin{equation}\label{eq:pfaffianblockform}
    \pf\!\left[\begin{matrix} A & D \\ -D^{\rm
          T} & B \end{matrix}\right] =
    \begin{dcases*}
      \qquad\pf[A]\pf[B] & if $k$ is even\\
      \pf\!\left[\begin{matrix} A & U \\ -U^{\sf T} &
          0 \end{matrix}\right]\pf\!\left[\begin{matrix} B & V \\ -V^{\sf T} &
          0 \end{matrix}\right] & if $k$ is odd.
    \end{dcases*}\end{equation}
\end{lem}

\begin{proof}
  Let us first consider the case $k$ even. More generally, we will consider the case where
  $A$ is $k_1\times k_1$, $B$ is $k_2\times k_2$ and $D$ is $k_1\times k_2$, with $k_1$
  and $k_2$ even. If $A$ or $B$ are singular then clearly both sides vanish, so we will
  assume that $A$ and $B$ are invertible. We have
  \[\pf\!\left[\begin{smallmatrix} A & D \\ -D^{\rm
          T} & B \end{smallmatrix}\right]^2=\det\!\left[\begin{smallmatrix} A & D \\ -D^{\rm
          T} & B \end{smallmatrix}\right]=\det[A]\det[B+D^{\sf T}A^{-1}D].\]
  Now $D^{\sf T}A^{-1}D=V(U^{\sf T}A^{-1}U)V^{\sf T}$, and the middle factor is a scalar
  which, by the skew-symmetry of $A^{-1}$, has to be zero. This shows that
  \begin{equation}\label{eq:prepfaffianblockform}
    \pf\!\left[\begin{smallmatrix} A & D \\ -D^{\rm
          T} & B \end{smallmatrix}\right]^2 = \pf[A]^2\pf[B]^2,
  \end{equation}
  which implies the desired identity up to sign. To determine the sign, replace $D$ by
  $\ep D$ for $\ep\geq0$. Since $\pf\!\left[\begin{smallmatrix} A & \ep D \\ -\ep D^{\rm
          T} & B \end{smallmatrix}\right]$ is a continuous function of $\ep$ which does
    not vanish for $\ep\geq0$ (recall that we are assuming that $A$ are $B$ are not singular)
  we see by \eqref{eq:prepfaffianblockform} that its value is either always $\pf[A]\pf[B]$
  or always $-\pf[A]\pf[B]$. Now taking $\ep=0$ the Pfaffian becomes
  $\pf\!\left[\begin{smallmatrix} A & 0 \\0 & B \end{smallmatrix}\right]=\pf[A]\pf[B]$, which shows that
    \eqref{eq:prepfaffianblockform} holds without the squares as desired.

  Now for the case $k$ odd, consider the matrix
  \[E=\left[\begin{smallmatrix} A & D & U & U \\
      -D^{\sf T} & B & -V & -V \\
      -U^{\sf T} & V^{\sf T} & 0 & 1 \\
      -U^{\sf T} & V^{\sf T} & -1 & 0
    \end{smallmatrix}\right].\]
  By the previous case, $\pf[E]=\pf\!\left[\begin{smallmatrix} A &  D \\ -D^{\rm
          T} & B \end{smallmatrix}\right]\pf\!\left[\begin{smallmatrix} 0 & 1 \\ -1 &
        0\end{smallmatrix}\right]
    =\pf\!\left[\begin{smallmatrix} A & D \\ -D^{\rm
          T} & B \end{smallmatrix}\right]$ so we may work with $E$. Now permuting the
    second and third rows and columns of $E$ does not change the value of the Pfaffian, and
    leaves us with
      \[E'=\left[\begin{smallmatrix} A &  U & D & U\\
      -U^{\sf T} & 0 &  V^{\sf T} & 1 \\
      -D^{\sf T} & -V & B & -V \\
      -U^{\sf T} & -1 & V^{\sf T} & 0 &
    \end{smallmatrix}\right].\]
  Noting that $\left[\begin{smallmatrix} D & U \\ -V^{\sf T} & 1 \end{smallmatrix}\right]$
  is rank one we see that we are in the first case ($k$ even) and the result follows.
\end{proof}

\begin{lem}\label{lem:resumpf}
  Consider three skew-symmetric kernels $A$, $B$, $C$ defined on some measurable space
  $(X,\mu)$. Then, assuming that all integrals converge, we have
  \begin{multline*}
    \inv{n!}\int_{X^{2n}}\mu^{\otimes2n}(d\vec x)\,\pf\big[(A+B)(x_a,x_b)\big]_{a,b=1}^{2n}
    \pf\big[C(x_a,x_b)\big]_{a,b=1}^{2n}\\
    =\sum_{k_1,k_2\geq0,\,k_1+k_2=n}\frac{1}{(2k_1)!(2k_2)!}
    \int_{X^{2n}}\mu^{\otimes2n}(d\vec x)\,\pf\big[A(x_a,x_b)\big]_{a,b=1}^{2k_1}\\
    \times\pf\big[B(x_a,x_b)\big]_{a,b=2k_1+1}^{2k_1+2k_2}\pf\big[C(x_a,x_b)\big]_{a,b=1}^{2n}.
  \end{multline*}
\end{lem}

\begin{proof}
  Using several times the definition of the Pfaffian, the fact that the Pfaffian is
  skew-symmetric, and the symmetry of the integration measure, the left hand side of the
  identity equals
  \begin{align*}
    &\inv{2^n(n!)^2}\int_{X^{2n}}\mu^{\otimes2n}(d\vec x)\,\sum_{\sigma\in S_{2n}}
    \prod_{a=1}^n(A+B)(x_{\sigma(2a-1)},x_{\sigma(2a)})
    \pf\big[C(x_{\sigma(a)},x_{\sigma(b)})\big]_{a,b=1}^{2n}\\
    &~=\inv{2^nn!}\int_{X^{2n}}\mu^{\otimes2n}(d\vec x)\,
    \prod_{a=1}^n(A+B)(x_{2a-1},x_{2a})
    \pf\big[C(x_{a},x_{b})\big]_{a,b=1}^{2n}\\
    &~=\inv{2^nn!}\int_{X^{2n}}\mu^{\otimes2n}(d\vec x)\,\sum_{I\subseteq\{1,\dotsc,n\}}
    \prod_{a\in I}A(x_{2a-1},x_{2a})\prod_{a\notin I}B(x_{2a-1},x_{2a})
    \pf\big[C(x_{a},x_{b})\big]_{a,b=1}^{2n}\\
    &~=\inv{2^nn!}\sum_{m=0}^n\binom{n}{m}\int_{X^{2n}}\mu^{\otimes2n}(d\vec x)\,
    \prod_{a=1}^mA(x_{2a-1},x_{2a})\prod_{a=m+1}^nB(x_{2a-1},x_{2a})\pf\big[C(x_a,x_b)\big]_{a,b=1}^{2n}\\
    \shortintertext{where we have used the fact that, for $|I|=m$, the permutation which
      maps the ordered $m$-tuple formed by the elements in the set
      $\bigcup_{a\in I}\{2a-1\}\cup\{2a\}$ into the $m$-tuple $(1,\dotsc,2m)$ is even,
      which implies that the Pfaffian of the kernel $C$ does not change after reordering
      the variables. Using again the symmetry of the integration measure and the
      antisymmetry of the Pfaffian we get that the above}
    &~=\sum_{m=0}^n\frac{1}{2^nm!(n-m)!}\int_{X^{2n}}\mu^{\otimes2n}(d\vec x)\,
    \inv{(2m)!}\sum_{\sigma_1\in S_{2m}}\prod_{a=1}^mA(x_{\sigma_1(2a-1)},x_{\sigma_1(2a)})\\
    &\hspace{0.2in}\times\inv{(2n-2m)!}\sum_{\sigma_2\in S_{2(n-m)}}\prod_{a=m+1}^nB(x_{\sigma_2(2a-1)},x_{\sigma_2(2a)})
    \sgn(\sigma_1)\sgn(\sigma_2)\pf\big[C(x_a,x_b)\big]_{a,b=1}^{2n}\\
    &~=\sum_{m=0}^n\frac{1}{(2m)!(2n-2m)!}\int_{X^{2n}}\mu^{\otimes2n}(d\vec
    x)\,\pf\big[A(x_a,x_b)\big]_{a,b=1}^{2m}\\
    &\hspace{3in}\times\pf\big[B(x_a,x_b)\big]_{a,b=2m+1}^{2n}
    \pf\big[C(x_a,x_b)\big]_{a,b=1}^{2n},
  \end{align*}
  which gives the desired result.
\end{proof}

Our second integration formula can be regarded as a certain Pfaffian version of the
Andr\'eief identity \cite{andreief} (sometimes referred to as the generalized Cauchy-Binet identity):

\begin{lem}\label{lem:pfaffInt}
  Let $(X,\mu)$ be some measurable space and suppose that for every $\vec x\in X^k$
  the matrix $\big[A_{a,b}(x_a,x_b)\big]_{a,b=1}^k$ is skew-symmetric. Let $B$ be
  another $k\times k$ skew-symmetric matrix and for $\vec x\in X^k$ consider another matrix
  $\big[U_{a,b}(x_b)\big]_{a,b=1}^k$. Finally consider functions $\phi_a$ defined on $X$. Then
  \begin{multline}
    \int_{X^k}\mu^{\otimes k}(d\vec x)\,\prod_{a=1}^k\phi_a(x_a)
    \pf\!\left[\begin{matrix}\big[A_{a,b}(x_a,x_b)\big]_{a,b=1}^k & \big[U_{a,b}(x_a)\big]_{a,b=1}^k \\
        \big[-U_{b,a}(x_b)\big]_{a,b=1}^k & \big[B_{a,b}\big]_{a,b=1}^k\end{matrix}\right]\\
    =\pf\!\left[\begin{matrix}\big[\int_{X^2}\mu(dx)\mu(dx')\phi_a(x)\phi_b(x')A_{a,b}(x,x')\big]_{a,b=1}^k
        &
        \big[ \int_X\mu(dx)\,\phi_a(x)U_{a,b}(x)\big]_{a,b=1}^k \\
        -\big[\int_X\mu(dx)\,\phi_b(x)U_{b,a}(x)\big]_{a,b=1}^k &
        \big[B_{a,b}\big]_{a,b=1}^k\end{matrix}\right]
  \end{multline}
  provided that all integrals converge.
\end{lem}

\begin{proof}
  Use \eqref{eq:diag2Pf} on the left hand side of the claimed identity and expand the
  Pfaffian using its definition to see that it equals
  \begin{align*}
    &\inv{2^kk!}\sum_{\sigma\in S_{2k}}\sgn(\sigma)\int_{X^k}\mu(dx_1)\dotsm\mu(dx_k)\\
    &\qquad\times\qquad\smashoperator{\prod_{a:~\substack{\sigma(2a-1)\leq
          k\\\sigma(2a)\leq k}}}\qquad
    A_{\sigma(2a-1),\sigma(2a)}(x_{\sigma(2a-1)},x_{\sigma(2a)})
    \phi_{\sigma(2a-1)}(x_{\sigma(2a-1)})\phi_{\sigma(2a)}(x_{\sigma(2a)})\\
    &\qquad\times\qquad\smashoperator{\prod_{a:~\substack{\sigma(2a-1)\leq
          k\\\sigma(2a)>k}}}\qquad U_{\sigma(2a-1),\sigma(2a)-k}(x_{\sigma(2a-1)})
    \phi_{\sigma(2a-1)}(x_{\sigma(2a-1)})\\
    &\qquad\times\qquad\smashoperator{\prod_{a:~\substack{\sigma(2a-1)> k\\\sigma(2a)\leq
          k}}}\qquad -U_{\sigma(2a)-k,\sigma(2a-1)}(x_{\sigma(2a)})
    \phi_{\sigma(2a)}(x_{\sigma(2a)})
    \quad\smashoperator{\prod_{a:~\substack{\sigma(2a-1)> k\\\sigma(2a)> k}}}\qquad
    B_{\sigma(2a-1)-k,\sigma(2a)-k}.
  \end{align*}
  Since each $x_a$ ($a=1,\dotsc,k$) appears in one and only one of the above factors, the
  integrals can be brought inside the corresponding factor, and now forming the resulting
  Pfaffian gives the identity.
\end{proof}

\subsection{Fredholm determinants}\label{sec:fredDet}

It is instructive to review briefly the theory of Fredholm determinants on $L^2$
spaces before discussing Fredholm Pfaffians. For more details see Section 2 of \cite{quastelRem-review}.

Let $A$ by an $m\times m$ matrix, write $[m]=\{1,\dotsc,m\}$ and let $\binom{[m]}{\ell}$
denote the family of subsets of $[m]$ of size $\ell$. A standard calculation shows that, for $\lambda\in\cc$,
\begin{equation}
\det(I+\lambda A)=\sum_{\ell=0}^m\lambda^\ell\sum_{S\in\binom{[m]}{\ell}}\det(A_{S\times S})\label{eq:finFred}
\end{equation}
with self-explanatory notation. This suggests a way of extending the determinant to the
infinite dimensional case. In the case of $L^2$ spaces it leads to the following.

Let $(X,\Sigma,\mu)$ be a measure space and $K\colon X\times X\longrightarrow\rr$ be a
kernel which defines an integral operator acting on $L^2(X)$ through
$Kf(x)=\int_X\mu(dy)\,K(x,y)f(y)$. The \emph{Fredholm determinant} of $K$ (on $L^2(X)$) is defined to be
the (formal) power series
\begin{equation}\label{eq:fredDet}
  \det\big[I+\lambda K\big]_{L^2(X)} = \sum_{k=0}^\infty\frac{\lambda^k}{k!}\int_{X^k}
  \mu\left(dx_1\right)\dotsc\mu\left(dx_n\right)\det\!\Big[K(x_a,x_b)\Big]_{a,b=1}^k.
\end{equation}
We will omit the subscript $L^2(X)$ in the determinant when no confusion can arise.  This
identity can be regarded as a numerical identity whenever the right hand side is
absolutely convergent. This is the case, for instance, whenever $\abs{K(x,y)}\leq C$ for
all $x,y\in X$ and some $C>0$, thanks to Hadamard's bound.

An interesting class of operators for which the Fredholm determinant power series is
absolutely convergent is the family of trace class operators. We recall that an operator
$K\colon L^2(X)\longrightarrow L^2(X)$ is said to be \emph{trace class} if it has finite
\emph{trace norm}:
\[\|K\|_1:=\sum_{n\geq1}\int_X\mu(dx)\,\psi_n(x)|K|\psi_n(x)<\infty,\]
where $(\psi_n)_{n\geq0}$ is any orthonormal basis of $L^2(X)$ and
$|K|=\sqrt{K^*K}$ is the unique positive square root of the operator $K^*K$. Such an
operator is necessarily compact, and in this case one has
\begin{equation}
\det[I+zK]=\prod_{k\geq1}(1+z\lambda_k)\label{eq:eigenpr}
\end{equation}
where the $\lambda_k$'s are the eigenvalues of $K$. This identity  (known as Lidskii's
Theorem) is highly non-trivial, and provides one possible way to extend the definition of
the Fredholm determinant to trace class operators on a general
separable Hilbert space. A nice property of the Fredholm determinant restricted
to trace class operators is that it is continuous: for $K_1$ and $K_2$ trace class one has
\begin{equation}
\big|\!\det[I+K_1]-\det[I+K_2]\big|\leq\|K_1-K_2\|_1e^{\|K_1\|_1+\|K_2\|_1+1}.\label{eq:detCont}
\end{equation}
This inequality follows from a rather general (and simple) argument based on the
analiticity in $\lambda\in\cc$ of the function $\det[I+\lambda K]$ and the inequality
$|\!\det[I+\lambda K]|\leq e^{|\lambda|\|K\|_1}$ (which essentially follows from
\eqref{eq:eigenpr}), see Theorem 3.4 of \cite{simon} or Corollary II.4.2 of
\cite{gohbergGoldbergKrupnik} (see also the discussion following (2.9) in
\cite{quastelRem-review}).

Another useful fact about Fredholm determinants is the so-called cyclic property: if
$K_1\colon L^2(X_1)\longrightarrow L^2(X_2)$ and $K_1\colon L^2(X_2)\longrightarrow
L^2(X_1)$ then
\begin{equation}
\det[I+K_1K_2]=\det[I+K_2K_1]\label{eq:cyclic}
\end{equation}
whenever the two sides are absolutely convergent (e.g. if both $K_1K_2$ and $K_2K_1$ are
trace class). An important consequence of this fact is the ability to conjugate an
operator without changing its Fredholm determinant:
\begin{equation}
  \label{eq:conjugate}
  \det[I+V^{-1}KV]=\det[I+K]
\end{equation}
whenever both sides make sense. This is often used to replace $K$ by a conjugate trace
class kernel. For more on all this and on extensions to other separable Hilbert spaces see
\cite{simon,gohbergGoldbergKrupnik}.

Finally let us briefly recall the notion of regularized determinants, since a similar
notion will be introduced below in the Pfaffian case. Suppose that $K$ is a
Hilbert-Schmidt operator (see \cite{quastelRem-review} for the definition) and recall that
the product of two Hilbert-Schmidt operators is trace class. We define the
\emph{regularized determinant} of $K$ as
\begin{equation}
  \label{eq:regDet}
  \det\nolimits_2[I+K]=\det[(I+K)e^{-K}].
\end{equation}
Note that $(I+K)e^{-K}-I=\sum_{n\geq2}\frac{(-1)^{n+1}(n-1)}{n!}K^n$ is a trace class
operator, so the definition makes sense for any Hilbert-Schmidt operator $K$ (one can also
define regularized determinants of higher order, but we will not do it here). It can be
shown that the analog of \eqref{eq:eigenpr} for the regularized determinant is
\[\det\nolimits_2[I+K]=\prod_{k\geq1}(1+\lambda_k)e^{-\lambda_k},\]
which further implies that $\big|\hspace{-0.125em}\det\nolimits_2[I+K]\big|\leq
e^{\inv2\|K\|_2^2}$. Using this bound one can obtain the analog of \eqref{eq:detCont} (by
the same argument): if $K_1$ and $K_2$ are Hilbert-Schmidt
operators then
\begin{equation}
  \label{eq:regDetCont}
  \big|\!\det\nolimits_2[I+K_1]-\det\nolimits_2[I+K_2]\big|\leq\|K_1-K_2\|_2\hspace{0.1em}e^{\inv2(\|K_1\|_2+\|K_1\|_2+1)^2},
\end{equation}
which in particular gives continuity of the regularized determinant with respect to the
Hilbert-Schmidt norm. Another way in which this notion can be useful is the
following. Suppose one knew that both $\det[I+K]$ and $\det[e^{-K}]$ are given by
asolutely convergent series. Then $\det[(I+K)e^{-K}]=\det[I+K]\det[e^{-K}]$ (which follows
from the general property $\det[(I+K_1)(I+K_2)]=\det[I+K_1]\det[I+K_2]$). The left hand
side can be controled in terms of the Hilbert-Schmidt norm of $K$, so if one has some
additional control on $\det[e^{-K}]$ (observe, in particular, that if $K$ is nice enough,
e.g.  trace class, then $\det[e^{-K}]=e^{-\tr[K]}$) this can be used to control
$\det[I+K]$. For much more on this see \cite{simon,gohbergGoldbergKrupnik}.

\subsection{Fredholm Pfaffians on \texorpdfstring{$L^2$}{L2} spaces}
\label{sec:fredPf}

We turn now to Fredholm Pfaffians. Given $n\in\zz_{\geq1}$ we define $J$ to be the
$2n\times2n$ block-diagonal matrix which has $2\times2$ blocks on the diagonal, all equal
to $\left[\begin{smallmatrix}0 & 1 \\ -1 & 0\end{smallmatrix}\right]$. If $A$ is a
$2n\times2n$ skew-symmetric matrix then one can show that
\[\pf(J+\lambda A)=\sum_{\ell=0}^m\lambda^\ell\sum_{S\in\binom{[m]}{\ell}}\pf(A_{S\times
  S}).\]
This is, of course, the Pfaffian analogue of \eqref{eq:fredDet}, and suggests the following infinite dimensional 
extension \cite{rainsCorr} (note that the above sum actually involves only even $\ell$, since $\pf(A_{S\times
  S})=0$ otherwise)

Consider a skew-symmetric $2\times2$-matrix kernel
\[K(\lambda_1,\lambda_2)=\left[
  \begin{smallmatrix}
   K_{1,1}(\lambda_1,\lambda_2) & K_{1,2}(\lambda_1,\lambda_2)\\
    -K_{1,2}(\lambda_2,\lambda_1) & K_{2,2}(\lambda_1,\lambda_2)        
  \end{smallmatrix}\right]\]
(the skew-symmetry condition in this case translates into
$K_{a,a}(\lambda_1,\lambda_2)=-K_{a,a}(\lambda_2,\lambda_1)$ for $a=1,2$). We regard
$K$ as an integral operator acting on $f\in L^2(X)\oplus L^2(X)$ as follows:
\[(Kf)_b(x)=\sum_{a=1}^2\int_X\mu(dy)\,K_{b,a}(x,y)f_a(y)\qquad\text{for }b=1,2.\] For any
such kernel we define its \emph{Fredholm Pfaffian} on $L^2(X)$ as the (formal) power
series
\begin{equation}\label{eq:fredPfSeries}
  \pf\!\big[J+\lambda K\big]_{L^2(X)}=
  \sum_{k\geq0}\frac{\lambda^k}{k!}\int_{X^k}\mu(dx_1)\dotsm\mu(dx_k)\pf\!\big[K(x_a,x_b)\big]_{a,b=1}^k
\end{equation}
whenever the right hand side is convergent. As for Fredholm determinants, we will usually
omit the subscript $L^2(X)$ in the Pfaffian\footnote{Observe that, although $K$ acts on $L^2(X)\oplus L^2(X)$, we are declaring
  this to be the Fredholm Pfaffian on $L^2(X)$. This is just a matter of convention, so we
  make this choice for notational convenience.}.

Observe that for a $2n\times2n$ skew-symmetric matrix $A$ we have by \eqref{eq:detPf} that
\begin{equation}
\pf(J+A)^2=\pf(J(I+J^{-1}A))^2=\det(J)\det(I+J^{-1}A)=\det(I-JA),\label{eq:finiteFredPfDet}
\end{equation}
where we have used the facts that $\det(J)=1$ and $J^2=-I$.
The following result extends this to a relationship between Fredholm determinants on $L^2(X)\oplus
    L^2(X)$ and Fredholm Pfaffians on $L^2(X)$.

\begin{prop}\label{prop:fredDetPf}
  For any skew-symmetric $2\times2$-matrix kernel $K$ and any $\lambda\in\cc$, as long as the integrals in the series \eqref{eq:fredDet} and \eqref{eq:fredPfSeries} are convergent, we have
  \begin{equation}
  \pf\!\big[J+\lambda K\big]^2_{L^2(X)}=\det\!\big[I-\lambda JK\big]_{L^2(X)\oplus
    L^2(X)},\label{eq:fredDetPf}
  \end{equation}
  where the identity is in the sense of formal power series.
\end{prop}

\begin{proof}
  Let $\pf\!\big[J+\lambda K\big]=\sum \lambda^n b_n$ and $\det\!\big[I-\lambda JK\big] = \sum \lambda^n a_n$. If
  $K$ is finite-rank then \eqref{eq:finiteFredPfDet} implies that $a_n = \sum_{j=0}^n
  \binom{n}{j} b_j b_{n-j}$. In the general case the kernel $K$ can be approximated by discretization (as
  explained in Section 2 of \cite{quastelRem-review}) to see that the relation between the
  coefficients of the power series still holds, which proves the identity.
\end{proof}

Of course, \eqref{eq:fredDetPf} is to be regarded as a numerical identity whenever both
sides are absolutely convergent. As we will see next, this is the case when $K$ is trace
class.

\begin{prop}\label{prop:fredDetPfTrCl}
  If $K$ is a skew-symmetric $2\times2$-matrix kernel which defines a trace class operator
  on $L^2(X)\oplus L^2(X)$, then both sides of identity \eqref{eq:fredDetPf} define
  absolutely convergent series, and in particular the identity is numerical. Moreover, if
  the Fredholm Pfaffian series of $K$ is given by $\pf[J+\lambda K]=\sum_{n\geq0}a_n\lambda^n$, then
  $\sum_{n\geq0}|a_n\lambda^n|\leq e^{\inv2|\lambda|\|K\|_1}$.
\end{prop}

\begin{proof}
  The fact that the right hand side of the identity defines an absolutely convergent
  series follows from our discussion about Fredholm determinants and the fact that
  $|JK|=|K|$, so that $\|JK\|_1=\|K\|_1$. Now suppose that $A$ is a $2n\times2n$
  skew-symmetric matrix and let $z\in\cc$. Since $J+zA$ has pure imaginary eigenvalues in complex
  conjugate pairs and $\det(J+zA)=-\det(I-zJA)$, we know that the eigenvalues of $JA$ are
  real and they all have even multiplicity. This means that
  $\det(I-zJA)=\pf(J+zA)^2=\prod_a(1-z\lambda_j)^2$ where the $\lambda_j$'s are the
  eigenvalues of $JA$ (appearing with half their multiplicity).  This identity can be
  extended by approximation to $K$ since $K$ is trace class, and this determines the
  Fredholm Pfaffian series for $K$ up to sign:
  \[\pf[J+zK]=\sum_{n\geq0}a_nz^n=\sigma\prod_j(1-z\lambda_j)\]
  for $\sigma\in\{-1,1\}$. In particular, this means that the $n$-th term of the Fredholm
  Pfaffian series is given by $a_n=\sigma z^n \sum_{j_1<\cdots<j_n}
  \lambda_{j_1}\cdots\lambda_{j_n}$.  Now $\left|\sum_{j_1<\cdots
      j_n}\lambda_{j_1}\cdots\lambda_{j_n}\right| \leq \inv{n!}(\inv2\|JK\|)^n_1$ by Lemma
  3.3 and (3.8) in \cite{simon} (the $1/2$ is because each eigenvalue is counted with half
  its multiplicity). Since $\|JK\|_1= \|K\|_1$ as explained above, we have proved that
  \[\sum_{n\geq0}\big|a_nz^n\big|\leq\sum_{n\geq0}\frac{1}{n!}|z|^n\big(\tfrac12\|K\|_1)^n
  =e^{\inv2|z|\|K\|_1}.\qedhere\]
\end{proof}

An important consequence of the last result is that the Fredholm Pfaffian restricted to
trace class operators is continuous: if $K_1$ and $K_2$ are trace class skew-symmetric
$2\times2$-matrix kernels then \noeqref{eq:pfCont}
\begin{equation}
  \label{eq:pfCont}
  \big|\!\pf[J+K_1]-\pf[J+K_2]\big|\leq\|K_1-K_2\|_1e^{\inv2(\|K_1\|_1+\|K_2\|_1+1)}.
\end{equation}
Like \eqref{eq:regDetCont}, this inequality can be proved in the same way as
\eqref{eq:detCont}, based now on the inequality $\big|\hspace{-0.125em}\pf[J+K]\big|\leq
e^{\inv2\|K\|_1}$ implicit in Proposition \ref{prop:fredDetPfTrCl}.

We mention also two Fredholm Pfaffian analogs of \eqref{eq:conjugate}. The first one
follows from \eqref{eq:detPf} by approximation and says that
\begin{equation}
  \label{eq:pfDetConj}
  \pf[(I+L^{\sf T})(I+K)(I+L)]=\det[I+L]\pf[I+K]
\end{equation}
as long as both sides define absolutely convergent power series. For the second one, let $M$ be a symplectic
$2\times2$ matrix kernel. Then, as long as both sides define absolutely convergent power series,
we have
\begin{equation}
  \label{eq:pfConjugate}
  \pf[J+\lambda K]=\pf[J+\lambda M^{\sf T}KM].
\end{equation}
The proof of \eqref{eq:pfConjugate} is elementary:
\begin{multline}
  \pf[J+\lambda M^{\sf T}KM]=\sum_{k\geq0}\frac{\lambda^k}{k!}\int_{X^k}\mu(dx_1)\dotsm\mu(dx_k)\,
  \pf\!\big[(M^{\sf T}KM)(x_a,x_b)\big]_{a,b=1}^k\\
  =\sum_{k\geq0}\frac{\lambda^k}{k!}\int_{X^k}\mu(dx_1)\dotsm\mu(dx_k)\,
  \det\!\big[M(x_a,x_b)\big]_{a,b=1}^k\pf\!\big[K(x_a,x_b)\big]_{a,b=1}^k
  =\pf[J-K],
\end{multline}
where we have used \eqref{eq:detPf} and the fact that the determinant of any symplectic
matrix equals $1$.

One can in principle use \eqref{eq:pfConjugate} to replace $K$ by a trace class kernel and
then use the (better-developed, and in a sense simpler) theory of Fredholm determinants to
study the Fredholm Pfaffian of $K$. On the other hand, even if this is not possible one
can still use Proposition \ref{prop:fredDetPf} to reduce the study of a Fredholm Pfaffian
to a Fredholm determinant, as long as they are both convergent. This is in fact the
situation we are in in Section \ref{sec:GOEPfaffian}, where we are dealing with a kernel
$K$ for which there is no obvious symplectic $2\times2$-matrix kernel such that $M^{\sf
  T}KM$ is trace class, but the reduction to a Fredholm determinant is still very useful.

Let us finish by introducing a notion of regularized Fredholm Pfaffians. Suppose that $K$
is a skew-symmetric $2\times2$-matrix kernel which defines a Hilbert-Schmidt
operator. Observe that $J$ is symplectic, and thus by \eqref{eq:pfConjugate} we have
\[\pf[J+K]=\pf[J+J^{\sf T}KJ]=\pf[J-JKJ].\]
Observe also that $(JK)^{\sf T}=KJ$. In view of these two facts we define the
\emph{regularized Pfaffian} of $K$ as
\begin{equation}
  \label{eq:regPf}
  \pf_2[J+K]=\pf[(J-JKJ)e^{KJ}]=\pf[e^{\inv2JK}(J-JKJ)e^{\inv2KJ}].
\end{equation}
The second equality actually holds at the level of the operators inside the two Fredholm
Pfaffians (as can be checked easily using the series expansion of the exponential) and has
the advantage of making the skew-symmetry of the argument more apparent. A calculation
shows that $(J-JKJ)e^{KJ}-J=-\sum_{n\geq2}\frac{n-1}{n!}(JK)^nJ$, and thus the right hand
side of \eqref{eq:regPf} makes sense as the Fredholm Pfaffian of a trace class operator
(since $JK$ is Hilbert-Schmidt, and thus $(JK)^n$ is trace class for any $n\geq2$). Note
also that
\[\pf_2[J+K]^2=\det\nolimits_2[I-JK].\]
Additionally, if $K_1$ and $K_2$ are Hilbert-Schmidt then 
\begin{equation}
  \label{eq:pfRegCont}
  \left|\pf_2[J+K_1]-\pf_2[J+K_2]\right|\leq\|K_1-K_2\|_2\,e^{\inv4(\|K_1\|_2+\|K_1\|_2+1)^2}
\end{equation}
(with the same proof as \eqref{eq:detCont}, now based on the inequality
$|\hspace{-0.1em}\pf_2[J+K]|=\sqrt{|\hspace{-0.1em}\det\nolimits_2[I-JK]|}\leq
e^{\inv4\|K\|_2^2}$).

Observe that if one knew in addition that $K$ is trace class, then by
\eqref{eq:pfDetConj} one would have
\begin{equation}
\pf_2[J+K]=\det[e^{\inv2JK}]\pf[J-JKJ]=e^{\inv2\!\tr[JK]}\pf[J+K].\label{eq:pftr}
\end{equation}
As in the case of the regularized determinants discussed in Section \ref{sec:fredDet}, one
situation in which the introduction of regularized Pfaffians can be useful is when $K$ is
Hilbert-Schmidt but not trace class, but on the other hand some additional control on $K$
is available which makes it possible to show that the first or both equalities in
\eqref{eq:pftr} hold. In such a case, since the left hand side can be controlled by the
Hilbert-Schmidt norm of $K$ (which is the same as that of $JK$), this identity provides a
possible route for controlling $\pf[J+K]$. This idea was used at the end of the proof of
Proposition \ref{prop:GOEPfaffian} to use continuity in $r$ of a certain kernel $K_r$ with
respect to the Hilbert-Schmidt norm, in order to upgrade the identity
$\pf[J-K_r]^2=\det[I+JK_r]$ to $\pf[J-K_r]=\sqrt{\det[I+JK_r]}$.

\vspace{14pt}

\noindent{\bf Acknowledgements.}
The authors would like to thank Alexei Borodin, Ivan Corwin, Pierre Le Doussal, and Pasquale Calabrese for numerous discussions
about the results in this paper.  JO was partially supported by the Natural Sciences and Engineering Research Council
of Canada. JQ gratefully acknowledges financial support from the Natural Sciences and Engineering
Research Council of Canada, the I.~W. Killam Foundation, and the Institute for Advanced
Study. DR was partially supported by Fondecyt Grant 1120309, by Conicyt Basal-CMM, and by
Programa Iniciativa Cient\'ifica Milenio grant number NC130062 through Nucleus Millenium
Stochastic Models of Complex and Disordered Systems. 

\printbibliography[heading=apa]

\end{document}